\documentclass[11pt,reqno]{amsart}
\usepackage{amsmath, color}
\usepackage{amssymb}
\usepackage{a4wide}
\usepackage{graphics}
\usepackage{epsfig}

\newtheorem{theorem}{Theorem}[section]
\newtheorem{lemma}[theorem]{Lemma}
\newtheorem{proposition}[theorem]{Proposition}

\newtheorem{remark}[theorem]{Remark}
\newtheorem{corollary}[theorem]{Corollary}

\newcommand{\R}{\mathbb{R}}
\newcommand{\E}{\mathrm{E}}
\renewcommand{\P}{\mathrm{P}}

\makeatletter
 \@addtoreset{equation}{section}
 \makeatother
 
\allowdisplaybreaks
\usepackage{appendix}
\usepackage{hyperref}
\title[Local asymptotic properties for CIR]{Local asymptotic properties for Cox-Ingersoll-Ross process with discrete observations}
\date{\today}
\author[Mohamed Ben Alaya, Ahmed Kebaier and Ngoc Khue Tran]{Mohamed Ben Alaya, Ahmed Kebaier and Ngoc Khue Tran}

\address{Mohamed Ben Alaya, Laboratoire De Math\'ematiques Rapha{\"{e}}l Salem, UMR 6085, Universit\'e De Rouen, Avenue de L'Universit\'e Technop\^ole du Madrillet, 76801 Saint-Etienne-Du-Rouvray, France}
\email{mohamed.ben-alaya@univ-rouen.fr}

\address{Ahmed Kebaier, Universit\'e  Sorbonne Paris Nord, LAGA, CNRS (UMR 7539), F-93430 Villetaneuse, France}
\email{kebaier@math.univ-paris13.fr}

\address{Ngoc Khue Tran, Department of Natural Science Education - Pham Van Dong University, 509 Phan Dinh Phung, Quang Ngai City, Quang Ngai, Vietnam}
\email{tnkhueprob@gmail.com}

\thanks{This research is supported by Laboratory of Excellence MME-DII, Grant no. ANR11LBX-0023-01 (\url{http://labex-mme-dii.u-cergy.fr/}).  Ahmed Kebaier benefited from the support of the chair Risques Financiers, Fondation du Risque. Ngoc Khue Tran wishes to thank Universit\'e Sorbonne Paris Nord for the hospitality. This research is funded by Vietnam National Foundation for Science and Technology Development (NAFOSTED) under grant number 101.03-2017.316.}


\begin{document}
\maketitle
\begin{abstract} 
In this paper, we consider a one-dimensional Cox-Ingersoll-Ross (CIR) process whose drift coefficient depends on unknown parameters. Considering the process discretely observed at high frequency, we prove the local asymptotic normality property in the subcritical case, the local asymptotic quadraticity in the critical case, and the local asymptotic mixed normality property in the supercritical case. To obtain these results, we use the Malliavin calculus techniques developed recently for CIR process by Al\`os et {\it al.} \cite{AE08} and Altmayer et {\it al.} \cite{AN14} together with the $L^p$-norm estimation for positive and negative moments of the CIR process obtained by Bossy et {\it al.} \cite{BD07} and Ben Alaya et {\it al.} \cite{BK12,BK13}. In this study, we require the same conditions of high frequency $\Delta_n\rightarrow 0$ and infinite horizon $n\Delta_n\rightarrow\infty$ as in the case of ergodic diffusions with globally Lipschitz coefficients studied earlier by Gobet \cite{G02}. However, in the non-ergodic cases, additional assumptions on the decreasing rate of $\Delta_n$ are required due to the fact that the square root diffusion coefficient of the CIR process is not regular enough. Indeed, we assume $n\Delta_n^{3}\to 0$ for the critical case and $\Delta_n^{2}e^{-b_0n\Delta_n}\to 0$ for the supercritical case.
\end{abstract}

\section{Introduction}

On a complete probability space $(\widehat{\Omega}, \widehat{\mathcal{F}},  \widehat{\P})$, we consider a Cox-Ingersoll-Ross (CIR) process $X^{a,b}=(X_t^{a,b})_{t \geq 0}$ in $\R$ defined by 
\begin{equation}\label{c2eq1}
X_t^{a,b}=x_0+\int_0^t (a-bX_s^{a,b})ds +\int_0^t\sqrt{2\sigma X_s^{a,b}}dB_s,
\end{equation}
where $X_0^{a,b}=x_0>0$ is a given initial condition and we assume that $a\geq \sigma>0$. Here, $B=(B_t)_{t \geq 0}$ is a standard Brownian motion. The parameters $(a,b)\in \Theta\times \Sigma$ are unknown to be estimated, and $\Theta$ and $\Sigma$ are closed intervals of $\R^{\ast}_+$ and $\R$, respectively, where $\R^{\ast}_+=\R_+\setminus\{0\}$. Let $\{\widehat{\mathcal{F}}_t\}_{t\geq 0}$ denote the natural filtration generated by $B$. We denote by $\widehat{\P}^{a,b}$ the probability measure induced by the CIR process $X^{a,b}$ on the canonical space $(C(\R_{+},\R),\mathcal{B}(C(\R_{+},\R))$ endowed with the natural filtration $\{\widehat{\mathcal{F}}_t\}_{t\geq 0}$. Here $C(\R_{+},\R)$ denotes the set of $\R$-valued continuous functions defined on $\R_{+}$, and $\mathcal{B}(C(\R_{+},\R))$ is its Borel $\sigma$-algebra. We denote by $\widehat{\E}^{a,b}$ the expectation with respect to (w.r.t.) $\widehat{\P}^{a,b}$. Let $\overset{\widehat{\P}^{a,b}}{\longrightarrow}$, $\overset{\mathcal{L}(\widehat{\P}^{a,b})}{\longrightarrow}$, $\widehat{\P}^{a,b}$-a.s., and $\overset{\mathcal{L}}{=}$ denote the convergence in $\widehat{\P}^{a,b}$-probability, in $\widehat{\P}^{a,b}$-law, in $\widehat{\P}^{a,b}$-almost surely, and equality in law, respectively. $\ast$ denotes the transpose.

By applying It\^o's formula to \eqref{c2eq1}, the unique strong solution of the equation \eqref{c2eq1} which is non-negative is given by
\begin{equation}
X_t^{a,b}=x_0e^{-bt}+a\int_0^te^{-b(t-s)}ds+\int_0^te^{-b(t-s)}\sqrt{2\sigma X_s^{a,b}}dB_s,
\end{equation}
for all $t\geq 0$. Notice that condition $a\geq \sigma>0$ guarantees that the process $X^{a,b}$ is always positive, i.e., $\P(X_t^{a,b}>0,\ \forall t\geq 0)=1$. 

Recall that the CIR process is extensively used in mathematical finance to model the evolution of short-term interest rates or to describe the dynamic of the stochastic volatility in the Heston model. The CIR process appears in the financial literature also as part of the class of affine processes, and a lot of interesting material can be found in this way, e.g. in the work of Teichmann et {\it al.} \cite{CFMT11},
Duffie, Filipovi{\'c} and Schachermayer \cite{DFS03}, Kallsen \cite{K06},  Keller-Ressel and Mijatovi{\'c} \cite{KM12} and other authors.

A fundamental concept in asymptotic theory of statistics is the local asymptotic normality (LAN) property introduced by Le Cam \cite{LC60} and then extended by Jeganathan \cite{JP82} to the local asymptotic mixed normality (LAMN) property. The local asymptotic quadraticity (LAQ) property was introduced by e.g. \cite{H14, CY90}. Initiated by Gobet \cite{G01}, the techniques of Malliavin calculus can be used to analyze the log-likelihood ratio of the discrete observation of continuous diffusion processes. Concretely, Gobet \cite{G01, G02} obtained the LAMN and LAN properties respectively for multidimensional elliptic diffusions and ergodic diffusions on the basis of discrete observations at high frequency. In the presence of jumps, several cases have been largely investigated, see e.g. A\"it-Sahalia and Jacod \cite{AJ07}, Kawai \cite{K13}, Cl\'ement et {\it al.} \cite{CDG14, CG15}, Kohatsu-Higa et {\it al.} \cite{KNT14, KNT15}, and Tran \cite{T15}. However, all these results deal with the stochastic differential equations whose coefficients are continuously differentiable and satisfy a global Lipschitz condition. The case where the coefficient functions of the model do not satisfy these standard assumptions, for instance the square root diffusion function in the CIR model which is neither differentiable at $0$ nor globally Lipschitz, still remains an open problem when the model is discretely observed at high frequency.

On the other hand, notice that most existing research works on statistics for CIR process and more generally for affine diffusions mainly focus on parameter estimation based on continuous observations. More precisely, Overbeck \cite{O98} showed the asymptotic properties of maximum likelihood estimator (MLE) as well as the LAN, LAMN and LAQ properties for CIR process in the subcritical (ergodic), critical, and supercritical (non-ergodic) submodels. Later, Ben Alaya and Kebaier \cite{BK12, BK13} show various asymptotic properties of MLE associated to the partial and global drift parameters of the CIR process in both ergodic and non-ergodic cases. Recently, Barczy and Pap \cite{BP16} have studied the asymptotic properties of MLE for Heston models. Later on, Benke and Pap \cite{BP15} have proved the LAN, LAMN and LAQ properties for Heston models. More recently, Barczy et {\it al.} \cite{BBKP16,BBKP162} have studied the asymptotic properties of MLE for jump-type Heston models and jump-type CIR process. Some results on parameter estimation based on discrete observations of CIR process can be found e.g. in \cite{OR97,BK13}. In \cite{OR97}, Overbeck and Ryd\'en proved the LAN property for the ergodic CIR process based on discrete observations at low frequency and did one-step-improvement in the sense of Le Cam \cite{LC60}.

However, as mentioned just above, it seems that the validity of the LAN, LAMN and LAQ properties for CIR process in ergodic and non-ergodic cases on the basis of discrete observations in a high frequency context has never been addressed in the literature. One difficulty comes from the fact that its diffusion coefficient is the square root function. Motivated by this fact, the main objective of this paper is to derive the LAN property in the subcritical case, the LAQ property in the critical case, and the LAMN property in the supercritical case for $X^{a,b}$ based on discrete observations. 

Let us now recall the notion of the LAQ, LAMN and LAN properties in our setting. Given the process $X^{a,b}=(X_t^{a,b})_{t \geq 0}$ and $n\geq 1$, we consider a discrete observation at equidistant and deterministic times $t_k=k \Delta_n$, $k \in \{0,\ldots,n\}$ of the process $X^{a,b}$, which is denoted by $X^{n,a,b}=(X_{t_0}^{a,b}, X_{t_1}^{a,b},\ldots,X_{t_n}^{a,b})$, where $\Delta_n\leq 1$ for all $n\geq 1$. We suppose that the high-frequency and infinite horizon conditions satisfy: $\Delta_n\rightarrow 0$ and $n\Delta_n\rightarrow\infty$ as $n\rightarrow\infty$. We denote by $\P^{a,b}_n$ and $p_n(\cdot;(a,b))$ the probability law and the density of the random vector $X^{n,a,b}$, respectively. 

For fixed $(a_0,b_0)\in \Theta\times \Sigma$, we consider a discrete observation of the process $X^{a_0,b_0}$ given by $X^{n,a_0,b_0}=(X_{t_0}^{a_0,b_0}, X_{t_1}^{a_0,b_0},\ldots,X_{t_n}^{a_0,b_0})$. For $(u,v)\in\R^2$, we set $a_n:=a_0+u\varphi_{1,n}(a_0,b_0)$ and $b_n:=b_0+v\varphi_{2,n}(a_0,b_0)$, where we assume that $\varphi_{1,n}(a_0,b_0)$ and $\varphi_{2,n}(a_0,b_0)$ tend to zero as $n\to\infty$. Suppose that there exist a $\R^2$-valued random vector $U(a_0,b_0)$ and a (random) matrix $I(a_0,b_0)$ such that for all $z=(u,v)^{\ast}\in\R^2$, as $n\to\infty$,
\begin{align}\label{LAQ}
\log\dfrac{d\P_n^{a_n,b_n}}{d\P^{a_0,b_0}_n}(X^{n,a_0,b_0})\overset{\mathcal{L}(\widehat{\P}^{a_0,b_0})}{\longrightarrow} z^{\ast}U(a_0,b_0)-\dfrac{1}{2}z^{\ast} I(a_0,b_0)z.
\end{align}
Then, we say that
\begin{itemize}
	\item[\textnormal a)]  The LAQ property holds at $(a_0,b_0)$ with rates of convergence $(\varphi_{1,n}^{-1}(a_0,b_0),\varphi_{2,n}^{-1}(a_0,b_0))$ and random matrix $I(a_0,b_0)$ if further for all $z=(u,v)^{\ast}\in\R^2$, we have 
	$$
	\widehat{\E}^{a_0,b_0}\Big[e^{z^{\ast}U(a_0,b_0)-\frac{1}{2}z^{\ast} I(a_0,b_0)z}\Big]=1.
	$$
	\item[\textnormal b)] The LAMN property holds at $(a_0,b_0)$ with rates of convergence $(\varphi_{1,n}^{-1}(a_0,b_0),\varphi_{2,n}^{-1}(a_0,b_0))$ and asymptotic random Fisher information matrix $I(a_0,b_0)$ if further we can write $U(a_0,b_0)=I(a_0,b_0)^{\frac{1}{2}}\mathcal{N}\left(0,I_2\right)$, where $\mathcal{N}(0,I_2)$ denotes a centered $\R^2$-valued Gaussian random vector with identity covariance matrix $I_2$, which is independent of the conditional covariance matrix $I(a_0,b_0)$.\\	
	In particular, when $I(a_0,b_0)$ is deterministic, i.e., $U(a_0,b_0)=\mathcal{N}\left(0,I(a_0,b_0)\right)$, we say that the LAN property holds at $(a_0,b_0)$ with rates of convergence $(\varphi_{1,n}^{-1}(a_0,b_0),\varphi_{2,n}^{-1}(a_0,b_0))$ and asymptotic Fisher information matrix $I(a_0,b_0)$.
\end{itemize}
For the full power of the LAN, LAMN or LAQ property, one should consider arbitrary bounded sequences $(u_n,v_n)_{n\geq 1}$ in $\R^2$ when the local scale at $(a_0,b_0)$ is denoted by $(\varphi_n(a_0,b_0))_{n\geq 1}=(\varphi_{1,n}(a_0,b_0),\varphi_{2,n}(a_0,b_0))_{n\geq 1}$, that is, $a_n:=a_0+u_n\varphi_{1,n}(a_0,b_0)$ and $b_n:=b_0+v_n\varphi_{2,n}(a_0,b_0)$. For details, we refer the reader to Subsection $7.1$ of H{\"{o}}pfner \cite{H14} or Le Cam and Lo Yang \cite{CY90}.

As we will see, the rates of convergence for LAQ, LAMN and LAN properties for CIR process depend strongly on the drift parameter $b$. Indeed, the value of the drift parameter $b>0$, $b=0$ and $b<0$ determines respectively the subcritical, critical and supercritical cases.

To show the convergence \eqref{LAQ}, one needs to derive an appropriate stochastic expansion of the log-likelihood ratio. For any $t>s$, the law of $X_t^{a,b}$ conditioned on $X_s^{a,b}=x$ admits a positive explicit transition density $p^{a,b}(t-s,x,y)$. From the explicit expression \eqref{densityexplicit} and \eqref{densityexplicit0} below, the transition density $p^{a,b}(t-s,x,y)$ is differentiable w.r.t. $a$ and $b$. Then using Markov property and the mean value theorem, the log-likelihood ratio can be decomposed as
\begin{equation}\label{maindecompo}\begin{split}
&\log\dfrac{d\P_n^{a_n,b_n}}{d\P^{a_0,b_0}_n}(X^{n,a_0,b_0})=\log\dfrac{p_n\left(X^{n,a_0,b_0};\left(a_n,b_n\right)\right)}{p_n\left(X^{n,a_0,b_0};\left(a_0,b_0\right)\right)}\\
&=\log\dfrac{p_n\left(X^{n,a_0,b_0};\left(a_n,b_0\right)\right)}{p_n\left(X^{n,a_0,b_0};\left(a_0,b_0\right)\right)}+\log\dfrac{p_n\left(X^{n,a_0,b_0};\left(a_n,b_n\right)\right)}{p_n\left(X^{n,a_0,b_0};\left(a_n,b_0\right)\right)}\\
&=\sum_{k=0}^{n-1}\log\dfrac{p^{a_n,b_0}}{p^{a_0,b_0}}(\Delta_n,X_{t_k}^{a_0,b_0},X_{t_{k+1}}^{a_0,b_0})+\sum_{k=0}^{n-1}\log\dfrac{p^{a_n,b_n}}{p^{a_n,b_0}}(\Delta_n,X_{t_k}^{a_0,b_0},X_{t_{k+1}}^{a_0,b_0})\\
&=\sum_{k=0}^{n-1}u\varphi_{1,n}(a_0,b_0)\int_0^1\dfrac{\partial_{a}p^{a(\ell),b_0}}{p^{a(\ell),b_0}}(\Delta_n,X_{t_k}^{a_0,b_0},X_{t_{k+1}}^{a_0,b_0})d\ell\\
&\qquad+\sum_{k=0}^{n-1}v\varphi_{2,n}(a_0,b_0)\int_0^1\dfrac{\partial_{b}p^{a_n,b(\ell)}}{p^{a_n,b(\ell)}}(\Delta_n,X_{t_k}^{a_0,b_0},X_{t_{k+1}}^{a_0,b_0})d\ell,
\end{split}
\end{equation}
where $a_n=a_0+u\varphi_{1,n}(a_0,b_0)$, $b_n=b_0+v\varphi_{2,n}(a_0,b_0)$, $a(\ell):=a_0+\ell u\varphi_{1,n}(a_0,b_0)$ and $b(\ell):=b_0+\ell v\varphi_{2,n}(a_0,b_0)$.

Since we are dealing with the CIR process, one way to proceed could be to use an explicit expression for the transition density function 
which is characterized in terms of a non-central chi-squared distribution (see e.g. \cite{CIR85}). That is, for any $t>0$ and $b\neq 0$,
\begin{align}\label{densityexplicit}
p^{a,b}(t,x,y)=\dfrac{-be^{bt}}{\sigma(1-e^{bt})}\left(\dfrac{y}{xe^{-bt}}\right)^{\frac{\nu}{2}}\exp\Big\{\dfrac{b}{\sigma}\dfrac{x+e^{bt}y}{1-e^{bt}}\Big\}I_{\nu}\Big(\dfrac{-2b\sqrt{xy}e^{\frac{b}{2}t}}{\sigma(1-e^{bt})}\Big),
\end{align}
and $b=0$,
\begin{align}\label{densityexplicit0}
p^{a,0}(t,x,y)=\dfrac{1}{\sigma t}\left(\dfrac{y}{x}\right)^{\frac{\nu}{2}}\exp\Big\{-\dfrac{x+y}{\sigma t}\Big\}I_{\nu}\Big(\dfrac{2\sqrt{xy}}{\sigma t}\Big),
\end{align}
where $\nu=\dfrac{a}{\sigma}-1$, and $I_{\nu}$ is the modified Bessel function of the first 
kind of order $\nu>0$ defined by
\begin{align*}
I_{\nu}(x)=\sum_{n=0}^{\infty}\dfrac{1}{n!\Gamma(\nu+n+1)}\left(\dfrac{x}{2}\right)^{2n+\nu},
\end{align*}
for any $x\in\R$. Here $\Gamma(\cdot)$ is the Gamma function defined by $
\Gamma(z)=\int_{0}^{\infty}x^{z-1}e^{-x}dx$ for $z\in\R_{+}$. To use this approach, one needs to well understand the behavior of the logarithm derivatives of the transition density w.r.t. the parameters $a$ and $b$. However, it is not clear if using \eqref{densityexplicit} and \eqref{densityexplicit0} would help  especially since the parameter $a$ appears in the index $\nu$ of the modified Bessel function and we have at least to find  appropriate estimates of several types of increments involving the parameter $a$.

In this paper, for our purpose, we use an alternative strategy based on the Malliavin calculus approach
initiated by Gobet \cite{G01, G02} in order to derive an explicit expression for the logarithm derivatives of the transition density in terms of a conditional expectation of a Skorohod integral. Let us mention here that Malliavin calculus for CIR process is established by Al\`os and Ewald \cite{AE08} and then Malliavin calculus for constant elasticity of variance (CEV) process is developed by Altmayer and Neuenkirch \cite{AN14}. Furthermore, these articles give an explicit expression for the Malliavin derivative of the CIR process which will be useful for our purpose. Later on, the Malliavin differentiablity of Heston stochastic volatility have recently used e.g. in \cite{AN14, DJ06, KY09, P08}.

The main results of this paper are the LAN property in the subcritical case, the LAQ property in the critical case, and the LAMN property in the supercritical case for CIR process. For this, let us formulate the following assumption on the ratio of the coefficients of equation \eqref{c2eq1} we shall work with 
\begin{align*}
{\bf(A)}\qquad\frac{a}{\sigma}>\frac{11+\sqrt{89}}{2}.
\end{align*}
Under condition {\bf(A)} and $b_0>0$, we obtain in Theorem \ref{c2theorem}	the LAN property for subcritical case with rates of convergence $(\varphi_{1,n}^{-1}(a_0,b_0),\varphi_{2,n}^{-1}(a_0,b_0))=(\sqrt{n\Delta_n},\sqrt{n\Delta_n})$. In the critical case, assuming condition {\bf(A)}, $b_0=0$ and that $n\Delta_n^{3}\to 0$ as $n\to\infty$, we prove in Theorem \ref{c2theorem2} the LAQ property holds with rates of convergence $(\varphi_{1,n}^{-1}(a_0,b_0),\varphi_{2,n}^{-1}(a_0,b_0))=(\sqrt{\log(n\Delta_n)},n\Delta_n)$. Finally, in the supercritical case when $a$ is known and $b$ is only unknown parameter, assuming condition {\bf(A)}, $b_0<0$ and $\Delta_n^{2}e^{-b_0n\Delta_n}\to 0$, we derive the validity of the LAMN property for the likelihood at $b_0$ with rate of convergence $\varphi_{2,n}^{-1}(b_0)=e^{-b_0\frac{n\Delta_n}{2}}$.

	When the LAN property holds at $(a_0,b_0)$ with rates of convergence $(\varphi_{1,n}^{-1}(a_0,b_0),\varphi_{2,n}^{-1}(a_0,b_0))$ and asymptotic Fisher information matrix $I(a_0,b_0)$, a sequence of estimators $\{(\widehat{a}_n,\widehat{b}_n)\}_{n\in\mathbb{N}^{\ast}}$ of $(a_0,b_0)$ is said to be asymptotically efficient at $(a_0,b_0)$ in the sense of H\'ajek-Le Cam convolution theorem if as $n\to\infty$,
	$$
	\varphi_n^{-1}(a_0,b_0)\big((\widehat{a}_n,\widehat{b}_n)-(a_0,b_0)\big)^{\ast}\overset{\mathcal{L}(\widehat{\P}^{a_0,b_0})}{\longrightarrow}\mathcal{N}\left(0,I(a_0,b_0)^{-1}\right),
	$$
	where $\varphi_n^{-1}(a_0,b_0):=\textup{diag}(\varphi_{1,n}^{-1}(a_0,b_0),\varphi_{2,n}^{-1}(a_0,b_0))$ is the diagonal matrix. Notice that a sequence of estimators which is asymptotically efficient in the sense of H\'ajek-Le Cam convolution theorem achieves asymptotically the Cram\'er-Rao lower bound $I(a_0,b_0)^{-1}$ for the estimation variance.	

	When the LAMN property holds for the likelihood at $(a_0,b_0)$ with rates of convergence $(\varphi_{1,n}^{-1}(a_0,b_0),\varphi_{2,n}^{-1}(a_0,b_0))$ and asymptotic random Fisher information matrix $I(a_0,b_0)$, convolution and minimax theorems can be applied. On the one hand, the convolution theorem \cite[Corollary 1]{JP82} suggests the notion of asymptotically efficient estimators. That is, a sequence of estimators $\{(\widehat{a}_n,\widehat{b}_n)\}_{n\in\mathbb{N}^{\ast}}$ of the parameter $(a_0,b_0)$ is said to be asymptotically efficient at $(a_0,b_0)$ in the sense of H\'ajek-Le Cam convolution theorem if as $n\to\infty$,
	$$
	\varphi_n^{-1}(a_0,b_0)\big((\widehat{a}_n,\widehat{b}_n)-(a_0,b_0)\big)^{\ast}\overset{\mathcal{L}(\widehat{\P}^{a_0,b_0})}{\longrightarrow}I(a_0,b_0)^{-\frac{1}{2}}\mathcal{N}\left(0,I_2\right),
	$$
	where $I(a_0,b_0)$ and $\mathcal{N}(0,I_2)$ are independent. On the other hand, as a consequence of the minimax theorem \cite[Proposition 2]{JP82}, the lower bound for the asymptotic variance of any estimators is given by $I(a_0,b_0)^{-1}$. For details, we refer the reader to e.g. \cite{JP82,CY90}. It is worth noticing that the estimators constructed from a discretization of the time-continuous MLE given in \cite[Theorem 8 and 9]{BK13}  are asymptotically efficient in the subcritical and critical cases since their variance attains the lower bound for the asymptotic variance of estimators with the optimal rate of convergence (see Remark \ref{efficiency}).
	
To obtain the aforementioned results, our Malliavin calculus approach allows to obtain an appropriate stochastic expansion of the log-likelihood ratio (see Proposition \ref{c2prop1}). To treat the main term in the asymptotic behavior of the expansion, in the subcritical case, we apply a central limit theorem for triangular arrays of random variables together with the ergodic property whereas in the critical and supercritical cases, the corresponding convergence results \eqref{criticalconvergence2}-\eqref{criticalconvergence} and \eqref{supercriticalconvergence2} on the CIR process are essentially used. The difficult part of the proof is to deal with the negligible terms of the expansion.  In \cite{G01, G02}, a change of transition density functions of the diffusion processes is performed in order to be able to use the upper and lower bounds of Gaussian type of the transition density functions. This allows to measure the deviation of the change of transition density functions when the parameters change. For the CIR process, the transition density estimates of Gaussian type may not exist since the diffusion coefficient and its derivative are not bounded. To overcome these difficulties, instead of changing transition density functions, change of measures is essentially used (see Lemma \ref{change}). Then a technical Lemma \ref{deviation1} is established in order to measure the deviation of the change of measures when the drift parameters change. 

Furthermore, some $L^p$-norm estimation for positive and negative polynomial moments and exponential moment estimates of the CIR process taken from \cite{BK13, BD07} are needed in order to show the convergence of the negligible terms (see Lemma \ref{moment2}-\ref{moment} and Lemma \ref{expmoment}-\ref{estimates}). For this, condition {\bf(A)} above turns out to be crucial (see Remark \ref{conditionA}). 

When using our strategy, we do not require some additional assumptions on the decreasing rate of $\Delta_n$ such as $n\Delta_n^{p}\to 0$ for some $p>1$ in the subcritical (ergodic) case. However, in the non-ergodic cases, we require $n\Delta_n^{3}\to 0$ for the critical case and $\Delta_n^{2}e^{-b_0n\Delta_n}\to 0$ for the supercritical case.

This paper is organized as follows. In Section 2, we state our main results in Theorem \ref{c2theorem}, \ref{c2theorem2} and \ref{c2theorem4} which correspond respectively to the subcritical, critical and supercritical cases. Section 3 is devoted to the proof of the main results, which follows the aforementioned strategy. Using techniques of Malliavin calculus, Section 4 introduces technical results needed for the proof of the main results, which are related to an explicit expression for the score functions in terms of a conditional expectation of a Skorohod integral, and a decomposition of the Skorohod integral. The proofs of these technical results are postponed to Appendix A in order to maintain the flow of the exposition. Section $5$ focuses on studying the convergence of the remainder terms in the expansion \eqref{generalexpansion} of the log-likelihood ratio. Finally, some useful results are presented in Appendix B.

\section{Main results}
In this section, we give a statement of our main results in this paper which is divided into three cases: subcritical, critical and supercritical cases. Recall that our sampling scheme satisfies the high-frequency and infinite horizon conditions. That is, $\Delta_n\rightarrow 0$ and $n\Delta_n\rightarrow\infty$ as $n\rightarrow\infty$. 
\subsection{Subcritical case}
Assume that $b>0$. In this case, $X^{a,b}$ is ergodic (see e.g. \cite{JMRT13}) and its unique stationary distribution which we denote by $\pi_{a,b}(dx)$ is a Gamma law with shape parameter $\frac{a}{\sigma}$ and scale parameter $\frac{\sigma}{b}$ (see \cite{BK13}). That is,
\begin{equation}\label{gamma}
\pi_{a,b}(dx)=\Big(\dfrac{b}{\sigma}\Big)^{\frac{a}{\sigma}}\dfrac{1}{\Gamma(\frac{a}{\sigma})}x^{\frac{a}{\sigma}-1}e^{-\frac{b}{\sigma}x}dx.
\end{equation}
Furthermore, $X_t^{a,b}$ converges in law as $t\to\infty$ towards a random variable $X_{\infty}^{a,b}$ whose distribution is given by $\pi_{a,b}(dx)$ (see, e.g., \cite[Proposition 3 and 4]{BK12} or \cite[Theorem 5.3]{BP15}). 
Moreover, for any $\pi_{a,b}(dx)$-integrable function $h$, $X^{a,b}$ has the ergodic property in the sense that as $t\to\infty$,
\begin{align}\label{convergenceergodic}
\dfrac{1}{t}\int_{0}^{t}h(X_s^{a,b})ds\longrightarrow\widehat{\E}^{a,b}[h(X_{\infty}^{a,b})]=\int_{0}^{\infty}h(x)\pi_{a,b}(dx),\quad \widehat{\P}^{a,b}\textnormal{-a.s.}
\end{align}
Now, for fixed $(a_0,b_0)\in \Theta\times \Sigma_{+}$ where $\Sigma_{+}$ is a closed interval of $\R^{\ast}_+$, we consider a discrete observation $X^{n,a_0,b_0}=(X_{t_0}^{a_0,b_0}, X_{t_1}^{a_0,b_0},\ldots,X_{t_n}^{a_0,b_0})$ of the process $X^{a_0,b_0}$. The first result of this paper is the LAN property.
\begin{theorem}\label{c2theorem} Assume condition {\bf(A)} and $b_0>0$. Then,
the LAN property holds for the likelihood at $(a_0,b_0)$ with rate of convergence $(\sqrt{n\Delta_n},\sqrt{n\Delta_n})$ and asymptotic Fisher information matrix $I(a_0,b_0)$. That is, for all $z=(u,v)^{\ast}\in\R^2$, as $n\to\infty$,
\begin{equation*} 
\log\dfrac{d\P_n^{a_n,b_n}}{d\P^{a_0,b_0}_n}(X^{n,a_0,b_0})
\overset{\mathcal{L}(\widehat{\P}^{a_0,b_0})}{\longrightarrow} z^{\ast}\mathcal{N}\left(0,I(a_0,b_0)\right)-\dfrac{1}{2}z^{\ast} I(a_0,b_0)z,
\end{equation*}
where $a_n:=a_0+\dfrac{u}{\sqrt{n\Delta_n}}$, $b_n:=b_0+\dfrac{v}{\sqrt{n\Delta_n}}$, and $\mathcal{N}(0,I(a_0,b_0))$ is a centered $\R^2$-valued Gaussian vector with covariance matrix 
$$I(a_0,b_0):=\dfrac{1}{2\sigma}\begin{pmatrix}\dfrac{b_0}{a_0-\sigma}&-1\\-1&\dfrac{a_0}{b_0}
\end{pmatrix}.$$
\end{theorem}

\subsection{Critical case}

Assume that $b=0$. In this case, from \cite[Proposition 1]{BK12}, as $t\to\infty$,
\begin{align}
\Big(\dfrac{X_t^{a,0}}{t},\dfrac{1}{t^2}\int_0^tX_s^{a,0}ds\Big)\overset{\mathcal{L}(\widehat{\P}^{a,0})}{\longrightarrow} \Big(R_1^{a,0},\int_0^1R_s^{a,0}ds\Big).\label{criticalconvergence2}
\end{align}
Here, $R^{a,0}=(R_t^{a,0})_{t\geq 0}$ is the CIR process starting from $0$, solution to
\begin{equation}\label{cirr0}
dR_t^{a,0}=adt+\sqrt{2\sigma R_t^{a,0}}dB_t,
\end{equation}
where $R_0^{a,0}=0$. Moreover, from \cite[Proposition 2]{BK12} when $a>\sigma$, as $t\to\infty$,
\begin{equation}\label{criticalconvergence}
\dfrac{1}{\log t}\int_0^t\dfrac{ds}{X_s^{a,0}}\overset{\widehat{\P}^{a,0}}{\longrightarrow}\dfrac{1}{a-\sigma}.
\end{equation}
It is worth noticing that the formulae \eqref{criticalconvergence2} can be obtained by a scaling argument. See Remark on page 614 of Ben Alaya and Kebaier \cite{BK12}.

For fixed $(a_0,0)\in\Theta\times\Sigma$, consider a discrete observation $X^{n,a_0,0}=(X_{t_0}^{a_0,0}, X_{t_1}^{a_0,0},\ldots,X_{t_n}^{a_0,0})$ of the process $X^{a_0,0}$. The second result of this paper is the following LAQ property.
\begin{theorem}\label{c2theorem2}
Assume condition {\bf(A)}, $b_0=0$ and that $n\Delta_n^{3}\to 0$ as $n\to\infty$. Then, the LAQ property holds for the likelihood at $(a_0,0)$ with rates of convergence $(\sqrt{\log(n\Delta_n)},n\Delta_n)$ and random matrix $I(a_0,0)$. That is, for all $z=(u,v)^{\ast}\in\R^2$, as $n\to\infty$,
\begin{equation*}
\log\dfrac{d\P_n^{a_n,b_n}}{d\P^{a_0,0}_n}(X^{n,a_0,0}) 
\overset{\mathcal{L}(\widehat{\P}^{a_0,0})}{\longrightarrow} z^{\ast}U(a_0,0)-\dfrac{1}{2}z^{\ast}I(a_0,0)z,
\end{equation*}
and
$$
\widehat{\E}^{a_0,0}\Big[e^{z^{\ast}U(a_0,0)-\frac{1}{2}z^{\ast}I(a_0,0)z}\Big]=1,
$$ 
where $a_n:=a_0+\frac{u}{\sqrt{\log(n\Delta_n)}}$, $b_n:=0+\frac{v}{n\Delta_n}$, and $U(a_0,0)$ is a $\R^2$-valued random vector given by
$$
U(a_0,0):=\Big(\dfrac{1}{\sqrt{2\sigma(a_0-\sigma)}}G,\dfrac{a_0-R_1^{a_0,0}}{2\sigma}\Big)^{\ast}
$$
with covariance matrix 
$$
I(a_0,0):=\dfrac{1}{2\sigma}\begin{pmatrix}\dfrac{1}{a_0-\sigma}&0\\0&\int_0^1R_s^{a_0,0}ds
\end{pmatrix}.
$$
Here, $G$ is a standard normal random variable independent of $(R_1^{a_0,0},\int_0^1R_s^{a_0,0}ds)$, where $(R_t^{a,0})_{t\geq 0}$ is the CIR process starting from $0$ defined by \eqref{cirr0}.
\end{theorem}
We notice that Theorem \ref{c2theorem2} solves the problem left open in the last three lines in Subsection 3.2 of Overbeck \cite{O98} for time-continuous observations.

	Next, we state the result in the case where there is only one unknown parameter.
	\begin{corollary}\label{oneparacritical}
			\textnormal{(i)} Let $b=0$ be known and let $a$ be unknown parameter. Assume condition {\bf(A)}. Then, the LAN property holds for the likelihood at $a_0$ with rate of convergence $\sqrt{\log(n\Delta_n)}$ and asymptotic Fisher information $I(a_0)=\frac{1}{2\sigma(a_0-\sigma)}$. That is, for all $u\in\R$, as $n\to\infty$,
			\begin{equation*}
			\log\dfrac{d\P_n^{a_n,0}}{d\P^{a_0,0}_n}(X^{n,a_0,0}) 
			\overset{\mathcal{L}(\widehat{\P}^{a_0,0})}{\longrightarrow} u\mathcal{N}(0,I(a_0))-\dfrac{u^2}{2}I(a_0).
			\end{equation*}
			
			\textnormal{(ii)} Let $a$ be known and let $b=0$ be unknown parameter. Assume condition {\bf(A)} and that $n\Delta_n^{3}\to 0$ as $n\to\infty$. Then, the LAQ property holds for the likelihood at $b_0=0$ with rate of convergence $n\Delta_n$ and random variable $U(0)=\frac{a-R_1^{a,0}}{2\sigma}$ whose conditional variance is given by $I(0)=\frac{1}{2\sigma}\int_0^1R_s^{a,0}ds$. That is, for all $v\in\R$, as $n\to\infty$,
			\begin{equation*}
			\log\dfrac{d\P_n^{a,b_n}}{d\P^{a,0}_n}(X^{n,a,0}) 
			\overset{\mathcal{L}(\widehat{\P}^{a,0})}{\longrightarrow} vU(0)-\dfrac{v^2}{2}I(0),
			\end{equation*}
			and
			$$
			\widehat{\E}^{a,0}\big[e^{vU(0)-\frac{v^2}{2}I(0)}\big]=1.
			$$			
		\end{corollary}

\subsection{Supercritical case}

Assume that $b<0$. In this case, from \cite[Proposition 3]{BK12}, as $t\to\infty$,
\begin{align}
&\Big(e^{bt}X_t^{a,b},e^{bt}\int_0^tX_s^{a,b}ds\Big)\overset{\mathcal{L}(\widehat{\P}^{a,b})}{\longrightarrow} \Big(R_{-\frac{1}{b}}^a,-\dfrac{1}{b}R_{-\frac{1}{b}}^a\Big).\label{supercriticalconvergence2}
\end{align}
Here, $(R_t^a)_{t\geq 0}$ is the CIR process starting from $x_0$, solution to 
\begin{equation}\label{cirr}
dR_t^a=adt+\sqrt{2\sigma R_t^a}dB_t,
\end{equation}
where $R_0^a=x_0$. The first component in the right hand side in \eqref{supercriticalconvergence2} specifies the limit variable called $W$ in the last line of Theorem 3 (i) of Overbeck \cite{O98}. This limit is also immediate from the formula (8.10) on page 236 of Ikeda-Watanabe \cite{IW89}	combined with the formula (1.1) of Pinsky \cite{P72}.

Let $a$ be known and let $b<0$ be unknown parameter. For fixed $b_0\in \Sigma_{-}$ where $\Sigma_{-}$ is a closed interval of $\R^{\ast}_-=\R_{-}\setminus\{0\}$, we consider a discrete observation $X^{n,a_0,b_0}=(X_{t_0}^{a_0,b_0}, X_{t_1}^{a_0,b_0},\ldots,X_{t_n}^{a_0,b_0})$ of the process $X^{a_0,b_0}$ where $a_0$ is known. The last result of this paper is the following LAMN property in the case where $b$ is the only one unknown parameter.
\begin{theorem}\label{c2theorem4} Assume condition {\bf(A)}, $b_0<0$ and that $\Delta_n^{2}e^{-b_0n\Delta_n}\to 0$ as $n\to\infty$. Then, the LAMN property holds for the likelihood at $b_0$ with rate of convergence $e^{-b_0\frac{n\Delta_n}{2}}$ and asymptotic random Fisher information $I(b_0):=-\frac{1}{2\sigma b_0}R_{-\frac{1}{b_0}}^{a_0}$. That is, for all $v\in\R$, as $n\to\infty$,
	\begin{equation*}
	\log\dfrac{d\P_n^{a_0,b_n}}{d\P^{a_0,b_0}_n}(X^{n,a_0,b_0})\overset{\mathcal{L}(\widehat{\P}^{a_0,b_0})}{\longrightarrow} v\sqrt{I(b_0)}\mathcal{N}\left(0,1\right)-\dfrac{v^2}{2}I(b_0),
	\end{equation*}
	where $b_n:=b_0+\frac{v}{e^{-b_0\frac{n\Delta_n}{2}}}$, $\mathcal{N}(0,1)$ is a centered standard Gaussian random variable which is independent of $I(b_0)$. Here, $(R_t^a)_{t\geq 0}$ is the CIR process starting from $x_0$ defined by \eqref{cirr}.
\end{theorem}
We notice that Theorem \ref{c2theorem4} improves the result obtained by Overbeck \cite{O98} in Subsection 3.3 for time-continuous observations. 
\begin{remark}\label{conditionA}
Let us mention that condition {\bf(A)} on the ratio of the coefficients $\frac{a}{\sigma}$ required in Theorem \ref{c2theorem}, \ref{c2theorem2} and \ref{c2theorem4} is similar to condition $(10)$ in \cite[Theorem 2.2]{BBD08} which is used to prove the strong convergence of the symmetrized Euler scheme applied to CIR process. In fact, this technical condition comes from our techniques used in this paper and it is needed to treat Lemma \ref{lemma1} and \ref{lemma4}. It is worth noticing that the lower bound $\frac{11+\sqrt{89}}{2}$ appearing in condition {\bf(A)} is fixed in an optimal way in the sense that the H\"older weights used to control the remainder terms are carefully chosen so that our final condition {\bf (A)}: $\frac{a}{\sigma}\gtrsim 10.22$ is slightly different from the  condition $\frac{a}{\sigma}> 5+3\sqrt{2}\approx 9.24$ needed for the representation of the score function in Proposition \ref{c2prop1}.
\end{remark}
\begin{remark}\label{efficiency} The asymptotic efficiency of MLE based on continuous observations can be obtained using \cite[Theorem 5 and 6]{BK13} and \cite[Theorem 1]{BK12}. Indeed, combining our main results Theorem \ref{c2theorem} and Theorem \ref{c2theorem2} together with Theorem 8 and Theorem 9 in \cite{BK13}, the discrete MLE estimator given by (15) in \cite{BK13} achieves the lower bound given by the inverse of the asymptotic Fisher information matrix $I(a_0,b_0)$ and $I(a_0,0)$ in the subcritical and critical cases, respectively. Note that the MLE estimator for $a$ or $(a,b)$ is not consistent in the supercritical case (see Theorem 2 in \cite{BK12} or Theorem 7 in \cite{BK13}). We can only estimate the parameter $b$ for this supercritical case (see Subsection 3.3 in \cite{O98}).
\end{remark}
\begin{remark}
From Section 3 of \cite{O98} and Theorem \ref{c2theorem}, \ref{c2theorem2} and \ref{c2theorem4}, it is worth noticing that in the subcritical, critical and supercritical cases, no information is lost with respect to the continuous case under the high-frequency discrete observation scheme since the rates of convergence and the limit information coincide in both continuous and discrete cases.
\end{remark}
As usual, positive constants will be denoted by $C, C_1, C_2$. They may change of value from one line to the next.

\section{Proof of main results}
In this section, we give the proof of Theorem \ref{c2theorem}, \ref{c2theorem2} and \ref{c2theorem4}. Let us first recall the following useful results taken from \cite{BK13}.
\begin{lemma}\label{moment2}\cite[Proposition $3$]{BK13} 
	\begin{itemize}
		\item[\textnormal{(i)}] For all $p<\frac{a}{\sigma}$ and $b>0$, we have \; $\sup_{t\geq 0}\widehat{\E}^{a,b}[\frac{1}{(X_t^{a,b})^{p}}]<\infty$.
		\item[\textnormal{(ii)}] For all $p<\frac{a}{\sigma}$ and  $b=0$, we have $$\sup_{0\leq t\leq 1}\widehat{\E}^{a,0}[\frac{1}{(X_t^{a,0})^{p}}]<\infty\quad\textnormal{and}\quad \sup_{t\geq 1}\frac{\widehat{\E}^{a,0}[\frac{1}{(X_t^{a,0})^{p}}]}{t^{^{-p}}}<\infty.$$
		\item[\textnormal{(iii)}] For all $0<p<\frac{a}{\sigma}$ and $b<0$, we have $$
		\sup_{0\leq t\leq 1}\widehat{\E}^{a,b}[\frac{1}{(X_t^{a,b})^{p}}]<\infty\quad\textnormal{and}\quad\sup_{t\geq 1}\frac{\widehat{\E}^{a,b}[\frac{1}{(X_t^{a,b})^{p}}]}{t^{^{-p}}}<\infty.$$
	\end{itemize}	
\end{lemma}
Recall that \textnormal{(i)} and \textnormal{(ii)} of Lemma \ref{moment2} are taken from \cite[Proposition $3$]{BK13}. The statement \textnormal{(iii)} of Lemma \ref{moment2} is due to the Comparison Theorem giving that $X_t^{a,0}\leq X_t^{a,b}$ a.s. for all $t\geq 0$ when $b<0$ as their corresponding drift coefficients satisfy $a-0x<a-bx$.
\begin{lemma}\label{moment}\cite[Proposition $4$ and $5$]{BK13} Let $t>s\geq 0$ satisfying that $0<t-s<1$.
	\begin{itemize}
		\item[\textnormal{(i)}] Assume that $b>0$. Then for any $q\geq 1$, there exists a constant $C>0$ such that
		$$
		\widehat{\E}^{a,b}\big[\vert X_t^{a,b}-X_s^{a,b}\vert^q\big]\leq C\left(t-s\right)^{\frac{q}{2}}.
		$$
		\item[\textnormal{(ii)}] Assume that $b=0$. Then there exists a constant $C>0$ such that
		\begin{align*}
		\widehat{\E}^{a,0}\big[\vert X_t^{a,0}-X_s^{a,0}\vert^q\big]\leq\begin{cases}
		C\left(t-s\right)^{\frac{q}{2}}\underset{u\in[s,t]}{\sup}\widehat{\E}^{a,0}[(X_u^{a,0})^{\frac{q}{2}}],\quad\textnormal{for any}\ q\geq 2,\\
		C\left(at+x_0\right)^{\frac{q}{2}}\left(t-s\right)^{\frac{q}{2}},\quad\textnormal{for any}\ q\in[1,2).
		\end{cases} 
		\end{align*}
	\end{itemize}	
\end{lemma}
From \eqref{maindecompo}, the log-likelihood ratio can be written as follows
\begin{equation}\label{generalexpansion}\begin{split}
\log\dfrac{d\P_n^{a_n,b_n}}{d\P^{a_0,b_0}_n}(X^{n,a_0,b_0})
&=\sum_{k=0}^{n-1}u\varphi_{1,n}(a_0,b_0)\int_0^1\dfrac{\partial_{a}p^{a(\ell),b_0}}{p^{a(\ell),b_0}}(\Delta_n,X_{t_k}^{a_0,b_0},X_{t_{k+1}}^{a_0,b_0})d\ell\\
&\qquad+\sum_{k=0}^{n-1}v\varphi_{2,n}(a_0,b_0)\int_0^1\dfrac{\partial_{b}p^{a_n,b(\ell)}}{p^{a_n,b(\ell)}}(\Delta_n,X_{t_k}^{a_0,b_0},X_{t_{k+1}}^{a_0,b_0})d\ell\\
&=\sum_{k=0}^{n-1}\left(\xi_{k,n}+\eta_{k,n}\right)+\sum_{k=0}^{n-1}\left(R_{k,n}+H_{k,n}\right),
\end{split}
\end{equation}
where $\xi_{k,n}$ and $\eta_{k,n}$ are the main terms whereas $R_{k,n}$, $H_{k,n}$ are the remainder terms in the expansion. Two couples $(\xi_{k,n},R_{k,n})$ and $(\eta_{k,n},H_{k,n})$ will appear from the analysis of two corresponding score functions $\frac{\partial_{a}p^{a(\ell),b_0}}{p^{a(\ell),b_0}}(\Delta_n,X_{t_k}^{a_0,b_0},X_{t_{k+1}}^{a_0,b_0})$ and $\frac{\partial_{b}p^{a_n,b(\ell)}}{p^{a_n,b(\ell)}}(\Delta_n,X_{t_k}^{a_0,b_0},X_{t_{k+1}}^{a_0,b_0})$. The main terms are given by
\begin{equation}\label{xieta}\begin{split}
\xi_{k,n}&=u\varphi_{1,n}(a_0,b_0)\dfrac{1}{2\sigma X_{t_k}^{a_0,b_0}}\big(\sqrt{2\sigma X_{t_k}^{a_0,b_0}}\left(B_{t_{k+1}}-B_{t_{k}}\right)-\dfrac{1}{2}u\varphi_{1,n}(a_0,b_0)\Delta_n\big),\\
\eta_{k,n}&=\dfrac{-v}{2\sigma}\varphi_{2,n}(a_0,b_0)\Big(\sqrt{2\sigma X_{t_k}^{a_0,b_0}}(B_{t_{k+1}}-B_{t_{k}})\\
&\qquad-\Delta_n\big(u\varphi_{1,n}(a_0,b_0)-\frac{1}{2}v\varphi_{2,n}(a_0,b_0)X_{t_k}^{a_0,b_0}\big)\Big),
\end{split}
\end{equation}
with in the subcritical case ($b_0>0$) $(\varphi_{1,n}(a_0,b_0),\varphi_{2,n}(a_0,b_0))=(\frac{1}{\sqrt{n\Delta_n}},\frac{1}{\sqrt{n\Delta_n}})$, in the critical case ($b_0=0$) $(\varphi_{1,n}(a_0,0),\varphi_{2,n}(a_0,0))=(\frac{1}{\sqrt{\log(n\Delta_n)}},\frac{1}{n\Delta_n})$, and in the supercritical case ($b_0<0$) $(\varphi_{1,n}(a_0,b_0),\varphi_{2,n}(a_0,b_0))=(0,e^{b_0\frac{n\Delta_n}{2}})$. The remainder terms $R_{k,n}$, $H_{k,n}$ are implicitly defined by the decomposition \eqref{generalexpansion} and they will be explicitly determined in the proof of Lemma \ref{negligibleterms}. Then, our aim will be to study the asymptotic of the main terms $\sum_{k=0}^{n-1}(\xi_{k,n}+\eta_{k,n})$. Note that the convergence of this sum requires a weaker condition than {\bf(A)}, that is $\frac{a}{\sigma}>4$. However, the convergence of the remainder terms $\sum_{k=0}^{n-1}(R_{k,n}+H_{k,n})$ given in the following lemma will require condition {\bf(A)}.
\begin{lemma}\label{negligibleterms}
	Assume condition {\bf(A)} and our setting $\Delta_n\rightarrow 0$ and $n\Delta_n\rightarrow\infty$ as $n\rightarrow\infty$. Then, for the subcritical case, as $n\to\infty$,
	\begin{align}\label{remainder}
	&\sum_{k=0}^{n-1}\left(R_{k,n}+H_{k,n}\right)\overset{\widehat{\P}^{a_0,b_0}}{\longrightarrow}0.
	\end{align}	
	Furthermore, if we assume the additional condition $n\Delta_n^{3}\to 0$ for the critical case and $\Delta_n^{2}e^{-b_0n\Delta_n}\to 0$ for the supercritical case, then \eqref{remainder} remains valid for both cases.
\end{lemma}
The proof of this lemma will be given in Section $5$. 

From \eqref{generalexpansion} and Lemma \ref{negligibleterms}, to prove the main results, it suffices to show that as $n\to\infty$,
\begin{align*} 
&\sum_{k=0}^{n-1}\left(\xi_{k,n}+\eta_{k,n}\right)\overset{\mathcal{L}(\widehat{\P}^{a_0,b_0})}{\longrightarrow} z^{\ast}U(a_0,b_0)-\dfrac{1}{2}z^{\ast} I(a_0,b_0)z.
\end{align*}
In order to control the convergence in the subcritical case, we will need the following classical discrete ergodic theorem which is analogous to \cite[Lemma 3.1]{G02} established in the case of ergodic diffusions with globally Lipschitz coefficients.
\begin{lemma}\label{c3ergodic} Let $b_0>0$. Then, for all $p<\frac{1}{2}(\frac{a_0}{\sigma}-1)$, as $n\to\infty$,
	\begin{equation*}
	\dfrac{1}{n}\sum_{k=0}^{n-1}\frac{1}{(X_{t_k}^{a_0,b_0})^p}\overset{\widehat{\P}^{a_0,b_0}}{\longrightarrow}\int_{0}^{\infty}\frac{1}{x^p}\pi_{a_0,b_0}(dx).
	\end{equation*}
\end{lemma}
\begin{proof} First, using It\^o's formula, \textnormal{(i)} of Lemma \ref{moment2}, and Burkholder-Davis-Gundy's (BDG's) inequality, for any $t_k\leq s\leq t_{k+1}$ we get that
	\begin{align*}
	&\widehat{\E}^{a_0,b_0}\Big[\big\vert \frac{1}{(X_s^{a_0,b_0})^p}-\frac{1}{(X_{t_k}^{a_0,b_0})^p}\big\vert\Big]=\widehat{\E}^{a_0,b_0}\Big[\big\vert \int_{t_k}^{s}\big(\frac{b_0p}{(X_u^{a_0,b_0})^{p}}+((p+1)\sigma-a_0)\frac{p}{(X_u^{a_0,b_0})^{p+1}}\big)du\\
	&-\sqrt{2\sigma}p\int_{t_k}^{s}\frac{dB_u}{(X_u^{a_0,b_0})^{p+\frac{1}{2}}}\big\vert\Big]\leq C(s-t_k)\big(\sup_{u\in[t_k,s]}\widehat{\E}^{a_0,b_0}[\frac{1}{(X_u^{a_0,b_0})^{p}}]+\sup_{u\in[t_k,s]}\widehat{\E}^{a_0,b_0}[\frac{1}{(X_u^{a_0,b_0})^{p+1}}]\big)\\
	&+C\big(\widehat{\E}^{a_0,b_0}\big[\big\vert \int_{t_k}^{s}\frac{dB_u}{(X_u^{a_0,b_0})^{p+\frac{1}{2}}}\big\vert^2\big]\big)^{\frac{1}{2}}\leq C\Delta_n+C\big(\int_{t_k}^{s}\widehat{\E}^{a_0,b_0}\big[\frac{1}{(X_u^{a_0,b_0})^{2p+1}}\big]du\big)^{\frac{1}{2}} 	\\ &\leq C\Delta_n+C\big((s-t_k)\sup_{u\in[t_k,s]}\widehat{\E}^{a_0,b_0}\big[\frac{1}{(X_u^{a_0,b_0})^{2p+1}}\big]\big)^{\frac{1}{2}}\leq C\sqrt{\Delta_n},	
	\end{align*}
	for some constant $C>0$ which is independent of $t_k$ and $s$, where the moments above are bounded thanks to $p<\frac{1}{2}(\frac{a_0}{\sigma}-1)$. This shows that
	\begin{align}\label{converg1}
	&\widehat{\E}^{a_0,b_0}\Big[\big\vert\frac{1}{n\Delta_n}\int_0^{n\Delta_n}\frac{ds}{(X_s^{a_0,b_0})^p}-\frac{1}{n}\sum_{k=0}^{n-1}\frac{1}{(X_{t_k}^{a_0,b_0})^p}\big\vert\Big]\notag\\
	&\leq\frac{1}{n\Delta_n}\sum_{k=0}^{n-1}\int_{t_k}^{t_{k+1}}\widehat{\E}^{a_0,b_0}\Big[\big\vert \frac{1}{(X_s^{a_0,b_0})^p}-\frac{1}{(X_{t_k}^{a_0,b_0})^p}\big\vert\Big]ds\leq C\sqrt{\Delta_n},
	\end{align}
	which tends to zero as $n\to\infty$. On the other hand, from \eqref{convergenceergodic}, as $n\to\infty$,
	\begin{align}\label{converg2}
	\frac{1}{n\Delta_n}\int_0^{n\Delta_n}\frac{ds}{(X_s^{a_0,b_0})^p}\longrightarrow\int_{0}^{\infty}\frac{1}{x^p}\pi_{a_0,b_0}(dx),\quad \widehat{\P}^{a_0,b_0}\textnormal{-a.s,}
	\end{align}
	where the limit above is finite. Thus, the result follows from \eqref{converg1} and \eqref{converg2}.
\end{proof}

\subsection{Proof of Theorem \ref{c2theorem}}
\label{sec:proof}
\begin{proof}
	From Lemma \ref{negligibleterms} with $(\varphi_{1,n}(a_0,b_0),\varphi_{2,n}(a_0,b_0))=(\frac{1}{\sqrt{n\Delta_n}},\frac{1}{\sqrt{n\Delta_n}})$, we need only to prove that
as $n\to\infty$,
	\begin{align*}
\sum_{k=0}^{n-1}\left(\xi_{k,n}+\eta_{k,n}\right)\overset{\mathcal{L}(\widehat{\P}^{a_0,b_0})}{\longrightarrow} z^{\ast}\mathcal{N}\left(0,I(a_0,b_0)\right)-\dfrac{1}{2}z^{\ast} I(a_0,b_0)z,
\end{align*}
where $\xi_{k,n}$ and $\eta_{k,n}$ are given by \eqref{xieta}. To do so, we will use Theorem \ref{clt}. At first, let us recall that in the subcritical case ($b>0$) the stationary distribution $\pi_{a,b}(dx)$ (see \eqref{gamma}) satisfies
\begin{align*}
\int_{0}^{\infty}x\pi_{a,b}(dx)=\dfrac{a}{b},\;
\int_{0}^{\infty}\dfrac{1}{x}\pi_{a,b}(dx)=\dfrac{b}{a-\sigma}.
\end{align*}
Then, by Lemma \ref{c3ergodic}, it is easy to check that as $n \rightarrow \infty$,
	\begin{align*}
	\sum_{k=0}^{n-1}\widehat{\E}^{a_0,b_0}[\xi_{k,n}\vert \widehat{\mathcal{F}}_{t_k}]=-\dfrac{u^2}{4\sigma}\dfrac{1}{n}\sum_{k=0}^{n-1}\frac{1}{X_{t_k}^{a_0,b_0}}&\overset{\widehat{\P}^{a_0,b_0}}{\longrightarrow}-\dfrac{u^2}{4\sigma}\int_{0}^{\infty}\frac{1}{x}\pi_{a_0,b_0}(dx)=-\dfrac{u^2}{2}\dfrac{b_0}{2\sigma(a_0-\sigma)},
	\end{align*}
	under condition $\frac{a}{\sigma}>3$, and 
	\begin{align*}
	&\sum_{k=0}^{n-1}\widehat{\E}^{a_0,b_0}[\eta_{k,n}\vert \widehat{\mathcal{F}}_{t_k}]=\dfrac{uv}{2\sigma}-\dfrac{v^2}{4\sigma}\dfrac{1}{n}\sum_{k=0}^{n-1}X_{t_k}^{a_0,b_0}\overset{\widehat{\P}^{a_0,b_0}}{\longrightarrow}\dfrac{uv}{2\sigma}-\dfrac{v^2}{4\sigma}\int_{0}^{\infty}x\pi_{a_0,b_0}(dx)=\dfrac{uv}{2\sigma}-\dfrac{v^2}{2}\dfrac{a_0}{2\sigma b_0},
	\end{align*}
	which gives condition (i) of Theorem \ref{clt}.
	
	Next, we check condition (ii) of Theorem \ref{clt} similarly by using the same arguments. We get
	\begin{align*}
	&\sum_{k=0}^{n-1}\big(\widehat{\E}^{a_0,b_0}[\xi_{k,n}^2\vert \widehat{\mathcal{F}}_{t_k}]-\big(\widehat{\E}^{a_0,b_0}[\xi_{k,n}\vert\widehat{\mathcal{F}}_{t_k}]\big)^2\big)=\dfrac{u^2}{2\sigma}\dfrac{1}{n}\sum_{k=0}^{n-1}\frac{1}{X_{t_k}^{a_0,b_0}}\overset{\widehat{\P}^{a_0,b_0}}{\longrightarrow}u^2\dfrac{b_0}{2\sigma(a_0-\sigma)},\\
	&\sum_{k=0}^{n-1}\big(\widehat{\E}^{a_0,b_0}[\eta_{k,n}^2\vert \widehat{\mathcal{F}}_{t_k}]-\big(\widehat{\E}^{a_0,b_0}[\eta_{k,n}\vert \widehat{\mathcal{F}}_{t_k}]\big)^2\big)=\dfrac{v^2}{2\sigma}\dfrac{1}{n}\sum_{k=0}^{n-1}X_{t_k}^{a_0,b_0}\overset{\widehat{\P}^{a_0,b_0}}{\longrightarrow}v^2\dfrac{a_0}{2\sigma b_0},\\
	&\sum_{k=0}^{n-1}\big(\widehat{\E}^{a_0,b_0}[\xi_{k,n}\eta_{k,n}\vert \widehat{\mathcal{F}}_{t_k}]-\widehat{\E}^{a_0,b_0}[\xi_{k,n}\vert \widehat{\mathcal{F}}_{t_k}]\widehat{\E}^{a_0,b_0}[\eta_{k,n}\vert \widehat{\mathcal{F}}_{t_k}]\big)=-\dfrac{uv}{2\sigma},
	\end{align*}
	where the first convergence holds under condition $\frac{a_0}{\sigma}>3$. Finally, it remains to check condition (iii) of Theorem \ref{clt}. Indeed, there exists $c>0$ such that
\begin{align*}
	&\sum_{k=0}^{n-1}\widehat{\E}^{a_0,b_0}[\xi_{k,n}^4\vert \widehat{\mathcal{F}}_{t_k}]\leq \dfrac{c}{n^2}\sum_{k=0}^{n-1}\dfrac{1}{(X_{t_k}^{a_0,b_0})^2}+\dfrac{c}{n^4}\sum_{k=0}^{n-1}\dfrac{1}{(X_{t_k}^{a_0,b_0})^4},\\
	&\sum_{k=0}^{n-1}\widehat{\E}^{a_0,b_0}[\eta_{k,n}^4\vert \widehat{\mathcal{F}}_{t_k}]\leq \dfrac{c}{n^2}\sum_{k=0}^{n-1}(X_{t_k}^{a_0,b_0})^2+\dfrac{c}{n^3}+\dfrac{c}{n^4}\sum_{k=0}^{n-1}(X_{t_k}^{a_0,b_0})^4,
\end{align*}
for some constant $c>0$. Thus applying Lemma \ref{moment2} \textnormal{(i)}, under condition $\frac{a_0}{\sigma}>4$ we have $\widehat{\E}^{a_0,b_0}[\sum_{k=0}^{n-1}\widehat{\E}^{a_0,b_0}[\xi_{k,n}^4\vert \widehat{\mathcal{F}}_{t_k}]]\leq \frac{C}{n}$ and $\widehat{\E}^{a_0,b_0}[\sum_{k=0}^{n-1}\widehat{\E}^{a_0,b_0}[\eta_{k,n}^4\vert \widehat{\mathcal{F}}_{t_k}]]\leq \frac{C}{n}$, which tends to zero. Clearly, this condition is satisfied under {\bf(A)}. Thus, the result follows.
\end{proof}

\subsection{Proof of Theorem \ref{c2theorem2}}
\begin{proof}
	From \eqref{generalexpansion}, Lemma \ref{negligibleterms} with $(\varphi_{1,n}(a_0,0),\varphi_{2,n}(a_0,0))=(\frac{1}{\sqrt{\log(n\Delta_n)}},\frac{1}{n\Delta_n})$, we need only to prove that
	as $n\to\infty$, 
	\begin{align}\label{laq1}
	\sum_{k=0}^{n-1}\left(\xi_{k,n}+\eta_{k,n}\right)\overset{\mathcal{L}(\widehat{\P}^{a_0,0})}{\longrightarrow} z^{\ast}U(a_0,0)-\dfrac{1}{2}z^{\ast}I(a_0,0)z,
	\end{align}
	and
	\begin{align}\label{laq2}
	\widehat{\E}^{a_0,0}\left[e^{z^{\ast}U(a_0,0)-\frac{1}{2}z^{\ast}I(a_0,0)z}\right]=1.
	\end{align}
To do this, we rewrite
	\begin{equation}\begin{split}\label{cr1}
	\sum_{k=0}^{n-1}\left(\xi_{k,n}+\eta_{k,n}\right)&=z^{\ast}U_n(a_0,0)-\dfrac{1}{2}z^{\ast}I_n(a_0,0)z+R_{n,1}^{a_0,0}+R_{n,2}^{a_0,0}+\dfrac{uv}{2\sigma\sqrt{\log(n\Delta_n)}},
	\end{split}
	\end{equation}		
	where $t_n=n\Delta_n$ and
	\begin{align*}
	&U_n(a_0,0)=\dfrac{1}{\sqrt{2\sigma}}\begin{pmatrix}\dfrac{1}{\sqrt{\log t_n}}\displaystyle\int_{0}^{t_n}\frac{dB_s}{\sqrt{X_s^{a_0,0}}}\\
	-\dfrac{1}{t_n}\displaystyle\int_{0}^{t_n}\sqrt{X_s^{a_0,0}}dB_s
	\end{pmatrix},\\
	&I_n(a_0,0)=\dfrac{1}{2\sigma}\begin{pmatrix}\dfrac{1}{\log t_n}\displaystyle\int_{0}^{t_n}\frac{ds}{X_s^{a_0,0}}&0\\
	0&\dfrac{1}{t_n^2}\displaystyle\int_{0}^{t_n}X_s^{a_0,0}ds
	\end{pmatrix},\\
	R_{n,1}^{a_0,0}&=\dfrac{u^2}{4\sigma\log t_n}\sum_{k=0}^{n-1}\int_{t_k}^{t_{k+1}}\big(\frac{1}{X_s^{a_0,0}}-\frac{1}{X_{t_k}^{a_0,0}}\big)ds+\dfrac{v^2}{4\sigma t_n^2}\sum_{k=0}^{n-1}\int_{t_k}^{t_{k+1}}(X_s^{a_0,0}-X_{t_k}^{a_0,0})ds,\\
	R_{n,2}^{a_0,0}&=-\dfrac{u}{\sqrt{2\sigma}\sqrt{\log t_n}}\sum_{k=0}^{n-1}\int_{t_k}^{t_{k+1}}(\frac{1}{\sqrt{X_s^{a_0,0}}}-\frac{1}{\sqrt{X_{t_k}^{a_0,0}}})dB_s\\
	&\qquad+\dfrac{v}{\sqrt{2\sigma}t_n}\sum_{k=0}^{n-1}\int_{t_k}^{t_{k+1}}(\sqrt{X_s^{a_0,0}}-\sqrt{X_{t_k}^{a_0,0}})dB_s.
	\end{align*}
	First, using It\^o's formula and equation \eqref{c2eq1}, we get that	
	\begin{equation*}\begin{split}
	U_n(a_0,0)=\dfrac{1}{\sqrt{2\sigma}}\begin{pmatrix}\dfrac{1}{\sqrt{2\sigma}\sqrt{\log t_n}}\Big(\log X_{t_n}^{a_0,0}-\log x_0+(\sigma-a_0) \displaystyle\int_{0}^{t_n}\frac{ds}{X_s^{a_0,0}}\Big)\\
	-\dfrac{1}{\sqrt{2\sigma}t_n}\big(X_{t_n}^{a_0,0}-x_0-a_0t_n\big)
	\end{pmatrix}.
	\end{split}
	\end{equation*}
	Moreover, it follows from the proof of \cite[Theorem 6]{BK13} and \eqref{criticalconvergence} that as $n\to\infty$,
	\begin{align*}
	&\left(\dfrac{\log X_{t_n}^{a_0,0}+(\sigma-a_0) \displaystyle\int_{0}^{t_n}\frac{ds}{X_s^{a_0,0}}}{\sqrt{\log t_n}},\ \dfrac{1}{t_n^2}\displaystyle\int_{0}^{t_n}X_s^{a_0,0}ds,\ \dfrac{X_{t_n}^{a_0,0}}{t_n},\ \dfrac{1}{\log t_n}\displaystyle\int_{0}^{t_n}\frac{ds}{X_s^{a_0,0}}\right)\\
	&\qquad\qquad\overset{\mathcal{L}(\widehat{\P}^{a_0,0})}{\longrightarrow}\left(\sqrt{\dfrac{2\sigma}{a_0-\sigma}}G,\ \int_0^1R_s^{a_0,0}ds,\ R_1^{a_0,0},\ \frac{1}{a_0-\sigma}\right),
	\end{align*}
	where $G$ is a standard normal random variable independent of $(R_1^{a_0,0},\int_0^1R_s^{a_0,0}ds)$. This implies that as $n\to\infty$,	
	\begin{equation}\begin{split}\label{cr2}
	\left(U_n(a_0,0),I_n(a_0,0)\right)\overset{\mathcal{L}(\widehat{\P}^{a_0,0})}{\longrightarrow}\left(U(a_0,0),I(a_0,0)\right).
	\end{split}
	\end{equation}
	Next, thanks to the proof of \cite[Theorem 9]{BK13}, under condition $\frac{a_0}{\sigma}>2$, we have that $R_{n,1}^{a_0,0}\overset{\widehat{\P}^{a_0,0}}{\longrightarrow}0$ as $n\to\infty$. Now, we treat $R_{n,2}^{a_0,0}$. For this, we write $R_{n,2}^{a_0,0}=\sum_{k=0}^{n-1}(\zeta_{k,n,1}+\zeta_{k,n,2})$, where 
	\begin{align*}
	\zeta_{k,n,1}&=-\dfrac{u}{\sqrt{2\sigma}\sqrt{\log t_n}}\int_{t_k}^{t_{k+1}}(\frac{1}{\sqrt{X_s^{a_0,0}}}-\frac{1}{\sqrt{X_{t_k}^{a_0,0}}})dB_s,\\
	\zeta_{k,n,2}&=	\dfrac{v}{\sqrt{2\sigma}t_n}\int_{t_k}^{t_{k+1}}(\sqrt{X_s^{a_0,0}}-\sqrt{X_{t_k}^{a_0,0}})dB_s.
	\end{align*}
	We apply Lemma \ref{zero} to $\zeta_{k,n,1}$ and $\zeta_{k,n,2}$. As it is obvious that $\sum_{k=0}^{n-1}\widehat{\E}^{a_0,0}[\zeta_{k,n,1}\vert \widehat{\mathcal{F}}_{t_k}]=0$, we only need to check that $\sum_{k=0}^{n-1}\widehat{\E}^{a_0,0}[\zeta_{k,n,1}^2\vert \widehat{\mathcal{F}}_{t_k}]\overset{\widehat{\P}^{a_0,0}}{\longrightarrow}0$. To do so, we use It\^o's formula, Lemma \ref{moment2} \textnormal{(ii)} and standard calculations to get
	\begin{align*}
	&\widehat{\E}^{a_0,0}\Big[\big\vert\sum_{k=0}^{n-1}\widehat{\E}^{a_0,0}[\zeta_{k,n,1}^2\vert \widehat{\mathcal{F}}_{t_k}]\big\vert\Big]=\sum_{k=0}^{n-1}\widehat{\E}^{a_0,0}[\widehat{\E}^{a_0,0}[\zeta_{k,n,1}^2\vert \widehat{\mathcal{F}}_{t_k}]]=\sum_{k=0}^{n-1}\widehat{\E}^{a_0,0}[\zeta_{k,n,1}^2]\\
	&=\dfrac{u^2}{2\sigma\log t_n}\sum_{k=0}^{n-1}\int_{t_k}^{t_{k+1}}\widehat{\E}^{a_0,0}\Big[\big(\int_{t_k}^{s}(\frac{-a_0}{2}+\frac{3\sigma}{4})\frac{du}{(X_u^{a_0,0})^{\frac{3}{2}}}+\int_{t_k}^{s}\sqrt{\frac{\sigma}{2}}\frac{dB_u}{X_u^{a_0,0}}\big)^2\Big]ds\\
	&\leq C\dfrac{u^2}{\log t_n}\sum_{k=0}^{n-1}\int_{t_k}^{t_{k+1}}\Big((s-t_k)^2\sup_{u\in[t_k,s]}\widehat{\E}^{a_0,0}\big[\frac{1}{(X_u^{a_0,0})^{3}}\big]+(s-t_k)\sup_{u\in[t_k,s]}\widehat{\E}^{a_0,0}\big[\frac{1}{(X_u^{a_0,0})^2}\big]\Big)ds\\
	&\leq C\dfrac{u^2\Delta_n}{\log (n\Delta_n)}\sum_{k=0}^{n-1}\int_{t_k}^{t_{k+1}}\big(\sup_{u\in[t_k,s]}\widehat{\E}^{a_0,0}\big[\frac{1}{(X_u^{a_0,0})^{3}}\big]+\sup_{u\in[t_k,s]}\widehat{\E}^{a_0,0}\big[\frac{1}{(X_u^{a_0,0})^2}\big]\big)ds\\
	&\leq \dfrac{u^2\Delta_n}{\log (n\Delta_n)}\big(C_1+C_2\big(\dfrac{1}{(n\Delta_n)^2}+\dfrac{1}{n\Delta_n}\big)\big),
	\end{align*}
	for some constants $C_1, C_2$, which tends to zero under $\frac{a_0}{\sigma}>3$. Thus,  $\sum_{k=0}^{n-1}\zeta_{k,n,1}\overset{\widehat{\P}^{a_0,0}}{\longrightarrow}0$.
	
	Similarly, we have $\sum_{k=0}^{n-1}\widehat{\E}^{a_0,0}[\zeta_{k,n,2}\vert \widehat{\mathcal{F}}_{t_k}]=0$. Next, applying It\^o's formula, Lemma \ref{moment2} \textnormal{(ii)} and standard calculations, we get
	\begin{align*}
	&\widehat{\E}^{a_0,0}[\vert\sum_{k=0}^{n-1}\widehat{\E}^{a_0,0}[\zeta_{k,n,2}^2\vert \widehat{\mathcal{F}}_{t_k}]\vert]=\sum_{k=0}^{n-1}\widehat{\E}^{a_0,0}[\zeta_{k,n,2}^2]=\dfrac{v^2}{2\sigma t_n^2}\sum_{k=0}^{n-1}\int_{t_k}^{t_{k+1}}\widehat{\E}^{a_0,0}[(\sqrt{X_s^{a_0,0}}-\sqrt{X_{t_k}^{a_0,0}})^2]ds\\
	&=\dfrac{v^2}{2\sigma t_n^2}\sum_{k=0}^{n-1}\int_{t_k}^{t_{k+1}}\widehat{\E}^{a_0,0}\Big[\big(\int_{t_k}^{s}(\frac{a_0}{2}-\frac{\sigma}{4})\frac{du}{\sqrt{X_u^{a_0,0}}}+\int_{t_k}^{s}\sqrt{\frac{\sigma}{2}}dB_u\big)^2\Big]ds\\
	&\leq C\dfrac{v^2}{t_n^2}\sum_{k=0}^{n-1}\int_{t_k}^{t_{k+1}}\big((s-t_k)^2\sup_{u\in[t_k,s]}\widehat{\E}^{a_0,0}\big[\frac{1}{X_u^{a_0,0}}\big]+(s-t_k)\big)ds\\
	&\leq Cv^2\dfrac{1}{(n\Delta_n)^2}\big(\Delta_n^2(C_1+C_2\log (n\Delta_n))+n\Delta_n^2\big),
	\end{align*}
	which tends to zero. Thus, $\sum_{k=0}^{n-1}\widehat{\E}^{a_0,0}[\zeta_{k,n,2}^2\vert \widehat{\mathcal{F}}_{t_k}]\overset{\widehat{\P}^{a_0,0}}{\longrightarrow}0$. Then, we have  $\sum_{k=0}^{n-1}\zeta_{k,n,2}\overset{\widehat{\P}^{a_0,0}}{\longrightarrow}0$. Consequently, under condition $\frac{a_0}{\sigma}>3$, we have that $R_{n,2}^{a_0,0}\overset{\widehat{\P}^{a_0,0}}{\longrightarrow}0$ as $n\to\infty$. Therefore, under $\frac{a_0}{\sigma}>3$, $R_{n,1}^{a_0,0}+R_{n,2}^{a_0,0}\overset{\widehat{\P}^{a_0,0}}{\longrightarrow}0$ as $n\to\infty$. This combined with \eqref{cr1}-\eqref{cr2} gives \eqref{laq1}.
	
	Finally, we treat \eqref{laq2}. For this, let us consider the CIR process $R^{a,b}=(R_t^{a,b})_{t\geq 0}$ starting from $0$ defined by
	\begin{equation*}
	dR_t^{a,b}=(a-b R_t^{a,b})dt+\sqrt{2\sigma R_t^{a,b}}dB_t,
	\end{equation*}
	where $R_0^{a,b}=0$, $(a,b)\in \R^{\ast}_+\times\R$. We denote by $\P_{R}^{a,b}$ the probability measure induced by the CIR process $R^{a,b}$ on the measurable space $(C(\R_{+},\R),\mathcal{B}(C(\R_{+},\R))$. Let $T>0$. For any $t>0$, let $\P_{R,t}^{a,b}$ be the restriction of $\P_{R}^{a,b}$ on $\widehat{\mathcal{F}}_t$. As a consequence of \cite[Lemma 3.1]{BP16}, for any $a>\sigma$, $b\in\R$, the probability measures $\P_{R,t}^{a,0}$ and $\P_{R,t}^{a,b}$ are absolutely continuous with respect to each other and its Radon-Nikodym derivative is given by
	\begin{equation}\label{RN}
	\dfrac{d\P_{R,t}^{a,b}}{d\P_{R,t}^{a,0}}\big((R_s^{a,0})_{s\in[0,t]}\big)=\exp\Big\{\dfrac{b}{\sqrt{2\sigma}}\int_0^t\sqrt{R_s^{a,0}}dB_s-\dfrac{b^2}{4\sigma}\int_0^tR_s^{a,0}ds\Big\}.
	\end{equation}
	On the other hand, as a consequence of Theorem 3.4 in Chapter III of Jacod and Shiryaev \cite{JS03}, the Radon-Nikodym derivative process	$
	(\frac{d\P_{R,t}^{a,b}}{d\P_{R,t}^{a,0}}((R_s^{a,0})_{s\in[0,t]}))_{t\in [0,T]}$ is a martingale w.r.t. the filtration $(\widehat{\mathcal{F}}_t)_{t\in [0,T]}$.
	
	Now, using the independence between $G$ and $(R_1^{a_0,0},\int_0^1R_s^{a_0,0}ds)$, we have 
	$$
	\widehat{\E}^{a_0,0}\big[e^{z^{\ast}U(a_0,0)-\frac{1}{2}z^{\ast}I(a_0,0)z}\big]=\E_1\E_2,
	$$
	where 
	\begin{align*}
	\E_1=\widehat{\E}^{a_0,0}\big[e^{\frac{u}{\sqrt{2\sigma(a_0-\sigma)}}G-\frac{u^2}{4\sigma(a_0-\sigma)}}\big],\;
	\E_2=\widehat{\E}^{a_0,0}\big[e^{\frac{v}{2\sigma}(a_0-R_1^{a_0,0})-\frac{v^2}{4\sigma}\int_0^1R_s^{a_0,0}ds}\big].
	\end{align*}
	Clearly, $\E_1=1$ since $G$ is the standard normal random variable. Using equation \eqref{cirr0} and \eqref{RN}, we have that
	\begin{align*}
	\E_2=\widehat{\E}^{a_0,0}\big[e^{\frac{v}{\sqrt{2\sigma}}\int_0^1\sqrt{R_s^{a_0,0}}dB_s-\frac{v^2}{4\sigma}\int_0^1R_s^{a_0,0}ds}\big]=\widehat{\E}^{a_0,0}\big[\dfrac{d\P_{R,1}^{a_0,v}}{d\P_{R,1}^{a_0,0}}\left((R_s^{a_0,0})_{s\in[0,1]}\right)\big].
	\end{align*}
	Using the fact that the Radon-Nikodym derivative process
	$(\frac{d\P_{R,t}^{a_0,v}}{d\P_{R,t}^{a_0,0}}((R_s^{a_0,0})_{s\in[0,t]}))_{t\in [0,T]}$ is a martingale w.r.t. the filtration $(\widehat{\mathcal{F}}_t)_{t\in [0,T]}$, we get that
	$$
	\widehat{\E}^{a_0,0}\big[\dfrac{d\P_{R,1}^{a_0,v}}{d\P_{R,1}^{a_0,0}}\big((R_s^{a_0,0})_{s\in[0,1]}\big)\big]=\widehat{\E}^{a_0,0}\big[\dfrac{d\P_{R,0}^{a_0,v}}{d\P_{R,0}^{a_0,0}}(R_0^{a_0,0})\big]=1,
	$$
	which implies that $\E_2=1$. This concludes \eqref{laq2}. Thus, the result follows.
\end{proof}	

\subsection{Proof of Corollary \ref{oneparacritical}}
\label{proofoneparacri}
\begin{proof}
	\textnormal{(i)} When $b=0$ is known and $a$ is unknown parameter, the result is obtained as a consequence of Theorem \ref{c2theorem2} by taking $v=0$. This is reason why the condition $n\Delta_n^{3}\to 0$ as $n\to\infty$ is removed.
	
	\textnormal{(ii)} When $a$ is known and $b=0$ is unknown parameter, the result is obtained as a consequence of Theorem \ref{c2theorem2} by taking $u=0$. 
\end{proof}	

\subsection{Proof of Theorem \ref{c2theorem4}}
\begin{proof}
	From \eqref{generalexpansion}, Lemma \ref{negligibleterms} with $(\varphi_{1,n}(a_0,b_0),\varphi_{2,n}(a_0,b_0))=(0,e^{b_0\frac{n\Delta_n}{2}})$, we need only to prove that
	as $n\to\infty$,
	\begin{align}\label{notlaq}
	\sum_{k=0}^{n-1}\left(\xi_{k,n}+\eta_{k,n}\right)\overset{\mathcal{L}(\widehat{\P}^{a_0,b_0})}{\longrightarrow} v\sqrt{I(b_0)}\mathcal{N}\left(0,1\right)-\dfrac{v^2}{2}I(b_0).
	\end{align}
 To do this, we rewrite
	\begin{equation}\begin{split}\label{supercr}
	\sum_{k=0}^{n-1}\left(\xi_{k,n}+\eta_{k,n}\right)=vU_n(a_0,b_0)-\dfrac{v^2}{2}I_n(a_0,b_0)+R_{n,3}^{a_0,b_0}+R_{n,4}^{a_0,b_0},
	\end{split}
	\end{equation}	
	where 
	\begin{align*}
	&U_n(a_0,b_0)=\dfrac{-e^{b_0\frac{t_n}{2}}}{\sqrt{2\sigma}}\int_{0}^{t_n}\sqrt{X_s^{a_0,b_0}}dB_s,\; I_n(a_0,b_0)=\dfrac{1}{2\sigma}e^{b_0t_n}\int_{0}^{t_n}X_s^{a_0,b_0}ds,\\
	&R_{n,3}^{a_0,b_0}=\dfrac{v^2}{4\sigma e^{-b_0n\Delta_n}}\sum_{k=0}^{n-1}\int_{t_k}^{t_{k+1}}(X_s^{a_0,b_0}-X_{t_k}^{a_0,b_0})ds,\\
	&R_{n,4}^{a_0,b_0}=\dfrac{v}{\sqrt{2\sigma}e^{-b_0\frac{n\Delta_n}{2}}}\sum_{k=0}^{n-1}\int_{t_k}^{t_{k+1}}\big(\sqrt{X_s^{a_0,b_0}}-\sqrt{X_{t_k}^{a_0,b_0}}\big)dB_s.
	\end{align*}	
	First, it follows from the proof of \cite[Theorem 7.1]{BP16} that as $n\to\infty$,
	\begin{align*}
	&\Big(e^{b_0\frac{t_n}{2}}\displaystyle\int_{0}^{t_n}\sqrt{X_s^{a_0,b_0}}dB_s,\ e^{b_0t_n}\displaystyle\int_{0}^{t_n}X_s^{a_0,b_0}ds\Big)\overset{\mathcal{L}(\widehat{\P}^{a_0,b_0})}{\longrightarrow}\Big(\big(-\dfrac{1}{b_0}R_{-\frac{1}{b_0}}^{a_0}\big)^{\frac{1}{2}}\mathcal{N}\left(0,1\right),\ -\dfrac{1}{b_0}R_{-\frac{1}{b_0}}^{a_0}\Big),
	\end{align*}
	where $\mathcal{N}\left(0,1\right)$ is a standard normal random variable independent of $R_{-\frac{1}{b_0}}^{a_0}$. This implies that 	
	\begin{equation}\begin{split}\label{supcr2}
	(U_n(a_0,b_0),I_n(a_0,b_0))\overset{\mathcal{L}(\widehat{\P}^{a_0,b_0})}{\longrightarrow}(\sqrt{I(b_0)}\mathcal{N}\left(0,1\right),I(b_0)).
	\end{split}
	\end{equation}
	Next, applying It\^o's formula, Lemma \ref{moment2} \textnormal{(iii)}, \eqref{supercriticalconvergence2} and standard calculations, we get
	\begin{align*}
	&\widehat{\E}^{a_0,b_0}\big[\vert R_{n,3}^{a_0,b_0}\vert\big]\leq \dfrac{v^2}{4\sigma e^{-b_0n\Delta_n}}\sum_{k=0}^{n-1}\int_{t_k}^{t_{k+1}}\widehat{\E}^{a_0,b_0}\Big[\big\vert \int_{t_k}^{s}(a_0-b_0X_u^{a_0,b_0})du+\int_{t_k}^{s}\sqrt{2\sigma X_u^{a_0,b_0}} dB_u\big\vert\Big]ds\\
	&\leq \dfrac{v^2}{4\sigma e^{-b_0n\Delta_n}}\sum_{k=0}^{n-1}\int_{t_k}^{t_{k+1}}\Big((s-t_k)+(s-t_k)\sup_{u\in[t_k,s]}\widehat{\E}^{a_0,b_0}\big[X_u^{a_0,b_0}\big]\\
	&\qquad+\big((s-t_k)\sup_{u\in[t_k,s]}\widehat{\E}^{a_0,b_0}\big[X_u^{a_0,b_0}\big]\big)^{\frac{1}{2}}\Big)ds\leq C\big(n\Delta_n^2e^{b_0n\Delta_n}+\Delta_n+\sqrt{\Delta_n}e^{\frac{1}{2}b_0n\Delta_n}\big),
	\end{align*}
	which tends to zero. Thus, $R_{n,3}^{a_0,b_0}\overset{\widehat{\P}^{a_0,b_0}}{\longrightarrow}0$. Finally, we treat $R_{n,4}^{a_0,b_0}$ by applying Lemma \ref{zero}. First, we have $\sum_{k=0}^{n-1}\widehat{\E}^{a_0,b_0}[R_{n,4}^{a_0,b_0}\vert \widehat{\mathcal{F}}_{t_k}]=0$. Next, applying It\^o's formula, Lemma \ref{moment2} \textnormal{(iii)}, \eqref{supercriticalconvergence2} and standard calculations, we get
	\begin{align*}
	&\widehat{\E}^{a_0,b_0}\Big[\big\vert\sum_{k=0}^{n-1}\widehat{\E}^{a_0,b_0}[(R_{n,4}^{a_0,b_0})^2\vert \widehat{\mathcal{F}}_{t_k}]\big\vert\Big]=\sum_{k=0}^{n-1}\widehat{\E}^{a_0,b_0}[(R_{n,4}^{a_0,b_0})^2]\\	
	&=\dfrac{v^2}{2\sigma e^{-b_0n\Delta_n}}\sum_{k=0}^{n-1}\int_{t_k}^{t_{k+1}}\widehat{\E}^{a_0,b_0}\Big[\big(\int_{t_k}^{s}((\frac{a_0}{2}-\frac{\sigma}{4})\frac{1}{\sqrt{X_u^{a_0,b_0}}}-\frac{b_0}{2}\sqrt{X_u^{a_0,b_0}})du+\int_{t_k}^{s}\sqrt{\frac{\sigma}{2}}dB_u\big)^2\Big]ds\\
	&\leq \dfrac{Cv^2}{e^{-b_0n\Delta_n}}\sum_{k=0}^{n-1}\int_{t_k}^{t_{k+1}}\big((s-t_k)^2\big(\sup_{u\in[t_k,s]}\widehat{\E}^{a_0,b_0}\big[\frac{1}{X_u^{a_0,b_0}}\big]+\sup_{u\in[t_k,s]}\widehat{\E}^{a_0,b_0}\big[X_u^{a_0,b_0}\big]\big)+(s-t_k)\big)ds\\
	&\leq Ce^{b_0n\Delta_n}\Big(\Delta_n^2+\Delta_n^2\log(n\Delta_n)+n\Delta_n^2\Big)+C\Delta_n^2,
	\end{align*}
	which tends to zero under condition $\frac{a_0}{\sigma}>1$. Thus, $\sum_{k=0}^{n-1}\widehat{\E}^{a_0,b_0}[(R_{n,4}^{a_0,b_0})^2\vert \widehat{\mathcal{F}}_{t_k}]\overset{\widehat{\P}^{a_0,b_0}}{\longrightarrow}0$. Then, we have  $\sum_{k=0}^{n-1}R_{n,4}^{a_0,b_0}\overset{\widehat{\P}^{a_0,b_0}}{\longrightarrow}0$. Therefore, under $\frac{a_0}{\sigma}>1$ as $n\to\infty$
	\begin{equation}\begin{split}\label{supercr4}	
	R_{n,3}^{a_0,b_0}+R_{n,4}^{a_0,b_0}\overset{\widehat{\P}^{a_0,b_0}}{\longrightarrow}0.	
	\end{split}
	\end{equation}
	Therefore, from \eqref{supercr}-\eqref{supercr4}, we conclude \eqref{notlaq}.
\end{proof}

\section{Technical results}
\label{sec:prelim}
The aim of this section is to use techniques of Malliavin calculus in order to represent the remainder terms in Lemma \ref{negligibleterms} as a conditional expectation of Skorohod integral (see Proposition \ref{c2prop1}) that will be analyzed in Section \ref{prooflemma21}. To be more precise, we first consider on a complete probability space $(\widetilde{\Omega}, \widetilde{\mathcal{F}}, \widetilde{\P})$ the flow process $Y^{a,b}(s,x)=(Y_t^{a,b}(s,x), t\geq s)$, $x\in\R^{\ast}_+$ on the time interval $[s,\infty)$ and with initial condition $Y_{s}^{a,b}(s,x)=x$ satisfying
\begin{equation}\label{flow}
Y_t^{a,b}(s,x)=x+\int_s^t \big(a-bY_u^{a,b}(s,x)\big)du +\int_s^t\sqrt{2\sigma Y_u^{a,b}(s,x)}dW_u,
\end{equation}
for any $t\geq s$, where $W=(W_t)_{t\geq 0}$ is a standard Brownian motion. Let $\{\widetilde{\mathcal{F}}_t\}_{t\geq 0}$ denote the natural filtration generated by $W$. For all $t\geq 0$ we denote $Y_t^{a,b}\equiv Y_t^{a,b}(0,x_0)$ which is defined by
\begin{equation}\label{c2eq1rajoute}
Y_t^{a,b}=x_0+\int_0^t \big(a-bY_s^{a,b}\big)ds +\int_0^t\sqrt{2\sigma Y_s^{a,b}}dW_s.
\end{equation}
Then, the Malliavin calculus on the Wiener space induced by $W$ will be applied to obtain a new representation of score functions. We refer the reader to Nualart \cite{N} for a detailed exposition of the Malliavin calculus and the notations therein. We denote by $D$ and $\delta$ the Malliavin derivative and the Skorohod integral w.r.t. $W$ on $[t_k,t_{k+1}]$. In what follows, let $\mathbb{D}^{1,2}$ denote the Sobolev space of random variables that are differentiable w.r.t. $W$ in the sense of Malliavin, and $\textnormal{Dom}\ \delta$ denote the domain of $\delta$. The Malliavin calculus for CIR process is developed e.g. in \cite{AE08, AN14}.  

For any $k \in \{0,...,n-1\}$, the process $(Y_t^{a,b}(t_k,x), t\in [t_k,t_{k+1}])$ is defined by
\begin{equation}\label{flowk}
Y_t^{a,b}(t_k,x)=x+\int_{t_k}^t \big(a-bY_u^{a,b}(t_k,x)\big)du +\int_{t_k}^t\sqrt{2\sigma Y_u^{a,b}(t_k,x)}dW_u.
\end{equation}
Recall that condition $a\geq \sigma$ ensures that the CIR process $Y^{a,b}$ remains almost surely strictly positive. Then, by \cite[Theorem V.39]{P05}, the process $(Y_t^{a,b}(t_k,x), t\in [t_k,t_{k+1}])$ is differentiable w.r.t. $x$ that we denote by $(\partial_{x}Y_t^{a,b}(t_k,x), t\in [t_k,t_{k+1}])$. Furthermore, this process admits derivatives w.r.t. $a$ and $b$ that we denote by $(\partial_{a}Y_t^{a,b}(t_k,x), t\in [t_k,t_{k+1}])$ and $(\partial_{b}Y_t^{a,b}(t_k,x), t\in [t_k,t_{k+1}])$, respectively, since this problem is similar to the derivative w.r.t. the initial
condition (see e.g. \cite[pages 294-295]{P16}). Under condition $a\geq \sigma$, these processes are solutions to the following equations
\begin{align}
&\partial_xY_t^{a,b}(t_k,x)=1-b\int_{t_k}^t \partial_xY_u^{a,b}(t_k,x)du +\int_{t_k}^t\frac{\sqrt{\sigma}}{\sqrt{2Y_u^{a,b}(t_k,x)}}\partial_xY_u^{a,b}(t_k,x)dW_u,\label{px}\\
&\partial_{a}Y_t^{a,b}(t_k,x)=\int_{t_k}^t\big(1-b\partial_{a}Y_u^{a,b}(t_k,x)\big)du +\int_{t_k}^t\frac{\sqrt{\sigma}}{\sqrt{2Y_u^{a,b}(t_k,x)}}\partial_aY_u^{a,b}(t_k,x)dW_u,\label{pa}\\
&\partial_{b}Y_t^{a,b}(t_k,x)=-\int_{t_k}^t\big(Y_u^{a,b}(t_k,x)+b\partial_{b}Y_u^{a,b}(t_k,x)\big)du +\int_{t_k}^t\frac{\sqrt{\sigma}\partial_bY_u^{a,b}(t_k,x)}{\sqrt{2Y_u^{a,b}(t_k,x)}}dW_u.\label{pb}
\end{align}
Therefore, their explicit solutions are respectively given by
\begin{align}
\partial_xY_t^{a,b}(t_k,x)&=\exp{\Big\{-b(t-t_k)-\dfrac{\sigma}{4}\int_{t_k}^t\dfrac{du}{Y_u^{a,b}(t_k,x)}+\sqrt{\dfrac{\sigma}{2}}\int_{t_k}^t\dfrac{dW_u}{\sqrt{Y_u^{a,b}(t_k,x)}}\Big\}},\label{dxe}\\
\partial_aY_t^{a,b}(t_k,x)&=\int_{t_k}^{t}\exp{\Big\{-b(t-r)-\dfrac{\sigma}{4}\int_{r}^t\dfrac{du}{Y_u^{a,b}(t_k,x)}+\sqrt{\dfrac{\sigma}{2}}\int_{r}^t\dfrac{dW_u}{\sqrt{Y_u^{a,b}(t_k,x)}}\Big\}}dr,\label{dxa}\\
\partial_bY_t^{a,b}(t_k,x)&=-\int_{t_k}^{t}Y_r^{a,b}(t_k,x)\exp{\big\{-b(t-r)-\frac{\sigma}{4}\int_{r}^t\frac{du}{Y_u^{a,b}(t_k,x)}+\sqrt{\frac{\sigma}{2}}\int_{r}^t\frac{dW_u}{\sqrt{Y_u^{a,b}(t_k,x)}}\big\}}dr.\label{dxb}
\end{align}
Observe that from \eqref{dxe}, \eqref{dxa} and \eqref{dxb}, we can write
\begin{align}
\partial_aY_t^{a,b}(t_k,x)&=\int_{t_k}^{t}\partial_xY_t^{a,b}(t_k,x)(\partial_xY_r^{a,b}(t_k,x))^{-1}dr,\label{dxa2}\\
\partial_bY_t^{a,b}(t_k,x)&=-\int_{t_k}^{t}Y_r^{a,b}(t_k,x)\partial_xY_t^{a,b}(t_k,x)(\partial_xY_r^{a,b}(t_k,x))^{-1}dr.\label{dxb2}
\end{align}
To analyze the score functions, some notations are introduced. Let $\widetilde{\P}^{a,b}$ denote the probability measure induced by  $Y^{a,b}=(Y_t^{a,b})_{t \geq 0}$ on the canonical space $(C(\R_{+},\R),\mathcal{B}(C(\R_{+},\R))$ endowed with the natural filtration $\{\widetilde{\mathcal{F}}_t\}_{t\geq 0}$, and $\widetilde{\E}^{a,b}$ denote the expectation w.r.t. $\widetilde{\P}^{a,b}$. For all $k \in \{0,...,n-1\}$ and $x\in\R^{\ast}_+$, we denote by $\widetilde{\P}_{t_k,x}^{a,b}$ and $\widetilde{\E}_{t_k,x}^{a,b}$ respectively the probability law of $Y^{a,b}$ starting at $x$ at time $t_k$ and the expectation w.r.t. $\widetilde{\P}_{t_k,x}^{a,b}$, i.e., $\widetilde{\P}_{t_k,x}^{a,b}(A)=\widetilde{\E}[{\bf 1}_{A}\vert Y_{t_{k}}^{a,b}=x]$ for all $A\in \widetilde{\mathcal{F}}$ and $\widetilde{\E}_{t_k,x}^{a,b}[V]=\widetilde{\E}[V\vert Y_{t_{k}}^{a,b}=x]$ for all $\widetilde{\mathcal{F}}$-measurable random variables $V$. Thus, $\widetilde{\E}_{t_k,x}^{a,b}$ is the expectation under the probability law of $Y^{a,b}$ starting at $x$ at time $t_k$. 

Next, we recall the following exponential moment estimate taken from \cite{BD07}.
\begin{lemma}\cite[Lemma 3.1]{BD07}\label{expmoment} Assume that $\frac{a}{\sigma}>2$. For any  $\mu\leq(\frac{a}{\sigma}-1)^2\frac{\sigma}{4}$, there exists a constant $C>0$ such that for any $k \in \{0,...,n-1\}$, $x\in\R^{\ast}_+$ and $t\in[t_k,t_{k+1}]$,
	\begin{align*}
	\widetilde{\E}_{t_k,x}^{a,b}\Big[\exp\Big\{\mu\int_{t_k}^t\dfrac{du}{Y_u^{a,b}(t_k,x)}\Big\}\Big]\leq C\big(1+\dfrac{1}{x^{\frac{1}{2}(\frac{a}{\sigma}-1)}}\big).
	\end{align*}
	Moreover, the statement remains valid for $Y^{a,b}$.
\end{lemma}
The proof for $b\geq 0$ is given in \cite[Lemma A.2]{BD07}. The proof for $b<0$ is based on the Comparison Theorem.

As mentioned before, our techniques are based on the Malliavin calculus, which will require some $L^p$-norm estimation for positive and negative polynomial moments of the flow process as well as the Malliavin derivative of the flow process and its inverse. 
\begin{lemma}\label{estimates} Assume condition $a\geq \sigma$. For all $p\geq 1$, there exists $C>0$ such that
	\begin{align}
	&\widetilde{\E}_{t_k,x}^{a,b}\big[\vert Y_t^{a,b}(t_k,x)\vert^p\big]\leq C\left(1+x^p\right).\label{e1}
	\end{align}
	For all $p\in[0,\frac{a}{\sigma}-1)$, there exists $C>0$ such that
	\begin{align}
	&\widetilde{\E}_{t_k,x}^{a,b}\big[\dfrac{1}{\vert Y_t^{a,b}(t_k,x)\vert^p}\big]\leq\dfrac{C}{x^p}.\label{e2}
	\end{align}
	For all $p\geq -\frac{(\frac{a}{\sigma}-1)^2}{2(\frac{a}{\sigma}-\frac{1}{2})}$, there exists $C>0$ such that
	\begin{align}
	&\widetilde{\E}_{t_k,x}^{a,b}\big[\vert \partial_xY_t^{a,b}(t_k,x)\vert^p\big]\leq C\big(1+\dfrac{1}{x^{\frac{\frac{a}{\sigma}-1+p}{2}}}\big),\label{e4}
	\end{align}
	for any $k \in \{0,...,n-1\}$, $x\in\R^{\ast}_+$ and $t\in[t_k,t_{k+1}]$.
\end{lemma} 

Now, let us recall that under condition $a\geq \sigma$, for any $t\in [t_k,t_{k+1}]$, the random variable $Y_t^{a,b}(t_k,x)$ belongs to $\mathbb{D}^{1,2}$ (see \cite[Corollary 4.2]{AE08}). From \eqref{flowk} and the chain rule of the Malliavin calculus, its corresponding Malliavin derivative satisfies the following equation
\begin{align}
D_sY_t^{a,b}(t_k,x)&=\sqrt{2\sigma Y_s^{a,b}(t_k,x)}-b\int_{s}^tD_sY_u^{a,b}(t_k,x)du \notag\\
&\qquad+\int_{s}^t\frac{\sqrt{\sigma}}{\sqrt{2Y_u^{a,b}(t_k,x)}}D_sY_u^{a,b}(t_k,x)dW_u,
\end{align}
for $s\leq t$, and $D_sY_t^{a,b}(t_k,x)=0$ for $s>t$. Furthermore, using the same arguments as \cite[(2.59)]{N}, the Malliavin derivative $D_sY_t^{a,b}(t_k,x)$ can be rewritten as
\begin{equation}\label{expression1}
D_sY_t^{a,b}(t_k,x)=\sqrt{2\sigma Y_s^{a,b}(t_k,x)}\partial_xY_t^{a,b}(t_k,x)(\partial_xY_s^{a,b}(t_k,x))^{-1}{\bf 1}_{[t_k,t]}(s).
\end{equation}
On the other hand, by condition $a\geq \sigma$ and \cite[Corollary 4.2]{AE08}, its explicit expression is given by
\begin{equation}\label{expression2}
D_sY_t^{a,b}(t_k,x)=\sqrt{2\sigma Y_t^{a,b}(t_k,x)}\exp\Big\{\int_s^t\Big(-\dfrac{b}{2}-\big(\dfrac{a}{2}-\dfrac{\sigma}{4}\big)\dfrac{1}{Y_u^{a,b}(t_k,x)}\Big)du\Big\}{\bf 1}_{[t_k,t]}(s).
\end{equation}
The Malliavin differentiability of the flow process and its inverse is guaranteed by the following lemma.
\begin{lemma}\label{Malliderivable} Let $(a,b)\in\Theta\times\Sigma$, $k \in \{0,...,n-1\}$, $t\in [t_k,t_{k+1}]$, and $x\in\R^{\ast}_+$.
	\begin{itemize}
		\item[\textnormal{(i)}] Assume condition $\frac{a}{\sigma}>4$. Then, $\partial_xY_t^{a,b}(t_k,x)\in \mathbb{D}^{1,2}$. Furthermore,
		\begin{align}\label{dmpx}
		D_s\big(\partial_xY_t^{a,b}(t_k,x)\big)&=\partial_xY_t^{a,b}(t_k,x)\Big(\sqrt{\frac{\sigma}{2}}\frac{1}{\sqrt{Y_s^{a,b}(t_k,x)}}+\frac{\sigma}{4}\int_{s}^t\frac{1}{(Y_u^{a,b}(t_k,x))^2}D_sY_u^{a,b}(t_k,x)du\notag\\
		&\qquad-\frac{1}{2}\sqrt{\frac{\sigma}{2}}\int_{s}^t\frac{1}{(Y_u^{a,b}(t_k,x))^{\frac{3}{2}}}D_sY_u^{a,b}(t_k,x)dW_u\Big){\bf 1}_{[t_k,t]}(s).
		\end{align} 
		\item[\textnormal{(ii)}] Assume condition $\frac{a}{\sigma}>\frac{9+\sqrt{57}}{2}$. Then, $(\partial_xY_t^{a,b}(t_k,x))^{-1}\in \mathbb{D}^{1,2}$ and $\frac{Y_{t}^{a,b}(t_k,x)}{\partial_xY_{t}^{a,b}(t_k,x)}\in \mathbb{D}^{1,2}$. Furthermore,
		\begin{align}
		&D_s\big(\dfrac{1}{\partial_xY_{t}^{a,b}(t_k,x)}\big)=\frac{-1}{(\partial_xY_t^{a,b}(t_k,x))^2}D_s(\partial_xY_t^{a,b}(t_k,x)){\bf 1}_{[t_k,t]}(s),\label{Mallinveflow1}\\
		&D_s\big(\dfrac{Y_{t}^{a,b}(t_k,x)}{\partial_xY_{t}^{a,b}(t_k,x)}\big)=\dfrac{1}{\partial_xY_{t}^{a,b}(t_k,x)}D_sY_{t}^{a,b}(t_k,x)+Y_{t}^{a,b}(t_k,x)D_s\big(\dfrac{1}{\partial_xY_{t}^{a,b}(t_k,x)}\big).\label{Mallinveflow2}
		\end{align} 
	\end{itemize}
\end{lemma}
For any $t>s$, the law of $Y_t^{a,b}$ conditioned on $Y_s^{a,b}=x$ possesses the transition density $p^{a,b}(t-s,x,y)$. Following \cite[Proposition 4.1]{G01}, the score functions are expressed in terms of a conditional expectation of a Skorohod integral.
\begin{proposition} \label{c2prop1}
	Assume condition $\frac{a}{\sigma}>\frac{9+\sqrt{57}}{2}$. Then, for all $k \in \{0,...,n-1\}$, $(a,b)\in\Theta\times \Sigma$, $\beta\in\{a,b\}$ and $x, y\in\R^{\ast}_+$,
	\begin{align*}
	\dfrac{\partial_{\beta}p^{a,b}}{p^{a,b}}(\Delta_n,x,y)=\dfrac{1}{\Delta_n}\widetilde{\E}_{t_k,x}^{a,b}\Big[\delta\big(\partial_{\beta}Y_{t_{k+1}}^{a,b}(t_k,x)U^{a,b}(t_k,x)\big)\vert Y_{t_{k+1}}^{a,b}=y\Big],
	\end{align*}
	where $U^{a,b}(t_k,x):=(U^{a,b}_t(t_k,x), t\in[t_k,t_{k+1}])$ with $U^{a,b}_t(t_k,x):=(D_tY_{t_{k+1}}^{a,b}(t_k,x))^{-1}$. 	
	
	Furthermore, under condition $\frac{a}{\sigma}>5+3\sqrt{2}$, the Skorohod integrals are decomposed as follows
	\begin{align}
	\delta\big(\partial_{a}Y_{t_{k+1}}^{a,b}(t_k,x)U^{a,b}(t_k,x)\big)&=\frac{\Delta_n}{\sqrt{2\sigma x}} (W_{t_{k+1}}-W_{t_{k}})+R_1^{a,b}+R_2^{a,b}+R_3^{a,b},\label{deria}\\
	\delta\big(\partial_{b}Y_{t_{k+1}}^{a,b}(t_k,x)U^{a,b}(t_k,x)\big)&=-\frac{\Delta_n}{\sqrt{2\sigma }}\sqrt{x} (W_{t_{k+1}}-W_{t_{k}})-xR_1^{a,b}+H_2^{a,b}+H_3^{a,b},\label{derib}
	\end{align}
	where 
	\begin{align*}
	&R_1^{a,b}=R_1^{a,b}(t_k,x)=\Delta_n\int_{t_k}^{t_{k+1}}\big(\frac{\partial_{x}Y_{s}^{a,b}(t_k,x)}{\sqrt{2\sigma Y_s^{a,b}(t_k,x)}}-\frac{\partial_{x}Y_{t_k}^{a,b}(t_k,x)}{\sqrt{2\sigma Y_{t_k}^{a,b}(t_k,x)}}\big)dW_s,\\
	&R_2^{a,b}=R_2^{a,b}(t_k,x)=\int_{t_k}^{t_{k+1}}\big(\frac{1}{\partial_{x}Y_{r}^{a,b}(t_k,x)}-\frac{1}{\partial_{x}Y_{t_k}^{a,b}(t_k,x)}\big)dr\int_{t_k}^{t_{k+1}}\frac{\partial_{x}Y_{s}^{a,b}(t_k,x)}{\sqrt{2\sigma Y_s^{a,b}(t_k,x)}}dW_s,\\
	&R_3^{a,b}=R_3^{a,b}(t_k,x)=-\int_{t_k}^{t_{k+1}}\int_{s}^{t_{k+1}}D_s\big(\dfrac{1}{\partial_xY_{r}^{a,b}(t_k,x)}\big)dr\frac{\partial_{x}Y_{s}^{a,b}(t_k,x)}{\sqrt{2\sigma Y_s^{a,b}(t_k,x)}}ds,\\
	&H_2^{a,b}=H_2^{a,b}(t_k,x)=-\int_{t_k}^{t_{k+1}}\big(\frac{Y_{r}^{a,b}(t_k,x)}{\partial_{x}Y_{r}^{a,b}(t_k,x)}-\frac{Y_{t_k}^{a,b}(t_k,x)}{\partial_{x}Y_{t_k}^{a,b}(t_k,x)}\big)dr\int_{t_k}^{t_{k+1}}\frac{\partial_{x}Y_{s}^{a,b}(t_k,x)}{\sqrt{2\sigma Y_s^{a,b}(t_k,x)}}dW_s,\\
	&H_3^{a,b}=H_3^{a,b}(t_k,x)=\int_{t_k}^{t_{k+1}}\int_{s}^{t_{k+1}}D_s\big(\dfrac{Y_{r}^{a,b}(t_k,x)}{\partial_xY_{r}^{a,b}(t_k,x)}\big)dr\frac{\partial_{x}Y_{s}^{a,b}(t_k,x)}{\sqrt{2\sigma Y_s^{a,b}(t_k,x)}}ds.
	\end{align*}		
\end{proposition}
The following crucial estimates for the terms appearing in Proposition \ref{c2prop1} will be needed.
\begin{lemma} \label{estimate}
	Let $(a,b)\in\Theta\times\Sigma$, $k \in \{0,...,n-1\}$, and $x\in\R^{\ast}_+$. 
	\begin{itemize}
	\item[\textnormal{(i)}] Under condition $\frac{a}{\sigma}>5+3\sqrt{2}$, we have
	\begin{align}
	&\widetilde{\E}_{t_k,x}^{a,b}\big[R_1^{a,b}+R_2^{a,b}+R_3^{a,b}\big]=0,\label{es1}\\
	&\widetilde{\E}_{t_k,x}^{a,b}\big[-xR_1^{a,b}+H_2^{a,b}+H_3^{a,b}\big]=0.\label{es3}
	\end{align}	
	\item[\textnormal{(ii)}] Let $q\in [1,\frac{13+\sqrt{89}}{20}]$ and assume condition {\bf(A)}. Then there exists a constant $C>0$ such that 
	\begin{align}
	&\widetilde{\E}_{t_k,x}^{a,b}\big[\big\vert R_1^{a,b}+R_2^{a,b}+R_3^{a,b}\big\vert^{2q}\big]\leq C\Delta_n^{4q}\big(\frac{1}{x^q}+\frac{1}{x^{\frac{a}{\sigma}-1}}\big),\label{es2}\\
	&\widetilde{\E}_{t_k,x}^{a,b}\big[\big\vert -xR_1^{a,b}+H_2^{a,b}+H_3^{a,b}\big\vert^{2q}\big]\leq C\Delta_n^{4q}(1+x^{2q})\big(\frac{1}{x^q}+\frac{1}{x^{\frac{a}{\sigma}-1}}\big).\label{es4}
	\end{align}	
	\end{itemize}
\end{lemma}

\section{Proof of Lemma \ref{negligibleterms}}
\label{prooflemma21}

\begin{proof}
To avoid confusion with the observed process $X^{a,b}$ generated by the Brownian motion $B$, on the probability space $(\widetilde{\Omega}, \widetilde{\mathcal{F}}, \widetilde{\P})$ we introduce an independent copy $Y^{a,b}=(Y_t^{a,b})_{t \geq 0}$ of $X^{a,b}$, which is associated to the Brownian motion $W$ independent of $B$ and which is defined by \eqref{c2eq1rajoute}. Then $(\Omega,\mathcal{F},\{\mathcal{F}_t\}_{t \geq 0},\P)$ denotes the product filtered probability space of $(\widehat{\Omega}, \widehat{\mathcal{F}},\{\widehat{\mathcal{F}}_t\}_{t \geq 0}, \widehat{\P})$ and $(\widetilde{\Omega}, \widetilde{\mathcal{F}},\{\widetilde{\mathcal{F}}_t\}_{t \geq 0}, \widetilde{\P})$, i.e., $\Omega=\widehat{\Omega}\times\widetilde{\Omega}$, $\mathcal{F}=\widehat{\mathcal{F}}\otimes\widetilde{\mathcal{F}}$, $\P=\widehat{\P}\otimes\widetilde{\P}$, $\mathcal{F}_t=\widehat{\mathcal{F}}_t\otimes\widetilde{\mathcal{F}}_t$, and $\E=\widehat{\E}\otimes\widetilde{\E}$, where $\E$, $\widehat{\E}$, $\widetilde{\E}$ denote the expectation w.r.t. $\P$, $\widehat{\P}$ and $\widetilde{\P}$.

Using the decomposition \eqref{generalexpansion} and Proposition \ref{c2prop1}, under condition $\frac{a}{\sigma}>5+3\sqrt{2}$, we obtain 
\begin{align*}
&\sum_{k=0}^{n-1}\left(R_{k,n}+H_{k,n}\right)=\sum_{k=0}^{n-1}u\varphi_{1,n}(a_0,b_0)\int_0^1\dfrac{\partial_{a}p^{a(\ell),b_0}}{p^{a(\ell),b_0}}(\Delta_n,X_{t_k}^{a_0,b_0},X_{t_{k+1}}^{a_0,b_0})d\ell\\
&\qquad+\sum_{k=0}^{n-1}v\varphi_{2,n}(a_0,b_0)\int_0^1\dfrac{\partial_{b}p^{a_n,b(\ell)}}{p^{a_n,b(\ell)}}(\Delta_n,X_{t_k}^{a_0,b_0},X_{t_{k+1}}^{a_0,b_0})d\ell-\sum_{k=0}^{n-1}\left(\xi_{k,n}+\eta_{k,n}\right)\\
&=\sum_{k=0}^{n-1}\dfrac{u\varphi_{1,n}(a_0,b_0)}{\Delta_n}\int_0^1\widetilde{\E}_{t_k,X_{t_k}^{a_0,b_0}}^{a(\ell),b_0}\big[\frac{\Delta_n}{\sqrt{2\sigma X_{t_{k}}^{a_0,b_0}}} (W_{t_{k+1}}-W_{t_{k}})+R^{a(\ell),b_0}\big\vert Y_{t_{k+1}}^{a(\ell),b_0}=X_{t_{k+1}}^{a_0,b_0}\big]d\ell\\
&\qquad+\sum_{k=0}^{n-1}\dfrac{v\varphi_{2,n}(a_0,b_0)}{\Delta_n}\int_0^1\widetilde{\E}_{t_k,X_{t_k}^{a_0,b_0}}^{a_n,b(\ell)}\big[\frac{-\Delta_n}{\sqrt{2\sigma }}\sqrt{X_{t_{k}}^{a_0,b_0}} (W_{t_{k+1}}-W_{t_{k}})\\
&\qquad+H^{a_n,b(\ell)}\vert Y_{t_{k+1}}^{a_n,b(\ell)}=X_{t_{k+1}}^{a_0,b_0}\big]d\ell-\sum_{k=0}^{n-1}\left(\xi_{k,n}+\eta_{k,n}\right),
\end{align*}
where 
\begin{align}\label{RH}
R^{a(\ell),b_0}=R_1^{a(\ell),b_0}+R_2^{a(\ell),b_0}+R_3^{a(\ell),b_0},\; H^{a_n,b(\ell)}=-X_{t_{k}}^{a_0,b_0}R_1^{a_n,b(\ell)}+H_2^{a_n,b(\ell)}+H_3^{a_n,b(\ell)}.
\end{align}
Then, from equation \eqref{c2eq1rajoute}, we write
\begin{align*}
W_{t_{k+1}}-W_{t_{k}}&=\frac{1}{\sqrt{2\sigma Y_{t_k}^{a,b}}}\big(Y_{t_{k+1}}^{a,b}-Y_{t_{k}}^{a,b}-(a-bY_{t_k}^{a,b})\Delta_n+b\int_{t_k}^{t_{k+1}}(Y_s^{a,b}-Y_{t_k}^{a,b})ds\\
&\qquad-\int_{t_k}^{t_{k+1}}(\sqrt{2\sigma Y_s^{a,b}}-\sqrt{2\sigma Y_{t_k}^{a,b}})dW_s\big).
\end{align*}
This implies that
\begin{align*}
&\sum_{k=0}^{n-1}\left(R_{k,n}+H_{k,n}\right)=\sum_{k=0}^{n-1}\dfrac{u\varphi_{1,n}(a_0,b_0)}{\Delta_n}\int_0^1\widetilde{\E}_{t_k,X_{t_k}^{a_0,b_0}}^{a(\ell),b_0}\Big[\frac{\Delta_n}{\sqrt{2\sigma X_{t_{k}}^{a_0,b_0}}} \frac{1}{\sqrt{2\sigma Y_{t_k}^{a(\ell),b_0}}}\\
&\qquad\times\big(Y_{t_{k+1}}^{a(\ell),b_0}-Y_{t_{k}}^{a(\ell),b_0}-(a(\ell)-b_0Y_{t_k}^{a(\ell),b_0})\Delta_n+b_0\int_{t_k}^{t_{k+1}}(Y_s^{a(\ell),b_0}-Y_{t_k}^{a(\ell),b_0})ds\\
&\qquad-\int_{t_k}^{t_{k+1}}(\sqrt{2\sigma Y_s^{a(\ell),b_0}}-\sqrt{2\sigma Y_{t_k}^{a(\ell),b_0}})dW_s\big)+R^{a(\ell),b_0}\big\vert Y_{t_{k+1}}^{a(\ell),b_0}=X_{t_{k+1}}^{a_0,b_0}\Big]d\ell\\
&\qquad+\sum_{k=0}^{n-1}\dfrac{v\varphi_{2,n}(a_0,b_0)}{\Delta_n}\int_0^1\widetilde{\E}_{t_k,X_{t_k}^{a_0,b_0}}^{a_n,b(\ell)}\Big[-\frac{\Delta_n}{\sqrt{2\sigma }}\sqrt{X_{t_{k}}^{a_0,b_0}} \frac{1}{\sqrt{2\sigma Y_{t_k}^{a_n,b(\ell)}}}\big(Y_{t_{k+1}}^{a_n,b(\ell)}-Y_{t_{k}}^{a_n,b(\ell)}\\
&\qquad-(a_n-b(\ell)Y_{t_k}^{a_n,b(\ell)})\Delta_n+b(\ell)\int_{t_k}^{t_{k+1}}(Y_s^{a_n,b(\ell)}-Y_{t_k}^{a_n,b(\ell)})ds-\int_{t_k}^{t_{k+1}}(\sqrt{2\sigma Y_s^{a_n,b(\ell)}}\\
&\qquad-\sqrt{2\sigma Y_{t_k}^{a_n,b(\ell)}})dW_s\big)+H^{a_n,b(\ell)}\vert Y_{t_{k+1}}^{a_n,b(\ell)}=X_{t_{k+1}}^{a_0,b_0}\Big]d\ell-\sum_{k=0}^{n-1}\left(\xi_{k,n}+\eta_{k,n}\right)\\
&=\sum_{k=0}^{n-1}\dfrac{u\varphi_{1,n}(a_0,b_0)}{\Delta_n}\int_0^1\Big\{\frac{\Delta_n}{2\sigma X_{t_{k}}^{a_0,b_0}}\big(X_{t_{k+1}}^{a_0,b_0}-X_{t_{k}}^{a_0,b_0}-(a(\ell)-b_0X_{t_{k}}^{a_0,b_0})\Delta_n\big)\\
&\qquad+\widetilde{\E}_{t_k,X_{t_k}^{a_0,b_0}}^{a(\ell),b_0}\Big[-R_4^{a(\ell),b_0}-R_5^{a(\ell),b_0}+R^{a(\ell),b_0}\big\vert Y_{t_{k+1}}^{a(\ell),b_0}=X_{t_{k+1}}^{a_0,b_0}\Big]\Big\}d\ell+\sum_{k=0}^{n-1}\dfrac{v\varphi_{2,n}(a_0,b_0)}{\Delta_n}\\
&\qquad\times\int_0^1\Big\{-\frac{\Delta_n}{2\sigma }\big(X_{t_{k+1}}^{a_0,b_0}-X_{t_{k}}^{a_0,b_0}-(a_n-b(\ell)X_{t_{k}}^{a_0,b_0})\Delta_n\big)+\widetilde{\E}_{t_k,X_{t_k}^{a_0,b_0}}^{a_n,b(\ell)}\big[X_{t_{k}}^{a_0,b_0}R_4^{a_n,b(\ell)}\\
&\qquad+X_{t_{k}}^{a_0,b_0}R_5^{a_n,b(\ell)}+H^{a_n,b(\ell)}\vert Y_{t_{k+1}}^{a_n,b(\ell)}=X_{t_{k+1}}^{a_0,b_0}\big]\Big\}d\ell-\sum_{k=0}^{n-1}(\xi_{k,n}+\eta_{k,n}),
\end{align*}
where 
\begin{align*}
&R_4^{a,b}=-\frac{\Delta_n}{2\sigma \sqrt{X_{t_{k}}^{a_0,b_0}Y_{t_k}^{a,b}}}b\int_{t_k}^{t_{k+1}}(Y_s^{a,b}-Y_{t_k}^{a,b})ds,\\
&R_5^{a,b}=\frac{\Delta_n}{\sqrt{2\sigma X_{t_{k}}^{a_0,b_0}Y_{t_k}^{a,b}}}\int_{t_k}^{t_{k+1}}(\sqrt{Y_s^{a,b}}-\sqrt{Y_{t_k}^{a,b}})dW_s.
\end{align*}
Next, using equation \eqref{c2eq1}, we write
\begin{align*}
&X_{t_{k+1}}^{a_0,b_0}-X_{t_{k}}^{a_0,b_0}=\sqrt{2\sigma X_{t_k}^{a_0,b_0}}(B_{t_{k+1}}-B_{t_{k}})+(a_0-b_0X_{t_{k}}^{a_0,b_0})\Delta_n\\
&\qquad-b_0\int_{t_k}^{t_{k+1}} (X_s^{a,b}-X_{t_k}^{a_0,b_0})ds +\int_{t_k}^{t_{k+1}} (\sqrt{2\sigma X_s^{a_0,b_0}}-\sqrt{2\sigma X_{t_k}^{a_0,b_0}})dB_s.
\end{align*}
Consequently, using the definition of $\xi_{k,n}$, $\eta_{k,n}$, we get
\begin{align*}
&\sum_{k=0}^{n-1}\left(R_{k,n}+H_{k,n}\right)=\sum_{k=0}^{n-1}\dfrac{u\varphi_{1,n}(a_0,b_0)}{\Delta_n}\int_0^1\Big\{ R^{a_0,b_0}_{6}+R^{a_0,b_0}_{7}+\widetilde{\E}_{t_k,X_{t_{k}}^{a_0,b_0}}^{a(\ell),b_0}\big[-R_4^{a(\ell),b_0}-R_5^{a(\ell),b_0}\\
&\qquad+R^{a(\ell),b_0}\big\vert Y_{t_{k+1}}^{a(\ell),b_0}=X_{t_{k+1}}^{a_0,b_0}\big]\Big\}d\ell+\sum_{k=0}^{n-1}\dfrac{v\varphi_{2,n}(a_0,b_0)}{\Delta_n}\int_0^1\Big\{-X_{t_{k}}^{a_0,b_0}R_6^{a_0,b_0}-X_{t_{k}}^{a_0,b_0}R_7^{a_0,b_0}\\
&\qquad+\widetilde{\E}_{t_k,X_{t_{k}}^{a_0,b_0}}^{a_n,b(\ell)}\left[X_{t_{k}}^{a_0,b_0}R_4^{a_n,b(\ell)}+X_{t_{k}}^{a_0,b_0}R_5^{a_n,b(\ell)}+H^{a_n,b(\ell)}\big\vert Y_{t_{k+1}}^{a_n,b(\ell)}=X_{t_{k+1}}^{a_0,b_0}\right]\Big\}d\ell,
\end{align*}
where 
\begin{align*}
&R_6^{a_0,b_0}=-\frac{\Delta_n}{2\sigma X_{t_k}^{a_0,b_0}}b_0\int_{t_k}^{t_{k+1}}(X_s^{a_0,b_0}-X_{t_k}^{a_0,b_0})ds,\\
&R_7^{a_0,b_0}=\frac{\Delta_n}{\sqrt{2\sigma} X_{t_k}^{a_0,b_0}}\int_{t_k}^{t_{k+1}}(\sqrt{X_s^{a_0,b_0}}-\sqrt{X_{t_k}^{a_0,b_0}})dB_s.
\end{align*}
Remark that in the supercritical case with $\varphi_{1,n}(a_0,b_0)=0$, $\sum_{k=0}^{n-1}(R_{k,n}+H_{k,n})$ contains only the second sum and $a_n=a_0$. This decomposition, together with Lemma \ref{lemma1}-\ref{lemma5} below, finishes the proof of Lemma \ref{negligibleterms}.
\end{proof}
Now, it remains to prove Lemma \ref{lemma1}-\ref{lemma5} below. To do so, we need the following preliminary part. Let $\widehat{\P}_{t_k,x}^{a,b}$ and $\widehat{\E}_{t_k,x}^{a,b}$ denote respectively the probability law of $X^{a,b}$ starting at $x$ at time $t_k$ and the expectation w.r.t. $\widehat{\P}_{t_k,x}^{a,b}$, i.e., $\widehat{\P}_{t_k,x}^{a,b}(A)=\widehat{\E}[{\bf 1}_{A}\vert X_{t_{k}}^{a,b}=x]$ for all $A\in \widehat{\mathcal{F}}$ and $\widehat{\E}_{t_k,x}^{a,b}[V]=\widehat{\E}[V\vert X_{t_{k}}^{a,b}=x]$ for all $\widehat{\mathcal{F}}$-measurable random variables $V$. We introduce change of measures on $I_k:=[t_k, t_{k+1}]$. For all $a, a_1\in\R^{\ast}_+$, $b, b_1\in\R$, $x\in\R^{\ast}_+$ and $k\in\{0,...,n-1\}$, by \cite[Lemma 3.1]{BP16}, for any $a>\sigma$ the probability measures $\widehat{\P}_{t_k,x}^{a_1,b_1}$ and $\widehat{\P}_{t_k,x}^{a,b}$ are absolutely continuous w.r.t. each other and its Radon-Nikodym derivative is given by
\begin{align}
&\dfrac{d\widehat{\P}_{t_k,x}^{a_1,b_1}}{d\widehat{\P}_{t_k,x}^{a,b}}\big((X_t^{a,b})_{t\in I_k}\big)\label{ratio2}\\
&=\exp\Big\{\int_{t_k}^{t_{k+1}}\frac{a_1-a-(b_1-b)X_s^{a,b}}{2\sigma X_s^{a,b}}dX_s^{a,b}-\frac{1}{2}\int_{t_k}^{t_{k+1}}\frac{(a_1-b_1X_s^{a,b})^2-(a-bX_s^{a,b})^2}{2\sigma X_s^{a,b}}ds\Big\}.\notag
\end{align}
Similarly,
\begin{align}
&\dfrac{d\widetilde{\P}_{t_k,x}^{a_1,b_1}}{d\widetilde{\P}_{t_k,x}^{a,b}}\big((Y_t^{a,b})_{t\in I_k}\big)\label{ratio3}\\
&=\exp\Big\{\int_{t_k}^{t_{k+1}}\frac{a_1-a-(b_1-b)Y_s^{a,b}}{2\sigma Y_s^{a,b}}dY_s^{a,b}-\frac{1}{2}\int_{t_k}^{t_{k+1}}\frac{(a_1-b_1Y_s^{a,b})^2-(a-bY_s^{a,b})^2}{2\sigma Y_s^{a,b}}ds\Big\}.\notag
\end{align}
To prove Lemma \ref{negligibleterms}, we need the following useful lemma related to the change of measures.
\begin{lemma}\label{change} Let $V$ denote a $\widetilde{\mathcal{F}}_{t_{k+1}}$-measurable random variable which can be $R^{a,b}$, $(R^{a,b})^2$, $H^{a,b}$, $(H^{a,b})^2$, $R_4^{a,b}$, $(R_4^{a,b})^2$, $R_5^{a,b}$, $(R_5^{a,b})^2$. Assume condition $a>\sigma$. Then, for any $k\in\{0,...,n-1\}$ and $x\in\R^{\ast}_+$,
	\begin{equation}\label{for1}\begin{split}
	&\widehat{\E}_{t_k,x}^{a_0,b_0}\Big[\widetilde{\E}_{t_k,x}^{a,b}\big[V\vert Y_{t_{k+1}}^{a,b}=
	X_{t_{k+1}}^{a_0,b_0}\big]\Big]=\widetilde{\E}_{t_k,x}^{a,b}\left[V\right]+\widetilde{\E}_{t_k,x}^{a,b}\Big[\widetilde{\E}_{t_k,x}^{a,b}\big[V\vert Y_{t_{k+1}}^{a,b}\big]\big(\frac{d\widetilde{\P}_{t_k,x}^{a_0,b_0}}{d\widetilde{\P}_{t_k,x}^{a,b}}((Y_t^{a,b})_{t\in I_k})-1\big)\Big].
	\end{split}
	\end{equation}	
	Similarly, let $\widehat{V}$ be a $\widehat{\mathcal{F}}_{t_{k+1}}$-measurable random variable. Then, we have
	\begin{equation}\label{for2}\begin{split}
	\widehat{\E}_{t_k,x}^{a_0,b_0}[\widehat{V}]=\widehat{\E}_{t_k,x}^{a,b}[\widehat{V}]+\widehat{\E}_{t_k,x}^{a,b}\Big[\widehat{V}\big(\frac{d\widehat{\P}_{t_k,x}^{a_0,b_0}}{d\widehat{\P}_{t_k,x}^{a,b}}((X_t^{a,b})_{t\in I_k})-1\big)\Big].
	\end{split}
	\end{equation}	
\end{lemma}
\begin{proof}
	For simplicity, we set $g(x,y):=\widetilde{\E}_{t_k,x}^{a,b}[V\vert Y_{t_{k+1}}^{a,b}=y]$ for all $y\in\R^{\ast}_+$. Applying change of measures, we have that 
	\begin{align*}
	&\widehat{\E}_{t_k,x}^{a_0,b_0}\Big[\widetilde{\E}_{t_k,x}^{a,b}\big[V\vert Y_{t_{k+1}}^{a,b}=
	X_{t_{k+1}}^{a_0,b_0}\big]\Big]=\widehat{\E}_{t_k,x}^{a_0,b_0}[g(x,X_{t_{k+1}}^{a_0,b_0})]=\widehat{\E}_{t_k,x}^{a,b}\Big[g(x,X_{t_{k+1}}^{a,b})\dfrac{d\widehat{\P}_{t_k,x}^{a_0,b_0}}{d\widehat{\P}_{t_k,x}^{a,b}}((X_t^{a,b})_{t\in I_k})\Big]\\
	&=\widetilde{\E}_{t_k,x}^{a,b}\Big[g(x,Y_{t_{k+1}}^{a,b})\dfrac{d\widetilde{\P}_{t_k,x}^{a_0,b_0}}{d\widetilde{\P}_{t_k,x}^{a,b}}((Y_t^{a,b})_{t\in I_k})\Big]\\
	&=\widetilde{\E}_{t_k,x}^{a,b}[g(x,Y_{t_{k+1}}^{a,b})]+\widetilde{\E}_{t_k,x}^{a,b}\Big[g(x,Y_{t_{k+1}}^{a,b})\big(\dfrac{d\widetilde{\P}_{t_k,x}^{a_0,b_0}}{d\widetilde{\P}_{t_k,x}^{a,b}}((Y_t^{a,b})_{t\in I_k})-1\big)\Big]\\
	&=\widetilde{\E}_{t_k,x}^{a,b}\Big[\widetilde{\E}_{t_k,x}^{a,b}\big[V\vert Y_{t_{k+1}}^{a,b}\big]\Big]+\widetilde{\E}_{t_k,x}^{a,b}\Big[\widetilde{\E}_{t_k,x}^{a,b}\big[V\vert Y_{t_{k+1}}^{a,b}\big]\big(\dfrac{d\widetilde{\P}_{t_k,x}^{a_0,b_0}}{d\widetilde{\P}_{t_k,x}^{a,b}}((Y_t^{a,b})_{t\in I_k})-1\big)\Big]\\
	&=\widetilde{\E}_{t_k,x}^{a,b}[V]+\widetilde{\E}_{t_k,x}^{a,b}\Big[\widetilde{\E}_{t_k,x}^{a,b}\big[V\vert Y_{t_{k+1}}^{a,b}\big]\big(\dfrac{d\widetilde{\P}_{t_k,x}^{a_0,b_0}}{d\widetilde{\P}_{t_k,x}^{a,b}}((Y_t^{a,b})_{t\in I_k})-1\big)\Big],
	\end{align*}
	where we have used the fact that $Y^{a,b}$ is the independent copy of $X^{a,b}$. Thus, \eqref{for1} follows. The equality \eqref{for2} is proven using the same arguments.
\end{proof}	
Next, the following lemma gives the estimate for the second term in Lemma \ref{change}.
\begin{lemma}\label{deviation1} Let $p, q>1$ satisfy that $\frac{1}{p}+\frac{1}{q}=1$, $q_1$ be large enough, $(a,b)=(a(\ell),b_0)$ or $(a,b)=(a_n,b(\ell))$, and $V$, $\widehat{V}$ be the same notation of the previous lemma. Assume that $p<\frac{a}{\sigma}-1$. Then for any $k\in\{0,...,n-1\}$, $x\in\R^{\ast}_+$ and for $n$ large enough, there exist constants $C_1, C_2>0$ which does not depend on $x$ such that 
	\begin{equation}\label{for3}\begin{split}
	&\Big\vert\widetilde{\E}_{t_k,x}^{a,b}\Big[\widetilde{\E}_{t_k,x}^{a,b}\big[V\vert Y_{t_{k+1}}^{a,b}\big]\big(\dfrac{d\widetilde{\P}_{t_k,x}^{a_0,b_0}}{d\widetilde{\P}_{t_k,x}^{a,b}}((Y_t^{a,b})_{t\in I_k})-1\big)\Big]\Big\vert\leq C_1\sqrt{\Delta_n}\big(\widetilde{\E}_{t_k,x}^{a,b}[\vert V\vert^q]\big)^{\frac{1}{q}}e^{C_2(b_0-b)^2\Delta_nx} \\
	&\qquad\times\Big(\vert a-a_0\vert\big(\frac{1}{\sqrt{x}}+\frac{\sqrt{\Delta_n}}{x}\vert a-a_0\vert+\sqrt{\Delta_n}\vert b-b_0\vert\big)   \big(1+\dfrac{1}{x^{\frac{1}{8pq_1}(\frac{a}{\sigma}-1)}}\big)\\
	&\qquad+\vert b-b_0\vert\big(1+\sqrt{x}+\sqrt{\Delta_n}\vert b-b_0\vert x\big) \Big).
	\end{split}
	\end{equation}	
	Moreover, assume that $p<2(\frac{a}{\sigma}-1)$. Then, we have
	\begin{equation}\label{for4}\begin{split}
	&\Big\vert\widehat{\E}_{t_k,x}^{a,b}\Big[\widehat{V}\big(\frac{d\widehat{\P}_{t_k,x}^{a_0,b_0}}{d\widehat{\P}_{t_k,x}^{a,b}}((X_t^{a,b})_{t\in I_k})-1\big)\Big]\Big\vert\leq C_1\dfrac{\sqrt{\Delta_n}}{\sqrt{x}}\big\vert\int_{a}^{a_0}\big(\widehat{\E}_{t_k,x}^{\alpha,b_0}[\vert \widehat{V}\vert^q]\big)^{\frac{1}{q}}d\alpha\big\vert\\
	&\qquad+C_2\sqrt{\Delta_n}(1+\sqrt{x})\big\vert\int_{b}^{b_0}\big(\widehat{\E}_{t_k,x}^{a,\beta}[\vert \widehat{V}\vert^q]\big)^{\frac{1}{q}}d\beta\big\vert.
	\end{split}
	\end{equation}
	\end{lemma}
\begin{proof} 	
	Using \eqref{ratio3}, we have the following decomposition
	\begin{align}\label{decom}
	&\dfrac{d\widetilde{\P}_{t_k,x}^{a_0,b_0}}{d\widetilde{\P}_{t_k,x}^{a,b}}((Y_t^{a,b})_{t\in I_k})-1=\int_{a}^{a_0}\dfrac{\partial}{\partial \alpha}\big(\dfrac{d\widetilde{\P}_{t_k,x}^{\alpha,b_0}}{d\widetilde{\P}_{t_k,x}^{a,b}}\big)((Y_t^{a,b})_{t\in I_k})d\alpha+\int_{b}^{b_0}\dfrac{\partial}{\partial \beta}\big(\dfrac{d\widetilde{\P}_{t_k,x}^{a,\beta}}{d\widetilde{\P}_{t_k,x}^{a,b}}\big)((Y_t^{a,b})_{t\in I_k})d\beta\notag\\
	&=\int_{a}^{a_0}\int_{t_k}^{t_{k+1}}\dfrac{1}{2\sigma Y_s^{a,b}}\big(dY_s^{a,b}-(\alpha-b_0Y_s^{a,b})ds\big)\frac{d\widetilde{\P}_{t_k,x}^{\alpha,b_0}}{d\widetilde{\P}_{t_k,x}^{a,b}}((Y_t^{a,b})_{t\in I_k})d\alpha\notag\\
	&\qquad-\int_{b}^{b_0}\int_{t_k}^{t_{k+1}}\dfrac{1}{2\sigma}\big(dY_s^{a,b}-(a-\beta Y_s^{a,b})ds\big)\dfrac{d\widetilde{\P}_{t_k,x}^{a,\beta}}{d\widetilde{\P}_{t_k,x}^{a,b}}((Y_t^{a,b})_{t\in I_k})d\beta.
	\end{align}	
	Applying equation \eqref{c2eq1rajoute}, H\"older's inequality with $\frac{1}{p}+\frac{1}{q}=1$ and $\frac{1}{p_1}+\frac{1}{q_1}=1$, BDG's and Jensen's inequalities, \eqref{e1} and \eqref{e2} for $Y^{a,b}$, we get that
	\begin{align}
	&\Big\vert\widetilde{\E}_{t_k,x}^{a,b}\Big[\widetilde{\E}_{t_k,x}^{a,b}\big[V\vert Y_{t_{k+1}}^{a,b}\big]\big(\dfrac{d\widetilde{\P}_{t_k,x}^{a_0,b_0}}{d\widetilde{\P}_{t_k,x}^{a,b}}((Y_t^{a,b})_{t\in I_k})-1\big)\Big]\Big\vert\notag\\	
	&=\Big\vert\int_{a}^{a_0}\widetilde{\E}_{t_k,x}^{a,b}\Big[\widetilde{\E}_{t_k,x}^{a,b}\big[V\vert Y_{t_{k+1}}^{a,b}\big]\int_{t_k}^{t_{k+1}}\dfrac{1}{2\sigma Y_s^{a,b}}\big(\sqrt{2\sigma Y_s^{a,b}}dW_s+(a-\alpha-(b-b_0)Y_s^{a,b})ds\big)\notag\\
	&\qquad\times\frac{d\widetilde{\P}_{t_k,x}^{\alpha,b_0}}{d\widetilde{\P}_{t_k,x}^{a,b}}((Y_t^{a,b})_{t\in I_k})\Big]d\alpha-\int_{b}^{b_0}\widetilde{\E}_{t_k,x}^{a,b}\Big[\widetilde{\E}_{t_k,x}^{a,b}\big[V\vert Y_{t_{k+1}}^{a,b}\big]\int_{t_k}^{t_{k+1}}\dfrac{1}{2\sigma}\big(\sqrt{2\sigma Y_s^{a,b}}dW_s\notag\\
	&\qquad-(b-\beta)Y_s^{a,b})ds\big)\dfrac{d\widetilde{\P}_{t_k,x}^{a,\beta}}{d\widetilde{\P}_{t_k,x}^{a,b}}((Y_t^{a,b})_{t\in I_k})\Big]d\beta\Big\vert\leq \Big\vert\int_{a}^{a_0} \big(\widetilde{\E}_{t_k,x}^{a,b}[\widetilde{\E}_{t_k,x}^{a,b}[\vert V\vert^q\vert Y_{t_{k+1}}^{a,b}]]\big)^{\frac{1}{q}}\notag\\
	&\qquad\times  \Big(\widetilde{\E}_{t_k,x}^{a,b}\Big[\big\vert \int_{t_k}^{t_{k+1}}\dfrac{1}{2\sigma Y_s^{a,b}}\big(\sqrt{2\sigma Y_s^{a,b}}dW_s+(a-\alpha-(b-b_0)Y_s^{a,b})ds\big)\big\vert^{pp_1}\Big]\Big)^{\frac{1}{pp_1}} \notag\\
	&\qquad \times \Big(\widetilde{\E}_{t_k,x}^{a,b}\Big[\big( \frac{d\widetilde{\P}_{t_k,x}^{\alpha,b_0}}{d\widetilde{\P}_{t_k,x}^{a,b}}((Y_t^{a,b})_{t\in I_k})\big)^{pq_1}\Big]\Big)^{\frac{1}{pq_1}} d\alpha\Big\vert+\Big\vert\int_{b}^{b_0} \big(\widetilde{\E}_{t_k,x}^{a,b}[\widetilde{\E}_{t_k,x}^{a,b}[\vert V\vert^{q}\vert Y_{t_{k+1}}^{a,b}]]\big)^{\frac{1}{q}}\notag\\
	&\qquad\times  \Big(\widetilde{\E}_{t_k,x}^{a,b}\Big[\big\vert \int_{t_k}^{t_{k+1}}\dfrac{1}{2\sigma}\big(\sqrt{2\sigma Y_s^{a,b}}dW_s-(b-\beta)Y_s^{a,b})ds\big)\big\vert^{pp_1}\Big]\Big)^{\frac{1}{pp_1}}\notag \\
	&\qquad \times \Big(\widetilde{\E}_{t_k,x}^{a,b}\Big[\big( \frac{d\widetilde{\P}_{t_k,x}^{a,\beta}}{d\widetilde{\P}_{t_k,x}^{a,b}}((Y_t^{a,b})_{t\in I_k})\big)^{pq_1}\Big]\Big)^{\frac{1}{pq_1}} d\beta\Big\vert\leq C_1\Big\vert\int_{a}^{a_0} (\widetilde{\E}_{t_k,x}^{a,b}[\vert V\vert^q])^{\frac{1}{q}} \sqrt{\Delta_n}\big(\frac{1}{\sqrt{x}}\notag\\
	&\qquad+\frac{\sqrt{\Delta_n}}{x}\vert a-a_0\vert+\sqrt{\Delta_n}\vert b-b_0\vert \big)   \Big(\widetilde{\E}_{t_k,x}^{a,b}\Big[\big( \frac{d\widetilde{\P}_{t_k,x}^{\alpha,b_0}}{d\widetilde{\P}_{t_k,x}^{a,b}}((Y_t^{a,b})_{t\in I_k})\big)^{pq_1}\Big]\Big)^{\frac{1}{pq_1}} d\alpha\Big\vert\notag\\
	&\qquad+C_2\Big\vert\int_{b}^{b_0} (\widetilde{\E}_{t_k,x}^{a,b}[\vert V\vert^{q}])^{\frac{1}{q}} \sqrt{\Delta_n}\big(1+\sqrt{x}+\sqrt{\Delta_n}\vert b-b_0\vert x\big)\notag\\
	&\qquad\times  \Big(\widetilde{\E}_{t_k,x}^{a,b}\Big[\big( \frac{d\widetilde{\P}_{t_k,x}^{a,\beta}}{d\widetilde{\P}_{t_k,x}^{a,b}}((Y_t^{a,b})_{t\in I_k})\big)^{pq_1}\Big]\Big)^{\frac{1}{pq_1}} d\beta\Big\vert,\label{a1}
	\end{align}
	for some constants $C_1, C_2>0$, where $p_1$ is close to $1$ and $p<\frac{a}{\sigma}-1$. 
	
	Then, using Cauchy-Schwarz inequality twice and Lemma \ref{expmoment}, provided that we have an exponential martingale, for $n$ large enough we get
	\begin{align*}
	&\widetilde{\E}_{t_k,x}^{a,b}\Big[\big( \frac{d\widetilde{\P}_{t_k,x}^{\alpha,b_0}}{d\widetilde{\P}_{t_k,x}^{a,b}}((Y_t^{a,b})_{t\in I_k})\big)^{pq_1}\Big]\\
	&\leq \big(\widetilde{\E}_{t_k,x}^{a,b}\big[\exp\big\{pq_1(2pq_1-1)\frac{1}{\sigma}\big((\alpha-a)^2\int_{t_k}^{t_{k+1}}\frac{ds}{Y_s^{a,b}}+(b_0-b)^2\int_{t_k}^{t_{k+1}}Y_s^{a,b}ds\big)\big\}\big]\big)^{\frac{1}{2}}\\
	&\leq C\big(1+\dfrac{1}{x^{\frac{1}{8}(\frac{a}{\sigma}-1)}}\big)\big(\widetilde{\E}_{t_k,x}^{a,b}\big[\exp\big\{2pq_1(2pq_1-1)\frac{1}{\sigma}(b_0-b)^2\int_{t_k}^{t_{k+1}}Y_s^{a,b}ds\big\}\big]\big)^{\frac{1}{4}},
	\end{align*}
	where we have used the fact that  	$2pq_1(2pq_1-1)\frac{1}{\sigma}(\alpha-a)^2\leq 2pq_1(2pq_1-1)\frac{1}{\sigma}(a_0-a)^2\leq 2pq_1(2pq_1-1)\frac{1}{\sigma}u^2\varphi_{1,n}^2(a_0,b_0)\leq(\frac{a}{\sigma}-1)^2\frac{\sigma}{4}$ for $n$ large enough since $\varphi_{1,n}(a_0,b_0)\to0$.
				
	Next, using Jensen's inequality and $(1.9)$ of \cite[Proposition 1.2.4]{A15}, for $n$ large enough we get	
	\begin{align}
	&\widetilde{\E}_{t_k,x}^{a,b}\Big[\big( \frac{d\widetilde{\P}_{t_k,x}^{\alpha,b_0}}{d\widetilde{\P}_{t_k,x}^{a,b}}((Y_t^{a,b})_{t\in I_k})\big)^{pq_1}\Big]\leq C\big(1+\dfrac{1}{x^{\frac{1}{8}(\frac{a}{\sigma}-1)}}\big)\notag\\
	&\times\big(\frac{1}{\Delta_n}\int_{t_k}^{t_{k+1}}\widetilde{\E}_{t_k,x}^{a,b}\big[\exp\big\{2pq_1(2pq_1-1)\frac{1}{\sigma}(b_0-b)^2\Delta_nY_s^{a,b}\big\}\big]ds\big)^{\frac{1}{4}}\notag\\
	&\leq C\big(1+\dfrac{1}{x^{\frac{1}{8}(\frac{a}{\sigma}-1)}}\big)\big(\sup_{s\in [t_k,t_{k+1}]}\widetilde{\E}_{t_k,x}^{a,b}\big[\exp\big\{2pq_1(2pq_1-1)\frac{1}{\sigma}(b_0-b)^2\Delta_nY_s^{a,b}\big\}\big]\big)^{\frac{1}{4}}\notag\\
	&=C\big(1+\dfrac{1}{x^{\frac{1}{8}(\frac{a}{\sigma}-1)}}\big)\big(\sup_{s\in [0,\Delta_n]}\widetilde{\E}_{0,x}^{a,b}\big[\exp\big\{2pq_1(2pq_1-1)\frac{1}{\sigma}(b_0-b)^2\Delta_nY_s^{a,b}\big\}\big]\big)^{\frac{1}{4}}\notag \\
	&\leq C\big(1+\dfrac{1}{x^{\frac{1}{8}(\frac{a}{\sigma}-1)}}\big)e^{c(b_0-b)^2\Delta_nx},\label{a4}
	\end{align}
		for some constants $C, c>0$. The same arguments can be used to check Novikov's condition and deduce the validity of the exponential martingale property used above.		
			
		Proceeding as for the term $\frac{d\widetilde{\P}_{t_k,x}^{\alpha,b_0}}{d\widetilde{\P}_{t_k,x}^{a,b}}((Y_t^{a,b})_{t\in I_k})$ above, for $n$ large enough we obtain
	    \begin{align}
	    \widetilde{\E}_{t_k,x}^{a,b}\Big[\big( \frac{d\widetilde{\P}_{t_k,x}^{a,\beta}}{d\widetilde{\P}_{t_k,x}^{a,b}}((Y_t^{a,b})_{t\in I_k})\big)^{pq_1}\Big]\leq Ce^{c(b_0-b)^2\Delta_nx},\label{a6}
	    \end{align}
	    for some constants $C, c>0$. Thus, \eqref{for3} follows from \eqref{a1}, \eqref{a4} and \eqref{a6}.		
		
		Next, using \eqref{decom} for $X^{a,b}$, change of measures, equation \eqref{c2eq1}, H\"older's, Jensen's and BDG's inequalities, \eqref{e1} and \eqref{e2} for $X^{a,b}$, we get that
		\begin{align*}
		&\Big\vert\widehat{\E}_{t_k,x}^{a,b}\Big[\widehat{V}\big(\frac{d\widehat{\P}_{t_k,x}^{a_0,b_0}}{d\widehat{\P}_{t_k,x}^{a,b}}((X_t^{a,b})_{t\in I_k})-1\big)\Big]\Big\vert=\Big\vert\int_{a}^{a_0}\widehat{\E}_{t_k,x}^{\alpha,b_0}\big[\widehat{V}\int_{t_k}^{t_{k+1}}\dfrac{dB_s}{\sqrt{2\sigma X_s^{\alpha,b_0}}}\big]d\alpha\\	
		&\qquad-\int_{b}^{b_0}\widehat{\E}_{t_k,x}^{a,\beta}\big[\widehat{V}\int_{t_k}^{t_{k+1}}\dfrac{\sqrt{X_s^{a,\beta}}}{\sqrt{2\sigma}}dB_s\big]d\beta\Big\vert\\
		&\leq C_1\big\vert\int_{a}^{a_0}\big(\widehat{\E}_{t_k,x}^{\alpha,b_0}[\vert \widehat{V}\vert^q]\big)^{\frac{1}{q}}\big(\Delta_n^{\frac{p}{2}-1}\int_{t_k}^{t_{k+1}}\widehat{\E}_{t_k,x}^{\alpha,b_0}[\dfrac{1}{\vert X_s^{\alpha,b_0}\vert^{\frac{p}{2}}}]ds\big)^{\frac{1}{p}}d\alpha\big\vert\\
		&\qquad+C_2\big\vert\int_{b}^{b_0}\big(\widehat{\E}_{t_k,x}^{a,\beta}[ \vert \widehat{V}\vert^q]\big)^{\frac{1}{q}}\big(\Delta_n^{\frac{p}{2}-1}\int_{t_k}^{t_{k+1}}\widehat{\E}_{t_k,x}^{a,\beta}[\vert X_s^{a,\beta}\vert^{\frac{p}{2}}]ds\big)^{\frac{1}{p}}d\beta\big\vert\\
		&\leq C_1\dfrac{\sqrt{\Delta_n}}{\sqrt{x}}\big\vert\int_{a}^{a_0}\big(\widehat{\E}_{t_k,x}^{\alpha,b_0}[\vert \widehat{V}\vert^q]\big)^{\frac{1}{q}}d\alpha\big\vert+C_2\sqrt{\Delta_n}\left(1+\sqrt{x}\right)\big\vert\int_{b}^{b_0}\big(\widehat{\E}_{t_k,x}^{a,\beta}[\vert \widehat{V}\vert^q]\big)^{\frac{1}{q}}d\beta\big\vert,
		\end{align*}
		for some constants $C_1, C_2>0$, where $p,q>1$ and $\frac{1}{p}+\frac{1}{q}=1$ with $p<2(\frac{a}{\sigma}-1)$. Thus, \eqref{for4} follows.	
	\end{proof}

\begin{lemma}\label{lemma1} Assume condition {\bf(A)}. Then, for both subcritical and critical cases as $n\to\infty$, we have
	\begin{align*}
	\sum_{k=0}^{n-1}\dfrac{u\varphi_{1,n}(a_0,b_0)}{\Delta_n}\int_0^1\widetilde{\E}_{t_k,X_{t_{k}}^{a_0,b_0}}^{a(\ell),b_0}[R^{a(\ell),b_0}\vert Y_{t_{k+1}}^{a(\ell),b_0}=X_{t_{k+1}}^{a_0,b_0}]d\ell\overset{\widehat{\P}^{a_0,b_0}}{\longrightarrow}0,
	\end{align*}
	where $R^{a(\ell),b_0}$ is given by \eqref{RH}.
\end{lemma}
\begin{proof} 
	It suffices to show that conditions (i) and (ii) of Lemma \ref{zero} satisfy under the measure $\widehat{\P}^{a_0,b_0}$ applied to the random variable
	$$
	\zeta_{k,n}:=\dfrac{u\varphi_{1,n}(a_0,b_0)}{\Delta_n}\int_0^1\widetilde{\E}_{t_k,X_{t_{k}}^{a_0,b_0}}^{a(\ell),b_0}[R^{a(\ell),b_0}\vert Y_{t_{k+1}}^{a(\ell),b_0}=X_{t_{k+1}}^{a_0,b_0}]d\ell.
	$$ 
	Now, we start showing condition (i) of Lemma \ref{zero}. Using \eqref{for1} of Lemma \ref{change} with $V=R^{a(\ell),b_0}$, the fact that  $\widetilde{\E}_{t_k,X_{t_k}^{a_0,b_0}}^{a(\ell),b_0}[R^{a(\ell),b_0}]=0$ by \eqref{es1}, \eqref{for3} of Lemma \ref{deviation1} with $q=p=2$ and $q_1$ large enough, \eqref{es2} of Lemma \ref{estimate} with $q=1$, and condition {\bf(A)}, we obtain for $n$ large enough,	
	\begin{align*} 
	&\vert\sum_{k=0}^{n-1}\widehat{\E}^{a_0,b_0}[\zeta_{k,n}\vert \widehat{\mathcal{F}}_{t_k}]\vert\leq C\sum_{k=0}^{n-1}\dfrac{\vert u\vert\varphi_{1,n}(a_0,b_0)}{\Delta_n}\sqrt{\Delta_n}\int_0^1\big(\widetilde{\E}_{t_k,X_{t_k}^{a_0,b_0}}^{a(\ell),b_0}[\vert R^{a(\ell),b_0}\vert^2]\big)^{\frac{1}{2}}\\
	&\qquad\times \vert u\vert\varphi_{1,n}(a_0,b_0)\big(\frac{1}{\sqrt{X_{t_k}^{a_0,b_0}}}+\frac{\sqrt{\Delta_n}}{X_{t_k}^{a_0,b_0}}\vert u\vert\varphi_{1,n}(a_0,b_0)\big)   \big(1+\dfrac{1}{(X_{t_k}^{a_0,b_0})^{\frac{1}{16q_1}(\frac{a(\ell)}{\sigma}-1)}}\big)d\ell\\
	&\leq  Cu^2\Delta_n^{\frac{3}{2}}(\varphi_{1,n}(a_0,b_0))^2\sum_{k=0}^{n-1}\int_0^1\big(\frac{1}{\sqrt{X_{t_k}^{a_0,b_0}}} +\frac{1}{(X_{t_k}^{a_0,b_0})^{\frac{1}{2}(\frac{a(\ell)}{\sigma}-1)}}\big) \big(\frac{1}{\sqrt{X_{t_k}^{a_0,b_0}}}\\
	&\qquad+\dfrac{1}{(X_{t_k}^{a_0,b_0})^{\frac{1}{2}+\frac{1}{16q_1}(\frac{a(\ell)}{\sigma}-1)}}+\vert u\vert\dfrac{\varphi_{1,n}(a_0,b_0)\sqrt{\Delta_n}}{(X_{t_k}^{a_0,b_0})^{1+\frac{1}{16q_1}(\frac{a(\ell)}{\sigma}-1)}}\big)d\ell,
	\end{align*}
	for some constant $C>0$. Now, we use the fact 
	$a_0-\vert u\vert\varphi_{1,n}(a_0,b_0)\leq a(\ell)=a_0+\ell u\varphi_{1,n}(a_0,b_0)\leq a_0+\vert u\vert\varphi_{1,n}(a_0,b_0)$ and $a_0-\vert u\vert\varphi_{1,n}(a_0,b_0)\to a_0$, $a_0+\vert u\vert\varphi_{1,n}(a_0,b_0)\to a_0$. Then for any $\epsilon>0$ small enough and $n$ large enough, we have $a_0-\epsilon<a_0-\vert u\vert\varphi_{1,n}(a_0,b_0)<a_0+\epsilon$ and $a_0-\epsilon<a_0+\vert u\vert\varphi_{1,n}(a_0,b_0)<a_0+\epsilon$. This implies that
    \begin{align} 	
    \dfrac{1}{(X_{t_k}^{a_0,b_0})^{a(\ell)}}&=\dfrac{{\bf 1}_{\{X_{t_k}^{a_0,b_0}>1\}}}{(X_{t_k}^{a_0,b_0})^{a(\ell)}}+\dfrac{{\bf 1}_{\{X_{t_k}^{a_0,b_0}\leq 1\}}}{(X_{t_k}^{a_0,b_0})^{a(\ell)}}\leq \dfrac{1}{(X_{t_k}^{a_0,b_0})^{a_0-\epsilon}}+\dfrac{1}{(X_{t_k}^{a_0,b_0})^{a_0+\epsilon}}.\label{major}
    \end{align}	
	For subcritical case $b_0>0$ with $\varphi_{1,n}(a_0,b_0)=\frac{1}{\sqrt{n\Delta_n}}$ (respectively critical case $b_0=0$ with $\varphi_{1,n}(a_0,0)=\frac{1}{\sqrt{\log(n\Delta_n)}}$), taking the expectation in both sides and using \eqref{major}, Lemma \ref{moment2} \textnormal{(i)} (respectively Lemma \ref{moment2} \textnormal{(ii)}) and standard calculations, for $n$ large enough we get $\widehat{\E}^{a_0,b_0}[\vert\sum_{k=0}^{n-1}\widehat{\E}^{a_0,b_0}[\zeta_{k,n}\vert \widehat{\mathcal{F}}_{t_k}]\vert]\leq C\sqrt{\Delta_n}$, which tends to zero, since all the obtained powers are smaller than $\frac{a_0}{\sigma}$ for $\epsilon$ small enough and $q_1$ and $n$ large enough. This completes the proof of condition (i) of Lemma \ref{zero}. 
	
	Similarly, we use Jensen's inequality, \eqref{for1} of Lemma \ref{change} with $V=(R^{a(\ell),b_0})^2$, \eqref{for3} of Lemma \ref{deviation1} with $q=q_0=\frac{13+\sqrt{89}}{20}>1$ and $q_1$ large enough, \eqref{es2} of Lemma \ref{estimate} with $q\in\{1,q_0\}$, and condition {\bf(A)}, for $n$ large enough $\sum_{k=0}^{n-1}\widehat{\E}^{a_0,b_0}[\zeta_{k,n}^2\vert \widehat{\mathcal{F}}_{t_k}]$ is bounded by
	\begin{align*}
	&\sum_{k=0}^{n-1}\frac{u^2(\varphi_{1,n}(a_0,b_0))^2}{\Delta_n^2}\int_0^1\widehat{\E}_{t_k,X_{t_k}^{a_0,b_0}}^{a_0,b_0}[\widetilde{\E}_{t_k,X_{t_k}^{a_0,b_0}}^{a(\ell),b_0}
	[(R^{a(\ell),b_0})^2\vert Y_{t_{k+1}}^{a(\ell),b_0}=X_{t_{k+1}}^{a_0,b_0}]]d\ell\\
	&\leq\sum_{k=0}^{n-1}\frac{u^2(\varphi_{1,n}(a_0,b_0))^2}{\Delta_n^2}\int_0^1\Big\{\widetilde{\E}_{t_k,X_{t_k}^{a_0,b_0}}^{a(\ell),b_0}[(R^{a(\ell),b_0})^2]+C_1\sqrt{\Delta_n}\big(\widetilde{\E}_{t_k,X_{t_k}^{a_0,b_0}}^{a(\ell),b_0}[\vert R^{a(\ell),b_0}\vert^{2q_0}]\big)^{\frac{1}{q_0}} \\
	&\qquad\times \vert u\vert\varphi_{1,n}(a_0,b_0)\big(\frac{1}{\sqrt{X_{t_k}^{a_0,b_0}}}+\frac{\sqrt{\Delta_n}}{X_{t_k}^{a_0,b_0}}\vert u\vert\varphi_{1,n}(a_0,b_0)\big)   \big(1+\dfrac{1}{(X_{t_k}^{a_0,b_0})^{\frac{1}{8pq_1}(\frac{a(\ell)}{\sigma}-1)}}\big)\Big\}d\ell\\
	&\leq Cu^2(\varphi_{1,n}(a_0,b_0))^2\Delta_n^2\sum_{k=0}^{n-1}\int_0^1\big(\frac{1}{X_{t_k}^{a_0,b_0}} +\frac{1}{(X_{t_k}^{a_0,b_0})^{\frac{a(\ell)}{\sigma}-1}}\big)d\ell+Cu^2\vert u\vert \Delta_n^{\frac{5}{2}}(\varphi_{1,n}(a_0,b_0))^3\\
	&\times\sum_{k=0}^{n-1}\int_0^1\big(\frac{1}{X_{t_k}^{a_0,b_0}} +\frac{1}{(X_{t_k}^{a_0,b_0})^{\frac{1}{q_0}(\frac{a(\ell)}{\sigma}-1)}}\big) \big(\frac{1}{\sqrt{X_{t_k}^{a_0,b_0}}}+\dfrac{1}{(X_{t_k}^{a_0,b_0})^{\frac{1}{2}+\frac{1}{8pq_1}(\frac{a(\ell)}{\sigma}-1)}}+\vert u\vert\varphi_{1,n}(a_0,b_0)\\
	&\qquad\times\dfrac{\sqrt{\Delta_n}}{(X_{t_k}^{a_0,b_0})^{1+\frac{1}{8pq_1}(\frac{a(\ell)}{\sigma}-1)}}\big)d\ell,
	\end{align*}
	for some constant $C>0$, where $p=\frac{q_0}{q_0-1}$. It is worth noticing that when using \eqref{for3} of Lemma \ref{deviation1} with $p$, $q=q_0$ satisfying $\frac{1}{p}+\frac{1}{q}=1$, the condition $p=\frac{q_0}{q_0-1}<\frac{a(\ell)}{\sigma}-1$ is valid under condition {\bf(A)}. 
	
	For subcritical case $b_0>0$ with $\varphi_{1,n}(a_0,b_0)=\frac{1}{\sqrt{n\Delta_n}}$ (respectively critical case $b_0=0$ with $\varphi_{1,n}(a_0,0)=\frac{1}{\sqrt{\log(n\Delta_n)}}$), taking the expectation in both sides and using \eqref{major}, Lemma \ref{moment2} \textnormal{(i)} (respectively Lemma \ref{moment2} \textnormal{(ii)}) and standard calculations, for $n$ large enough we get $\widehat{\E}^{a_0,b_0}[\vert\sum_{k=0}^{n-1}\widehat{\E}^{a_0,b_0}[\zeta_{k,n}^2\vert \widehat{\mathcal{F}}_{t_k}]\vert]\leq C\Delta_n$, which tends to zero, since all the obtained powers are smaller than $\frac{a_0}{\sigma}$ for $\epsilon$ small enough and $q_1$ and $n$ large enough. This completes the proof.
\end{proof} 
\begin{lemma}\label{lemma3} Assume condition $\frac{a_0}{\sigma}>4$. Then, for both subcritical and critical cases as $n\to\infty$, we have
	\begin{align*}
	\sum_{k=0}^{n-1}\dfrac{u\varphi_{1,n}(a_0,b_0)}{\Delta_n}\int_0^1\big(R^{a_0,b_0}_{7}-\widetilde{\E}_{t_k,X_{t_{k}}^{a_0,b_0}}^{a(\ell),b_0}[R_5^{a(\ell),b_0}\vert Y_{t_{k+1}}^{a(\ell),b_0}=X_{t_{k+1}}^{a_0,b_0}]\big)d\ell\overset{\widehat{\P}^{a_0,b_0}}{\longrightarrow}0.
	\end{align*}
\end{lemma}
\begin{proof} 
Similarly as above, we apply Lemma \ref{zero} with
	$$
	\zeta_{k,n}:=\dfrac{u\varphi_{1,n}(a_0,b_0)}{\Delta_n}\int_0^1\big(R^{a_0,b_0}_{7}-\widetilde{\E}_{t_k,X_{t_{k}}^{a_0,b_0}}^{a(\ell),b_0}[R_5^{a(\ell),b_0}\vert Y_{t_{k+1}}^{a(\ell),b_0}=X_{t_{k+1}}^{a_0,b_0}]\big)d\ell.
	$$
	Using \eqref{for1} of Lemma \ref{change} with $V=R_5^{a(\ell),b_0}$,  $\widehat{\E}_{t_k,X_{t_k}^{a_0,b_0}}^{a_0,b_0}[R^{a_0,b_0}_{7}]=0$ and $\widetilde{\E}_{t_k,X_{t_{k}}^{a_0,b_0}}^{a(\ell),b_0}[R_5^{a(\ell),b_0}]=0$, \eqref{for3} of Lemma \ref{deviation1} with $q=p=2$ and $q_1$ large enough, and condition $\frac{a_0}{\sigma}>3$ we obtain for $n$ large enough
	\begin{align*} 
	&\vert\sum_{k=0}^{n-1}\widehat{\E}^{a_0,b_0}[\zeta_{k,n}\vert \widehat{\mathcal{F}}_{t_k}]\vert=\big\vert\sum_{k=0}^{n-1}\frac{u\varphi_{1,n}(a_0,b_0)}{\Delta_n}\int_0^1\widehat{\E}_{t_k,X_{t_k}^{a_0,b_0}}^{a_0,b_0}[\widetilde{\E}_{t_k,X_{t_{k}}^{a_0,b_0}}^{a(\ell),b_0}[R_5^{a(\ell),b_0}\vert Y_{t_{k+1}}^{a(\ell),b_0}=X_{t_{k+1}}^{a_0,b_0}]]d\ell\big\vert\\
	&\leq C\dfrac{u^2(\varphi_{1,n}(a_0,b_0))^2 }{\sqrt{\Delta_n}}\sum_{k=0}^{n-1}\int_0^1\big(\widetilde{\E}_{t_k,X_{t_k}^{a_0,b_0}}^{a(\ell),b_0}[\vert R_5^{a(\ell),b_0}\vert^2]\big)^{\frac{1}{2}}\big(\frac{1}{\sqrt{X_{t_k}^{a_0,b_0}}}+\frac{\sqrt{\Delta_n}}{X_{t_k}^{a_0,b_0}}\vert u\vert\varphi_{1,n}(a_0,b_0)\big)  \\
	&\qquad\times  \big(1+\dfrac{1}{(X_{t_k}^{a_0,b_0})^{\frac{1}{16q_1}(\frac{a(\ell)}{\sigma}-1)}}\big)d\ell.
	\end{align*}	
	For $p\geq 1$, using BDG's inequality, It\^o's formula, H\"older's inequality, \eqref{e1}, \eqref{e2}, we get 
	\begin{align}\label{sqrtincre}
	&\widetilde{\E}_{t_k,X_{t_k}^{a_0,b_0}}^{a,b}\big[\vert R_5^{a,b}\vert^{2p}\big]\leq C\dfrac{\Delta_n^{2p}}{(X_{t_k}^{a_0,b_0})^{2p}}\Delta_n^{p-1}\int_{t_k}^{t_{k+1}}\widetilde{\E}_{t_k,X_{t_k}^{a_0,b_0}}^{a,b}\big[\vert\sqrt{Y_s^{a,b}}-\sqrt{Y_{t_k}^{a,b}}\vert^{2p}\big]ds\notag\\
	&\leq C\dfrac{\Delta_n^{3p-1}}{(X_{t_k}^{a_0,b_0})^{2p}}\int_{t_k}^{t_{k+1}}\widetilde{\E}_{t_k,X_{t_k}^{a_0,b_0}}^{a,b}\big[\vert\int_{t_k}^{s}\big((\frac{a}{2}-\frac{\sigma}{4})\frac{1}{\sqrt{Y_u^{a,b}}}-\frac{b}{2}\sqrt{Y_u^{a,b}}\big)du+\int_{t_k}^{s}\sqrt{\frac{\sigma}{2}}dW_u\vert^{2p}\big]ds\notag\\
	&\leq C\frac{\Delta_n^{3p-1}}{(X_{t_k}^{a_0,b_0})^{2p}}\int_{t_k}^{t_{k+1}}\big((s-t_k)^{2p-1}\int_{t_k}^{s}\big(\widetilde{\E}_{t_k,X_{t_k}^{a_0,b_0}}^{a,b}\big[\frac{1}{\vert Y_u^{a,b}\vert^p}\big]+\vert b\vert\widetilde{\E}_{t_k,X_{t_k}^{a_0,b_0}}^{a,b}[\vert Y_u^{a,b}\vert^p]\big)du\notag\\
	&\qquad+(s-t_k)^p\big)ds\leq  C\Delta_n^{4p}\big(\frac{1}{(X_{t_k}^{a_0,b_0})^{3p}}+\frac{\vert b\vert}{(X_{t_k}^{a_0,b_0})^p}+\frac{1}{(X_{t_k}^{a_0,b_0})^{2p}}\big),
	\end{align}
	provided that $p<\frac{a}{\sigma}-1$. Then, applying \eqref{sqrtincre} with $p=1$, condition $\frac{a_0}{\sigma}>3$, we get
	\begin{align*} 
	&\vert\sum_{k=0}^{n-1}\widehat{\E}^{a_0,b_0}[\zeta_{k,n}\vert \widehat{\mathcal{F}}_{t_k}]\vert\leq Cu^2(\varphi_{1,n}(a_0,b_0))^2 \Delta_n^{\frac{3}{2}}\sum_{k=0}^{n-1}\int_0^1\big(\dfrac{1}{(X_{t_k}^{a_0,b_0})^{\frac{3}{2}}}+\dfrac{\vert b_0\vert}{(X_{t_k}^{a_0,b_0})^{\frac{1}{2}}}+\dfrac{1}{X_{t_k}^{a_0,b_0}}\big)\\
	&\qquad\times\big(\frac{1}{\sqrt{X_{t_k}^{a_0,b_0}}}+\dfrac{1}{(X_{t_k}^{a_0,b_0})^{\frac{1}{2}+\frac{1}{16q_1}(\frac{a(\ell)}{\sigma}-1)}}+\vert u\vert\varphi_{1,n}(a_0,b_0)\dfrac{\sqrt{\Delta_n}}{(X_{t_k}^{a_0,b_0})^{1+\frac{1}{16q_1}(\frac{a(\ell)}{\sigma}-1)}}\big)d\ell.
	\end{align*}
	For subcritical case $b_0>0$ with $\varphi_{1,n}(a_0,b_0)=\frac{1}{\sqrt{n\Delta_n}}$ (respectively critical case $b_0=0$ with $\varphi_{1,n}(a_0,0)=\frac{1}{\sqrt{\log(n\Delta_n)}}$), taking the expectation in both sides and using \eqref{major}, Lemma \ref{moment2} \textnormal{(i)} (respectively Lemma \ref{moment2} \textnormal{(ii)}) and standard calculations, for $n$ large enough we get $\widehat{\E}^{a_0,b_0}[\vert\sum_{k=0}^{n-1}\widehat{\E}^{a_0,b_0}[\zeta_{k,n}\vert \widehat{\mathcal{F}}_{t_k}]\vert]\leq C\sqrt{\Delta_n}$, which tends to zero, since all the obtained powers are smaller than $\frac{a_0}{\sigma}$ for $\epsilon$ small enough and $q_1$ and $n$ large enough. 
	
	Similarly, we use Jensen's inequality, \eqref{for1} of Lemma \ref{change} with $V=(R_5^{a(\ell),b_0})^2$, \eqref{for3} of Lemma \ref{deviation1} with $q=p=2$ and $q_1$ large enough, \eqref{sqrtincre} with $p\in\{1,2\}$, and condition $\frac{a_0}{\sigma}>3$ in the same way as above we obtain for $n$ large enough 
	\begin{align*}
	&\sum_{k=0}^{n-1}\widehat{\E}^{a_0,b_0}[\zeta_{k,n}^2\vert \widehat{\mathcal{F}}_{t_k}]\leq 2\frac{u^2(\varphi_{1,n}(a_0,b_0))^2}{\Delta_n^2}\sum_{k=0}^{n-1}\int_0^1\Big\{\widehat{\E}_{t_k,X_{t_k}^{a_0,b_0}}^{a_0,b_0}[(R^{a_0,b_0}_{7})^2]\\
	&\qquad+\widehat{\E}_{t_k,X_{t_k}^{a_0,b_0}}^{a_0,b_0}\big[\widetilde{\E}_{t_k,X_{t_{k}}^{a_0,b_0}}^{a(\ell),b_0}[(R_5^{a(\ell),b_0})^2\vert Y_{t_{k+1}}^{a(\ell),b_0}=X_{t_{k+1}}^{a_0,b_0}]\big]\Big\}d\ell\\	
	&\leq Cu^2\Delta_n^2(\varphi_{1,n}(a_0,b_0))^2\sum_{k=0}^{n-1}\big(\frac{1}{(X_{t_{k}}^{a_0,b_0})^3}+\frac{\vert b_0\vert}{X_{t_{k}}^{a_0,b_0}}+\frac{1}{(X_{t_{k}}^{a_0,b_0})^2}\big)+Cu^2\vert u\vert \Delta_n^{\frac{5}{2}}(\varphi_{1,n}(a_0,b_0))^3\\
	&\qquad\times\sum_{k=0}^{n-1}\int_0^1\big(\frac{1}{(X_{t_k}^{a_0,b_0})^{3}}+\frac{\vert b_0\vert}{X_{t_k}^{a_0,b_0}}+\frac{1}{(X_{t_k}^{a_0,b_0})^{2}}\big)\big(\frac{1}{\sqrt{X_{t_k}^{a_0,b_0}}}+\dfrac{1}{(X_{t_k}^{a_0,b_0})^{\frac{1}{2}+\frac{1}{16q_1}(\frac{a(\ell)}{\sigma}-1)}}\\
	&\qquad+\vert u\vert\varphi_{1,n}(a_0,b_0)\dfrac{\sqrt{\Delta_n}}{(X_{t_k}^{a_0,b_0})^{1+\frac{1}{16q_1}(\frac{a(\ell)}{\sigma}-1)}}\big)d\ell,
	\end{align*}
	where the estimate for $\widehat{\E}_{t_k,X_{t_k}^{a_0,b_0}}^{a_0,b_0}[(R^{a_0,b_0}_{7})^2]$ is proceeded similarly as for \eqref{sqrtincre}.
	
	For subcritical case $b_0>0$ (respectively critical case $b_0=0$), taking the expectation in both sides and using \eqref{major}, Lemma \ref{moment2} \textnormal{(i)} (respectively Lemma \ref{moment2} \textnormal{(ii)}) and standard calculations, under $\frac{a_0}{\sigma}>4$ we get $\widehat{\E}^{a_0,b_0}[\vert\sum_{k=0}^{n-1}\widehat{\E}^{a_0,b_0}[\zeta_{k,n}^2\vert \widehat{\mathcal{F}}_{t_k}]\vert]\leq C\Delta_n$ for $n$ large enough, which tends to zero, since all the obtained powers are smaller than $\frac{a_0}{\sigma}$ for $\epsilon$ small enough and $q_1$ and $n$ large enough. This completes the proof.	 
\end{proof}

\begin{lemma}\label{lemma2} Assume condition $\frac{a_0}{\sigma}>3$. Then, for the subcritical case as $n\to\infty$, we have
	\begin{align*}
	\sum_{k=0}^{n-1}\dfrac{u\varphi_{1,n}(a_0,b_0)}{\Delta_n}\int_0^1\big(R^{a_0,b_0}_{6}-\widetilde{\E}_{t_k,X_{t_{k}}^{a_0,b_0}}^{a(\ell),b_0}[R_4^{a(\ell),b_0}\vert Y_{t_{k+1}}^{a(\ell),b_0}=X_{t_{k+1}}^{a_0,b_0}]\big)d\ell\overset{\widehat{\P}^{a_0,b_0}}{\longrightarrow}0.
	\end{align*}
	For the critical case when $b_0=0$, we have $R_4^{a(\ell),b_0}=R^{a_0,b_0}_{6}=0$. 
\end{lemma}
\begin{proof}
	Similarly as above, we apply Lemma \ref{zero} with 
	$$
	\zeta_{k,n}:=\dfrac{u\varphi_{1,n}(a_0,b_0)}{\Delta_n}\int_0^1\big(R^{a_0,b_0}_{6}-\widetilde{\E}_{t_k,X_{t_{k}}^{a_0,b_0}}^{a(\ell),b_0}[R_4^{a(\ell),b_0}\vert Y_{t_{k+1}}^{a(\ell),b_0}=X_{t_{k+1}}^{a_0,b_0}]\big)d\ell.
	$$
	Using \eqref{for1} of Lemma \ref{change} with $V=R_4^{a(\ell),b_0}$, \eqref{for2} of Lemma \ref{change} with $\widehat{V}=R_6^{a_0,b_0}$,  \eqref{for3} and \eqref{for4} of Lemma \ref{deviation1} with $q=p=2$ and $q_1$ large enough, and condition $\frac{a_0}{\sigma}>3$, for $n$ large enough $\vert\sum_{k=0}^{n-1}\widehat{\E}^{a_0,b_0}[\zeta_{k,n}\vert \widehat{\mathcal{F}}_{t_k}]\vert$ is bounded by
	\begin{align*} 
	&C\dfrac{\vert u\vert\varphi_{1,n}(a_0,b_0) }{\Delta_n}\sum_{k=0}^{n-1}\int_0^1\Big\{\dfrac{\sqrt{\Delta_n}}{\sqrt{X_{t_k}^{a_0,b_0}}}\big\vert\int_{a(\ell)}^{a_0}\big(\widehat{\E}_{t_k,X_{t_k}^{a_0,b_0}}^{\alpha,b_0}\big[(R_6^{\alpha,b_0})^2\big]\big)^{\frac{1}{2}}d\alpha\big\vert +\sqrt{\Delta_n} \vert u\vert\varphi_{1,n}(a_0,b_0)\\
	&\quad\times\big(\widetilde{\E}_{t_k,X_{t_k}^{a_0,b_0}}^{a(\ell),b_0}[\vert R_4^{a(\ell),b_0}\vert^2]\big)^{\frac{1}{2}}\big(\frac{1}{\sqrt{X_{t_k}^{a_0,b_0}}}+\frac{\sqrt{\Delta_n}}{X_{t_k}^{a_0,b_0}}\vert u\vert\varphi_{1,n}(a_0,b_0)\big)  \big(1+\dfrac{1}{(X_{t_k}^{a_0,b_0})^{\frac{1}{16q_1}(\frac{a(\ell)}{\sigma}-1)}}\big)\Big\} d\ell,
	\end{align*}
	where we have used  $\widehat{\E}_{t_k,X_{t_k}^{a_0,b_0}}^{a(\ell),b_0}[R_6^{a(\ell),b_0}]-\widetilde{\E}_{t_k,X_{t_k}^{a_0,b_0}}^{a(\ell),b_0}[R_4^{a(\ell),b_0}]=0$ since $Y^{a(\ell),b_0}$ is the independent copy of $X^{a(\ell),b_0}$.
	
	For $p\geq 1$, using equation \eqref{c2eq1rajoute}, BDG's inequality, H\"older's inequality and \eqref{e1}, we get 
	\begin{align}\label{incrementY}
	&\widetilde{\E}_{t_k,X_{t_k}^{a_0,b_0}}^{a,b}\big[\vert \int_{t_k}^{t_{k+1}}(Y_s^{a,b}-Y_{t_k}^{a,b})ds\vert^{2p}\big]\leq \Delta_n^{2p-1}\int_{t_k}^{t_{k+1}}\widetilde{\E}_{t_k,X_{t_k}^{a_0,b_0}}^{a,b}\big[\vert Y_s^{a,b}-Y_{t_k}^{a,b}\vert^{2p}\big]ds\notag\\
	&\leq C\Delta_n^{2p-1}\int_{t_k}^{t_{k+1}}\Big((s-t_k)^{2p-1}\int_{t_k}^{s}\big(1+\vert b\vert\widetilde{\E}_{t_k,X_{t_k}^{a_0,b_0}}^{a,b}[\vert Y_u^{a,b} \vert^{2p}]\big)du+(s-t_k)^{p-1}\notag\\
	&\qquad\times\int_{t_k}^{s}\widetilde{\E}_{t_k,X_{t_k}^{a_0,b_0}}^{a,b}[\vert Y_u^{a,b} \vert^{p}]du\Big)ds\leq C\Delta_n^{3p}\big(1+\vert b\vert(X_{t_k}^{a_0,b_0})^{2p}+(X_{t_k}^{a_0,b_0})^{p}\big).
	\end{align}
	Then, applying \eqref{incrementY} with $p=1$, we get
	\begin{align*} 
	&\Big\vert\sum_{k=0}^{n-1}\widehat{\E}^{a_0,b_0}[\zeta_{k,n}\vert \widehat{\mathcal{F}}_{t_k}]\Big\vert \leq Cu^2\vert b_0\vert(\varphi_{1,n}(a_0,b_0))^2 \Delta_n^{2}\sum_{k=0}^{n-1}\int_0^1\Big\{\big(\dfrac{1}{(X_{t_k}^{a_0,b_0})^{\frac{3}{2}}}+\dfrac{\vert b_0\vert}{\sqrt{X_{t_k}^{a_0,b_0}}}+\dfrac{1}{X_{t_k}^{a_0,b_0}}\big)\\
	&\qquad+\big(\dfrac{1}{X_{t_k}^{a_0,b_0}}+\vert b_0\vert+\dfrac{1}{(X_{t_k}^{a_0,b_0})^{\frac{1}{2}}}\big)\big(\frac{1}{\sqrt{X_{t_k}^{a_0,b_0}}}+\dfrac{1}{(X_{t_k}^{a_0,b_0})^{\frac{1}{2}+\frac{1}{16q_1}(\frac{a(\ell)}{\sigma}-1)}}\\
	&\qquad+\vert u\vert\varphi_{1,n}(a_0,b_0)\dfrac{\sqrt{\Delta_n}}{(X_{t_k}^{a_0,b_0})^{1+\frac{1}{16q_1}(\frac{a(\ell)}{\sigma}-1)}}\big)\Big\}d\ell.
	\end{align*}
	For subcritical case $b_0>0$, using \eqref{major} and Lemma \ref{moment2} \textnormal{(i)}, $\widehat{\E}^{a_0,b_0}[\vert\sum_{k=0}^{n-1}\widehat{\E}^{a_0,b_0}[\zeta_{k,n}\vert \widehat{\mathcal{F}}_{t_k}]\vert]\leq C\Delta_n$, which tends to zero, since all the obtained powers are smaller than $\frac{a_0}{\sigma}$ for $\epsilon$ small enough and $q_1$ and $n$ large enough. 
	
	Next, using Jensen's inequality, \eqref{for1} of Lemma \ref{change} with $V=(R_4^{a(\ell),b_0})^2$, \eqref{for3} of Lemma \ref{deviation1} with $q=p=2$ and $q_1$ large enough, \eqref{incrementY} with $p\in\{1,2\}$, and condition $\frac{a_0}{\sigma}>3$ in the same way as above we obtain for $n$ large enough, similarly,
	\begin{align*}
	&\sum_{k=0}^{n-1}\widehat{\E}^{a_0,b_0}[\zeta_{k,n}^2\vert \widehat{\mathcal{F}}_{t_k}]\leq Cu^2b_0^2\Delta_n^3(\varphi_{1,n}(a_0,b_0))^2\sum_{k=0}^{n-1}\big(\frac{1}{(X_{t_{k}}^{a_0,b_0})^2}+\vert b_0\vert+\frac{1}{X_{t_{k}}^{a_0,b_0}}\big)\\
	&\qquad+Cu^2\vert u\vert b_0^2\Delta_n^{\frac{7}{2}}(\varphi_{1,n}(a_0,b_0))^3\sum_{k=0}^{n-1}\int_0^1\big(\frac{1}{(X_{t_{k}}^{a_0,b_0})^2}+\vert b_0\vert+\frac{1}{X_{t_{k}}^{a_0,b_0}}\big)\big(\frac{1}{\sqrt{X_{t_k}^{a_0,b_0}}}\\
	&\qquad+\dfrac{1}{(X_{t_k}^{a_0,b_0})^{\frac{1}{2}+\frac{1}{16q_1}(\frac{a(\ell)}{\sigma}-1)}}+\vert u\vert\varphi_{1,n}(a_0,b_0)\dfrac{\sqrt{\Delta_n}}{(X_{t_k}^{a_0,b_0})^{1+\frac{1}{16q_1}(\frac{a(\ell)}{\sigma}-1)}}\big)d\ell ,
	\end{align*}
	where the estimate for $\widehat{\E}_{t_k,X_{t_k}^{a_0,b_0}}^{a_0,b_0}[(R^{a_0,b_0}_{6})^2]$ is proceeded similarly as for \eqref{incrementY}.
	 
	For subcritical case $b_0>0$, using \eqref{major} and Lemma \ref{moment2} \textnormal{(i)},  $\widehat{\E}^{a_0,b_0}[\vert\sum_{k=0}^{n-1}\widehat{\E}^{a_0,b_0}[\zeta_{k,n}^2\vert \widehat{\mathcal{F}}_{t_k}]\vert]\leq C\Delta_n^2$, which tends to zero, since all the obtained powers are smaller than $\frac{a_0}{\sigma}$ for $\epsilon$ small enough and $q_1$ and $n$ large enough. This completes the proof.
\end{proof}

\begin{lemma}\label{lemma4} Assume condition {\bf(A)}. Then, as $n\to\infty$, we have
	\begin{align*} 
	\sum_{k=0}^{n-1}\dfrac{v\varphi_{2,n}(a_0,b_0)}{\Delta_n}\int_0^1\widetilde{\E}_{t_k,X_{t_k}^{a_0,b_0}}^{a_n,b(\ell)}[H^{a_n,b(\ell)}\vert Y^{a_n,b(\ell)}_{t_{k+1}}=X_{t_{k+1}}^{a_0,b_0}]d\ell\overset{\widehat{\P}^{a_0,b_0}}{\longrightarrow}0.
	\end{align*}
\end{lemma}
\begin{proof} 
Similarly as above, we apply Lemma \ref{zero} with
$$
\zeta_{k,n}:=\dfrac{v\varphi_{2,n}(a_0,b_0)}{\Delta_n}\int_0^1\widetilde{\E}_{t_k,X_{t_k}^{a_0,b_0}}^{a_n,b(\ell)}[H^{a_n,b(\ell)}\vert Y^{a_n,b(\ell)}_{t_{k+1}}=X_{t_{k+1}}^{a_0,b_0}]d\ell.
$$ 
Using \eqref{for1} of Lemma \ref{change} with $V=H^{a_n,b(\ell)}$,  $\widetilde{\E}_{t_k,X_{t_k}^{a_0,b_0}}^{a_n,b(\ell)}[H^{a_n,b(\ell)}]=0$ by \eqref{es3}, \eqref{for3} of Lemma \ref{deviation1} with $q=p=2$ and $q_1$ large enough, \eqref{es4} of Lemma \ref{estimate} with $q=1$, and condition {\bf(A)}, we obtain for $n$ large enough
\begin{align*} 
&\vert\sum_{k=0}^{n-1}\widehat{\E}^{a_0,b_0}[\zeta_{k,n}\vert \widehat{\mathcal{F}}_{t_k}]\vert\leq C\sum_{k=0}^{n-1}\dfrac{\vert v\vert\varphi_{2,n}(a_0,b_0)}{\Delta_n}\int_0^1\sqrt{\Delta_n}\big(\widetilde{\E}_{t_k,X_{t_k}^{a_0,b_0}}^{a_n,b(\ell)}[\vert H^{a_n,b(\ell)}\vert^2]\big)^{\frac{1}{2}} \\
&\qquad\times e^{C_2v^2(\varphi_{2,n}(a_0,b_0))^2\Delta_nX_{t_k}^{a_0,b_0}}\Big(\vert u\vert\varphi_{1,n}(a_0,b_0)\big(\frac{1}{\sqrt{X_{t_k}^{a_0,b_0}}}+\frac{\sqrt{\Delta_n}}{X_{t_k}^{a_0,b_0}}\vert u\vert\varphi_{1,n}(a_0,b_0)\\
&\qquad+\sqrt{\Delta_n}\vert v\vert\varphi_{2,n}(a_0,b_0)\big)    (1+\dfrac{1}{(X_{t_k}^{a_0,b_0})^{\frac{1}{16q_1}(\frac{a_n}{\sigma}-1)}})+\vert v\vert\varphi_{2,n}(a_0,b_0)\big(1+\sqrt{X_{t_k}^{a_0,b_0}}\\
&\qquad+\sqrt{\Delta_n}\vert v\vert\varphi_{2,n}(a_0,b_0)X_{t_k}^{a_0,b_0}\big) \Big)d\ell
\leq C\vert v\vert\varphi_{2,n}(a_0,b_0)\Delta_n^{\frac{3}{2}}\sum_{k=0}^{n-1}e^{C_2v^2(\varphi_{2,n}(a_0,b_0))^2\Delta_nX_{t_k}^{a_0,b_0}}\\
&\qquad\times\Big(\vert u\vert\varphi_{1,n}(a_0,b_0)\big(\frac{1}{\sqrt{X_{t_k}^{a_0,b_0}}}+\frac{\sqrt{\Delta_n}}{X_{t_k}^{a_0,b_0}}\vert u\vert\varphi_{1,n}(a_0,b_0)+\sqrt{\Delta_n}\vert v\vert\varphi_{2,n}(a_0,b_0)\big)  \\
&\qquad\times \big(1+\dfrac{1}{(X_{t_k}^{a_0,b_0})^{\frac{1}{16q_1}(\frac{a_n}{\sigma}-1)}}\big)+\vert v\vert\varphi_{2,n}(a_0,b_0)\big(1+\sqrt{X_{t_k}^{a_0,b_0}}+\sqrt{\Delta_n}\vert v\vert\varphi_{2,n}(a_0,b_0)X_{t_k}^{a_0,b_0}\big)\Big)\\
&\qquad\times(1+X_{t_k}^{a_0,b_0})\big(\frac{1}{\sqrt{X_{t_k}^{a_0,b_0}}} +\frac{1}{(X_{t_k}^{a_0,b_0})^{\frac{1}{2}(\frac{a_n}{\sigma}-1)}}\big).
\end{align*}
Then, using Young's inequality for products with $\frac{1}{\overline{p}_0}+\frac{1}{\overline{q}_0}=1$ and $\overline{q}_0$ close to $1$, for $n$ large enough $\vert\sum_{k=0}^{n-1}\widehat{\E}^{a_0,b_0}[\zeta_{k,n}\vert \widehat{\mathcal{F}}_{t_k}]\vert$ is bounded by
\begin{align*} 
&C\vert v\vert\varphi_{2,n}(a_0,b_0)\Delta_n^{\frac{3}{2}}\frac{1}{\overline{p}_0}\big(\vert u\vert\varphi_{1,n}(a_0,b_0)+\vert v\vert\varphi_{2,n}(a_0,b_0)\big)\sum_{k=0}^{n-1}e^{C_2\overline{p}_0v^2(\varphi_{2,n}(a_0,b_0))^2\Delta_nX_{t_k}^{a_0,b_0}}\\
&+C\vert v\vert\varphi_{2,n}(a_0,b_0)\Delta_n^{\frac{3}{2}}\frac{1}{\overline{q}_0}\sum_{k=0}^{n-1}\Big(\vert u\vert\varphi_{1,n}(a_0,b_0)\big((X_{t_k}^{a_0,b_0})^{\frac{\overline{q}_0}{2}}+(\sqrt{\Delta_n}\vert v\vert\varphi_{2,n}(a_0,b_0))^{\overline{q}_0}(X_{t_k}^{a_0,b_0})^{\overline{q}_0}\\
&+\dfrac{1}{(X_{t_k}^{a_0,b_0})^{\overline{q}_0(\frac{1}{2}+\frac{1}{16q_1}(\frac{a_n}{\sigma}-1))}} +\frac{(\sqrt{\Delta_n})^{\overline{q}_0}}{(X_{t_k}^{a_0,b_0})^{\overline{q}_0(1+\frac{1}{16q_1}(\frac{a_n}{\sigma}-1))}}(\vert u\vert\varphi_{1,n}(a_0,b_0))^{\overline{q}_0} \big)+\vert v\vert\varphi_{2,n}(a_0,b_0)\\
& \times\big(1+(X_{t_k}^{a_0,b_0})^{\frac{3}{2}\overline{q}_0}+(\sqrt{\Delta_n}\vert v\vert\varphi_{2,n}(a_0,b_0))^{\overline{q}_0}(X_{t_k}^{a_0,b_0})^{2\overline{q}_0}\big)\Big)\big(\frac{1}{(X_{t_k}^{a_0,b_0})^{\frac{\overline{q}_0}{2}}} +\frac{1}{(X_{t_k}^{a_0,b_0})^{\frac{\overline{q}_0}{2}(\frac{a_n}{\sigma}-1)}}\big).
\end{align*}
Now using $(1.9)$ of \cite[Proposition 1.2.4]{A15}, for $n$ large enough, we get
\begin{align}\label{boundedexp}
M_n:=\widehat{\E}^{a_0,b_0}\Big[e^{C_2\overline{p}_0v^2(\varphi_{2,n}(a_0,b_0))^2\Delta_nX_{t_k}^{a_0,b_0}}\Big]\leq C,
\end{align}
for some constant $C>0$ for three cases. Indeed, we prove the convergence of $M_n$ toward $1$ as $n\to\infty$. For this, for subcritical case $b_0>0$ (respectively critical case $b_0=0$, respectively supercritical case $b_0<0$), we take $\varphi_{2,n}(a_0,b_0)=\frac{1}{\sqrt{n\Delta_n}}$ (respectively $\varphi_{2,n}(a_0,0)=\frac{1}{n\Delta_n}$, respectively   $\varphi_{2,n}(a_0,b_0)=e^{b_0\frac{n\Delta_n}{2}}$) and use that $1-e^{-b_0t_k}\leq b_0t_k$, $\frac{t_k}{n\Delta_n}\leq 1$ and $e^{-b_0t_k}\leq 1$ (respectively $\frac{t_k}{n\Delta_n}\leq 1$, respectively $e^{b_0(n\Delta_n-t_k)}\leq 1$ and $n\Delta_n\to\infty$).

Thus, for subcritical case $b_0>0$ with $(\varphi_{1,n}(a_0,b_0),\varphi_{2,n}(a_0,b_0))=(\frac{1}{\sqrt{n\Delta_n}},\frac{1}{\sqrt{n\Delta_n}})$ (respectively critical case $b_0=0$ with $(\varphi_{1,n}(a_0,0),\varphi_{2,n}(a_0,0))=(\frac{1}{\sqrt{\log(n\Delta_n)}},\frac{1}{n\Delta_n})$, respectively supercritical case $b_0<0$ with $(\varphi_{1,n}(a_0,b_0),\varphi_{2,n}(a_0,b_0))=(0,e^{b_0\frac{n\Delta_n}{2}})$), using \eqref{boundedexp}, Lemma \ref{moment2} \textnormal{(i)} (respectively Lemma \ref{moment2} \textnormal{(ii)}, respectively Lemma \ref{moment2} \textnormal{(iii)}, \eqref{supercriticalconvergence2}) and standard calculations, we get $\widehat{\E}^{a_0,b_0}[\vert\sum_{k=0}^{n-1}\widehat{\E}^{a_0,b_0}[\zeta_{k,n}\vert \widehat{\mathcal{F}}_{t_k}]\vert]\leq C\sqrt{\Delta_n}$ (respectively $\widehat{\E}^{a_0,0}[\vert\sum_{k=0}^{n-1}\widehat{\E}^{a_0,0}[\zeta_{k,n}\vert \widehat{\mathcal{F}}_{t_k}]\vert]\leq C(n\Delta_n^3)^{\frac{1}{4}}$, respectively $\widehat{\E}^{a_0,b_0}[\vert\sum_{k=0}^{n-1}\widehat{\E}^{a_0,b_0}[\zeta_{k,n}\vert \widehat{\mathcal{F}}_{t_k}]\vert]\leq C(e^{b_0n\Delta_n}n\Delta_n+(\Delta_n^{2}e^{-b_0n\Delta_n})^{\frac{1}{4}})$), which tends to zero, since all the obtained powers are smaller than $\frac{a_0}{\sigma}$ for $\overline{q}_0$ close to $1$ and $q_1$ and $n$ large enough.

Next, using Jensen's inequality, \eqref{for1} of Lemma \ref{change} with $V=(H^{a_n,b(\ell)})^2$, \eqref{for3} of Lemma \ref{deviation1} with $q=q_0=\frac{13+\sqrt{89}}{20}$ and $q_1$ large enough, \eqref{es4} of Lemma \ref{estimate} with $q\in\{1,q_0\}$, condition {\bf(A)}, and Young's inequality for products with $\frac{1}{\overline{p}_0}+\frac{1}{\overline{q}_0}=1$ and $\overline{q}_0$ close to $1$, the same arguments as above work here and then for $n$ large enough $\sum_{k=0}^{n-1}\widehat{\E}^{a_0,b_0}[\zeta_{k,n}^2\vert \widehat{\mathcal{F}}_{t_k}]$ is bounded by
\begin{align*}
&\sum_{k=0}^{n-1}\frac{v^2(\varphi_{2,n}(a_0,b_0))^2}{\Delta_n^2}\int_0^1\Big\{\widetilde{\E}_{t_k,X_{t_k}^{a_0,b_0}}^{a_n,b(\ell)}
[(H^{a_n,b(\ell)})^2]+C_1\sqrt{\Delta_n}\big(\widetilde{\E}_{t_k,X_{t_k}^{a_0,b_0}}^{a_n,b(\ell)}[\vert H^{a_n,b(\ell)}\vert^{2q_0}]\big)^{\frac{1}{q_0}} \\
&\qquad\times e^{C_2v^2(\varphi_{2,n}(a_0,b_0))^2\Delta_nX_{t_k}^{a_0,b_0}}\Big(\vert u\vert\varphi_{1,n}(a_0,b_0)\big(\frac{1}{\sqrt{X_{t_k}^{a_0,b_0}}}+\frac{\sqrt{\Delta_n}}{X_{t_k}^{a_0,b_0}}\vert u\vert\varphi_{1,n}(a_0,b_0)   \\
&\qquad+\sqrt{\Delta_n}\vert v\vert\varphi_{2,n}(a_0,b_0)\big) \big(1+\dfrac{1}{(X_{t_k}^{a_0,b_0})^{\frac{1}{8pq_1}(\frac{a_n}{\sigma}-1)}}\big)+\vert v\vert\varphi_{2,n}(a_0,b_0)\big(1+\sqrt{X_{t_k}^{a_0,b_0}}\\
&\qquad+\sqrt{\Delta_n}\vert v\vert\varphi_{2,n}(a_0,b_0)X_{t_k}^{a_0,b_0}\big) \Big)\Big\}d\ell\\
&\leq Cv^2(\varphi_{2,n}(a_0,b_0))^2\Delta_n^2\sum_{k=0}^{n-1}(1+(X_{t_k}^{a_0,b_0})^2)\big(\frac{1}{X_{t_k}^{a_0,b_0}} +\frac{1}{(X_{t_k}^{a_0,b_0})^{\frac{a_n}{\sigma}-1}}\big)+Cv^2(\varphi_{2,n}(a_0,b_0))^2\Delta_n^{\frac{5}{2}}\\
&\qquad\times\frac{1}{\overline{p}_0}\big(\vert u\vert\varphi_{1,n}(a_0,b_0)+\vert v\vert\varphi_{2,n}(a_0,b_0)\big)\sum_{k=0}^{n-1}e^{C_2\overline{p}_0v^2(\varphi_{2,n}(a_0,b_0))^2\Delta_nX_{t_k}^{a_0,b_0}}+Cv^2(\varphi_{2,n}(a_0,b_0))^2\\
&\qquad\times\frac{\Delta_n^{\frac{5}{2}}}{\overline{q}_0}\sum_{k=0}^{n-1}\Big(\vert u\vert\varphi_{1,n}(a_0,b_0)\big((X_{t_k}^{a_0,b_0})^{\frac{3}{2}\overline{q}_0}+(\sqrt{\Delta_n}\vert v\vert\varphi_{2,n}(a_0,b_0))^{\overline{q}_0}(X_{t_k}^{a_0,b_0})^{2\overline{q}_0}\\
&+\dfrac{1}{(X_{t_k}^{a_0,b_0})^{\overline{q}_0(\frac{1}{2}+\frac{1}{8pq_1}(\frac{a_n}{\sigma}-1))}} +\frac{(\sqrt{\Delta_n})^{\overline{q}_0}}{(X_{t_k}^{a_0,b_0})^{\overline{q}_0(1+\frac{1}{8pq_1}(\frac{a_n}{\sigma}-1))}}(\vert u\vert\varphi_{1,n}(a_0,b_0))^{\overline{q}_0} \big)+\vert v\vert\varphi_{2,n}(a_0,b_0)\\
&\times\big(1+(X_{t_k}^{a_0,b_0})^{\frac{5}{2}\overline{q}_0} +(\sqrt{\Delta_n}\vert v\vert\varphi_{2,n}(a_0,b_0))^{\overline{q}_0}(X_{t_k}^{a_0,b_0})^{3\overline{q}_0}\big) \Big)\big(\frac{1}{(X_{t_k}^{a_0,b_0})^{\overline{q}_0}} +\frac{1}{(X_{t_k}^{a_0,b_0})^{\frac{\overline{q}_0}{q_0}(\frac{a_n}{\sigma}-1)}}\big),
\end{align*}
where $p=\frac{q_0}{q_0-1}$. 

Thus, for subcritical case $b_0>0$ (respectively critical case $b_0=0$, respectively supercritical case $b_0<0$), using \eqref{boundedexp}, Lemma \ref{moment2} \textnormal{(i)} (respectively Lemma \ref{moment2} \textnormal{(ii)}, respectively Lemma \ref{moment2} \textnormal{(iii)}, \eqref{supercriticalconvergence2}) and standard calculations, we get $\widehat{\E}^{a_0,b_0}[\vert\sum_{k=0}^{n-1}\widehat{\E}^{a_0,b_0}[\zeta_{k,n}^2\vert \widehat{\mathcal{F}}_{t_k}]\vert]\leq C\Delta_n$ (respectively $\widehat{\E}^{a_0,0}[\vert\sum_{k=0}^{n-1}\widehat{\E}^{a_0,0}[\zeta_{k,n}^2\vert \widehat{\mathcal{F}}_{t_k}]\vert]\leq C\Delta_n$, respectively $\widehat{\E}^{a_0,b_0}[\vert\sum_{k=0}^{n-1}\widehat{\E}^{a_0,b_0}[\zeta_{k,n}^2\vert \widehat{\mathcal{F}}_{t_k}]\vert]\leq C(\Delta_n^{2}e^{-b_0n\Delta_n})^{\frac{3}{4}}$), which tends to zero, since all the obtained powers are smaller than $\frac{a_0}{\sigma}$ for $\overline{q}_0$ close to $1$ and $q_1$ and $n$ large enough. This completes the proof.
\end{proof}

\begin{lemma}\label{lemma6} Assume condition $\frac{a_0}{\sigma}>3$. Then, as $n\to\infty$, we have
	\begin{align*}
	\sum_{k=0}^{n-1}\dfrac{v\varphi_{2,n}(a_0,b_0)}{\Delta_n}\int_0^1\big(-X_{t_{k}}^{a_0,b_0}R_7^{a_0,b_0}+\widetilde{\E}_{t_k,X_{t_{k}}^{a_0,b_0}}^{a_n,b(\ell)}[X_{t_{k}}^{a_0,b_0}R_5^{a_n,b(\ell)}\vert Y_{t_{k+1}}^{a_n,b(\ell)}=X_{t_{k+1}}^{a_0,b_0}]\big)d\ell\overset{\widehat{\P}^{a_0,b_0}}{\longrightarrow}0.
	\end{align*}
\end{lemma}
\begin{proof} 	
	Similarly as above, we apply Lemma \ref{zero} with 
	$$
	\zeta_{k,n}:=\dfrac{v\varphi_{2,n}(a_0,b_0)}{\Delta_n}\int_0^1\big(-X_{t_{k}}^{a_0,b_0}R_7^{a_0,b_0}+\widetilde{\E}_{t_k,X_{t_{k}}^{a_0,b_0}}^{a_n,b(\ell)}[X_{t_{k}}^{a_0,b_0}R_5^{a_n,b(\ell)}\vert Y_{t_{k+1}}^{a_n,b(\ell)}=X_{t_{k+1}}^{a_0,b_0}]\big)d\ell.
	$$
	We use \eqref{for1} of Lemma \ref{change} with $V=X_{t_{k}}^{a_0,b_0}R_5^{a_n,b(\ell)}$, the fact that $\widehat{\E}_{t_k,X_{t_k}^{a_0,b_0}}^{a_0,b_0}[X_{t_{k}}^{a_0,b_0}R_7^{a_0,b_0}]=0$ and $\widetilde{\E}_{t_k,X_{t_{k}}^{a_0,b_0}}^{a_n,b(\ell)}[X_{t_{k}}^{a_0,b_0}R_5^{a_n,b(\ell)}]=0$, \eqref{for3} of Lemma \ref{deviation1} with $q=p=2$ and $q_1$ large enough, \eqref{sqrtincre} with $p=1$, and Young's inequality for products with $\frac{1}{\overline{p}_0}+\frac{1}{\overline{q}_0}=1$ and $\overline{q}_0$ close to $1$, under $\frac{a_0}{\sigma}>3$ for $n$ large enough $\vert\sum_{k=0}^{n-1}\widehat{\E}^{a_0,b_0}[\zeta_{k,n}\vert \widehat{\mathcal{F}}_{t_k}]\vert$ is bounded by
	\begin{align*} 
	&C\dfrac{\vert v\vert\varphi_{2,n}(a_0,b_0) }{\Delta_n}\sum_{k=0}^{n-1}\int_0^1\sqrt{\Delta_n}X_{t_{k}}^{a_0,b_0}\big(\widetilde{\E}_{t_k,X_{t_k}^{a_0,b_0}}^{a_n,b(\ell)}[\vert R_5^{a_n,b(\ell)}\vert^2]\big)^{\frac{1}{2}} e^{C_2v^2(\varphi_{2,n}(a_0,b_0))^2\Delta_nX_{t_k}^{a_0,b_0}}\\
	&\qquad\times \big(\vert u\vert\varphi_{1,n}(a_0,b_0)\big(\frac{1}{\sqrt{X_{t_k}^{a_0,b_0}}}+\frac{\sqrt{\Delta_n}}{X_{t_k}^{a_0,b_0}}\vert u\vert\varphi_{1,n}(a_0,b_0)+\sqrt{\Delta_n}\vert v\vert\varphi_{2,n}(a_0,b_0)\big)   \\
	&\qquad\times (1+\dfrac{1}{(X_{t_k}^{a_0,b_0})^{\frac{1}{16q_1}(\frac{a_n}{\sigma}-1)}})+\vert v\vert\varphi_{2,n}(a_0,b_0)(1+\sqrt{X_{t_k}^{a_0,b_0}}+\sqrt{\Delta_n}\vert v\vert\varphi_{2,n}(a_0,b_0)X_{t_k}^{a_0,b_0}) \big)d\ell\\
	&\leq C\vert v\vert\varphi_{2,n}(a_0,b_0)\Delta_n^{\frac{3}{2}}\frac{1}{\overline{p}_0}\big(\vert u\vert\varphi_{1,n}(a_0,b_0)+\vert v\vert\varphi_{2,n}(a_0,b_0)\big)\sum_{k=0}^{n-1}e^{C_2\overline{p}_0v^2(\varphi_{2,n}(a_0,b_0))^2\Delta_nX_{t_k}^{a_0,b_0}}\\
	&\qquad+C\vert v\vert\varphi_{2,n}(a_0,b_0)\Delta_n^{\frac{3}{2}}\frac{1}{\overline{q}_0}\sum_{k=0}^{n-1}\Big(\vert u\vert\varphi_{1,n}(a_0,b_0)\big(1+(\vert b_0\vert+\vert v\vert\varphi_{2,n}(a_0,b_0))^{\overline{q}_0}\\
	&\qquad\times(\sqrt{\Delta_n}\vert v\vert\varphi_{2,n}(a_0,b_0))^{\overline{q}_0}(X_{t_k}^{a_0,b_0})^{\frac{\overline{q}_0}{2}}+\dfrac{1}{(X_{t_k}^{a_0,b_0})^{\overline{q}_0(1+\frac{1}{16q_1}(\frac{a_n}{\sigma}-1))}}+\frac{(\sqrt{\Delta_n})^{\overline{q}_0}}{(X_{t_k}^{a_0,b_0})^{\overline{q}_0(\frac{3}{2}+\frac{1}{16q_1}(\frac{a_n}{\sigma}-1))}}\\
	&\qquad    \times(\vert u\vert\varphi_{1,n}(a_0,b_0))^{\overline{q}_0}\big) +\vert v\vert\varphi_{2,n}(a_0,b_0)\big(1+\frac{1}{(X_{t_k}^{a_0,b_0})^{\frac{\overline{q}_0}{2}}}+(\vert b_0\vert+\vert v\vert\varphi_{2,n}(a_0,b_0))^{\overline{q}_0}\\
	& \qquad\times(X_{t_k}^{a_0,b_0})^{\overline{q}_0}+(\vert b_0\vert+\vert v\vert\varphi_{2,n}(a_0,b_0))^{\overline{q}_0}(\sqrt{\Delta_n}\vert v\vert\varphi_{2,n}(a_0,b_0))^{\overline{q}_0}(X_{t_k}^{a_0,b_0})^{\frac{3}{2}\overline{q}_0}\big)\Big).
	\end{align*}
	Thus, for subcritical case $b_0>0$ (respectively critical case $b_0=0$, respectively supercritical case $b_0<0$), using \eqref{boundedexp}, Lemma \ref{moment2} \textnormal{(i)} (respectively Lemma \ref{moment2} \textnormal{(ii)}, respectively Lemma \ref{moment2} \textnormal{(iii)}, \eqref{supercriticalconvergence2}) and standard calculations, we get $\widehat{\E}^{a_0,b_0}[\vert\sum_{k=0}^{n-1}\widehat{\E}^{a_0,b_0}[\zeta_{k,n}\vert \widehat{\mathcal{F}}_{t_k}]\vert]\leq C\sqrt{\Delta_n}$ (respectively $\widehat{\E}^{a_0,0}[\vert\sum_{k=0}^{n-1}\widehat{\E}^{a_0,0}[\zeta_{k,n}\vert \widehat{\mathcal{F}}_{t_k}]\vert]\leq C\sqrt{\Delta_n}$, respectively $\widehat{\E}^{a_0,b_0}[\vert\sum_{k=0}^{n-1}\widehat{\E}^{a_0,b_0}[\zeta_{k,n}\vert \widehat{\mathcal{F}}_{t_k}]\vert]\leq C(\Delta_n^{2}e^{-b_0n\Delta_n})^{\frac{1}{4}}$), which tends to zero, since all the obtained powers are smaller than $\frac{a_0}{\sigma}$ for $\overline{q}_0$ close to $1$ and $q_1$ and $n$ large enough.
	
	Next, using Jensen's inequality, \eqref{for1} of Lemma \ref{change} with $V=(X_{t_{k}}^{a_0,b_0}R_5^{a_n,b(\ell)})^2$, \eqref{for3} of Lemma \ref{deviation1} with $q=p=2$ and $q_1$ large enough, \eqref{sqrtincre} with $p\in\{1,2\}$, and Young's inequality for products with $\frac{1}{\overline{p}_0}+\frac{1}{\overline{q}_0}=1$ and $\overline{q}_0$ close to $1$, under $\frac{a_0}{\sigma}>3$ for $n$ large enough $\sum_{k=0}^{n-1}\widehat{\E}^{a_0,b_0}[\zeta_{k,n}^2\vert \widehat{\mathcal{F}}_{t_k}]$ is bounded by
	\begin{align*}
	& 2\dfrac{v^2(\varphi_{2,n}(a_0,b_0))^2}{\Delta_n^2}\sum_{k=0}^{n-1}\int_0^1(X_{t_{k}}^{a_0,b_0})^2\Big\{\widehat{\E}_{t_k,X_{t_k}^{a_0,b_0}}^{a_0,b_0}\big[(R_7^{a_0,b_0})^2\big]+\widetilde{\E}_{t_k,X_{t_k}^{a_0,b_0}}^{a_n,b(\ell)}
	[(R_5^{a_n,b(\ell)})^2]\\
	&\qquad+C_1\sqrt{\Delta_n}(\widetilde{\E}_{t_k,X_{t_k}^{a_0,b_0}}^{a_n,b(\ell)}[\vert R_5^{a_n,b(\ell)}\vert^4])^{\frac{1}{2}} e^{C_2v^2(\varphi_{2,n}(a_0,b_0))^2\Delta_nX_{t_k}^{a_0,b_0}}\Big(\vert u\vert\varphi_{1,n}(a_0,b_0)\big(\frac{1}{\sqrt{X_{t_k}^{a_0,b_0}}}\\
	& \qquad+\frac{\sqrt{\Delta_n}}{X_{t_k}^{a_0,b_0}}\vert u\vert\varphi_{1,n}(a_0,b_0)+\sqrt{\Delta_n}\vert v\vert\varphi_{2,n}(a_0,b_0)\big)\big(1+\dfrac{1}{(X_{t_k}^{a_0,b_0})^{\frac{1}{16q_1}(\frac{a_n}{\sigma}-1)}}\big)   \\
	&\qquad+\vert v\vert\varphi_{2,n}(a_0,b_0)\big(1+\sqrt{X_{t_k}^{a_0,b_0}}+\sqrt{\Delta_n}\vert v\vert\varphi_{2,n}(a_0,b_0)X_{t_k}^{a_0,b_0}\big) \Big)\Big\}d\ell\\
	&\leq Cv^2(\varphi_{2,n}(a_0,b_0))^2\Delta_n^2\sum_{k=0}^{n-1}\big(\dfrac{1}{X_{t_k}^{a_0,b_0}}+X_{t_k}^{a_0,b_0}\big)+C(\varphi_{2,n}(a_0,b_0))^{2}\Delta_n^{\frac{5}{2}}\frac{1}{\overline{p}_0}\big(\vert u\vert\varphi_{1,n}(a_0,b_0)\\
	&\qquad+\vert v\vert\varphi_{2,n}(a_0,b_0)\big)\sum_{k=0}^{n-1}e^{C_2\overline{p}_0v^2(\varphi_{2,n}(a_0,b_0))^2\Delta_nX_{t_k}^{a_0,b_0}}+C(\varphi_{2,n}(a_0,b_0))^{2}\Delta_n^{\frac{5}{2}}\frac{1}{\overline{q}_0}\sum_{k=0}^{n-1}\Big(\varphi_{1,n}(a_0,b_0)\\
	&\qquad\times\big(1+(\vert b_0\vert+\vert v\vert\varphi_{2,n}(a_0,b_0))^{\overline{q}_0}(X_{t_k}^{a_0,b_0})^{\frac{\overline{q}_0}{2}}+(\vert b_0\vert+\vert v\vert\varphi_{2,n}(a_0,b_0))^{\overline{q}_0}(\sqrt{\Delta_n}\vert v\vert\varphi_{2,n}(a_0,b_0))^{\overline{q}_0}\\
	&\qquad  \times (X_{t_k}^{a_0,b_0})^{\overline{q}_0} +\dfrac{1}{(X_{t_k}^{a_0,b_0})^{\overline{q}_0(\frac{3}{2}+\frac{1}{16q_1}(\frac{a_n}{\sigma}-1))}} +\frac{(\sqrt{\Delta_n})^{\overline{q}_0}}{(X_{t_k}^{a_0,b_0})^{\overline{q}_0(2+\frac{1}{16q_1}(\frac{a_n}{\sigma}-1))}}(\vert u\vert\varphi_{1,n}(a_0,b_0))^{\overline{q}_0}\big)\\
	&\qquad+\varphi_{2,n}(a_0,b_0)\big(1+\frac{1}{(X_{t_k}^{a_0,b_0})^{\overline{q}_0}}+(\vert b_0\vert+\vert v\vert\varphi_{2,n}(a_0,b_0))^{\overline{q}_0}(X_{t_k}^{a_0,b_0})^{\frac{3}{2}\overline{q}_0}\\
	&\qquad+(\vert b_0\vert+\vert v\vert\varphi_{2,n}(a_0,b_0))^{\overline{q}_0}(\sqrt{\Delta_n}\varphi_{2,n}(a_0,b_0))^{\overline{q}_0}(X_{t_k}^{a_0,b_0})^{2\overline{q}_0}\big) \Big).
	\end{align*}
	Thus, for subcritical case $b_0>0$ (respectively critical case $b_0=0$, respectively supercritical case $b_0<0$), using \eqref{boundedexp}, Lemma \ref{moment2} \textnormal{(i)} (respectively Lemma \ref{moment2} \textnormal{(ii)}, respectively Lemma \ref{moment2} \textnormal{(iii)}, \eqref{supercriticalconvergence2}) and standard calculations, we get $\widehat{\E}^{a_0,b_0}[\vert\sum_{k=0}^{n-1}\widehat{\E}^{a_0,b_0}[\zeta_{k,n}^2\vert \widehat{\mathcal{F}}_{t_k}]\vert]\leq C\Delta_n$ (respectively $\widehat{\E}^{a_0,0}[\vert\sum_{k=0}^{n-1}\widehat{\E}^{a_0,0}[\zeta_{k,n}^2\vert \widehat{\mathcal{F}}_{t_k}]\vert]\leq C\Delta_n$, respectively $\widehat{\E}^{a_0,b_0}[\vert\sum_{k=0}^{n-1}\widehat{\E}^{a_0,b_0}[\zeta_{k,n}^2\vert \widehat{\mathcal{F}}_{t_k}]\vert]\leq C(\Delta_n^{2}e^{-b_0n\Delta_n})^{\frac{3}{4}}$), which tends to zero, since all the obtained powers are smaller than $\frac{a_0}{\sigma}$ for $\overline{q}_0$ close to $1$ and $q_1$ and $n$ large enough. This completes the proof.	
\end{proof}

\begin{lemma}\label{lemma5} Assume condition $\frac{a_0}{\sigma}>3$. Then, as $n\to\infty$, we have
	\begin{align*}
	\sum_{k=0}^{n-1}\dfrac{v\varphi_{2,n}(a_0,b_0)}{\Delta_n}\int_0^1\big(-X_{t_{k}}^{a_0,b_0}R_6^{a_0,b_0}+\widetilde{\E}_{t_k,X_{t_{k}}^{a_0,b_0}}^{a_n,b(\ell)}[X_{t_{k}}^{a_0,b_0}R_4^{a_n,b(\ell)}\vert Y_{t_{k+1}}^{a_n,b(\ell)}=X_{t_{k+1}}^{a_0,b_0}]\big)d\ell\overset{\widehat{\P}^{a_0,b_0}}{\longrightarrow}0.
	\end{align*}
\end{lemma}
\begin{proof} 	
	We rewrite 
	\begin{align*}
	&\dfrac{v\varphi_{2,n}(a_0,b_0)}{\Delta_n}\int_0^1(-X_{t_{k}}^{a_0,b_0}R_6^{a_0,b_0}+\widetilde{\E}_{t_k,X_{t_{k}}^{a_0,b_0}}^{a_n,b(\ell)}[X_{t_{k}}^{a_0,b_0}R_4^{a_n,b(\ell)}\vert Y_{t_{k+1}}^{a_n,b(\ell)}=X_{t_{k+1}}^{a_0,b_0}])d\ell\\
	&=M_{k,n,1}+M_{k,n,2},
	\end{align*}	
	where	
	\begin{align*}
	M_{k,n,1}&=-\dfrac{v^2}{2\sigma }(\varphi_{2,n}(a_0,b_0))^2\int_0^1\ell\int_{t_k}^{t_{k+1}}(X_s^{a_0,b_0}-X_{t_k}^{a_0,b_0})dsd\ell,\\
	M_{k,n,2}&=\dfrac{v}{2\sigma}\varphi_{2,n}(a_0,b_0)\int_0^1b(\ell)\Big\{\int_{t_k}^{t_{k+1}}(X_s^{a_0,b_0}-X_{t_k}^{a_0,b_0})ds\\
	&\qquad-\widetilde{\E}_{t_k,X_{t_{k}}^{a_0,b_0}}^{a_n,b(\ell)}\big[\int_{t_k}^{t_{k+1}}(Y_s^{a_n,b(\ell)}-Y_{t_k}^{a_n,b(\ell)})ds\vert Y_{t_{k+1}}^{a_n,b(\ell)}=X_{t_{k+1}}^{a_0,b_0}\big]\Big\}d\ell.
	\end{align*}
	First, for subcritical case $b_0>0$ with $\varphi_{2,n}(a_0,b_0)=\frac{1}{\sqrt{n\Delta_n}}$ (respectively critical case $b_0=0$ with $\varphi_{2,n}(a_0,0)=\frac{1}{n\Delta_n}$, respectively supercritical case $b_0<0$ with $\varphi_{2,n}(a_0,b_0)=e^{b_0\frac{n\Delta_n}{2}}$),  applying Lemma \ref{moment} \textnormal{(i)} (respectively Lemma \ref{moment} \textnormal{(ii)}, respectively equation \eqref{c2eq1}), we obtain that $\widehat{\E}^{a_0,b_0}[\vert\sum_{k=0}^{n-1} M_{k,n,1}\vert]\leq C\sqrt{\Delta_n}$, which tends to zero.
	Thus, $\sum_{k=0}^{n-1}M_{k,n,1}\overset{\widehat{\P}^{a_0,b_0}}{\longrightarrow}0$. 

Next, we are going to show that $\sum_{k=0}^{n-1}M_{k,n,2}\overset{\widehat{\P}^{a_0,b_0}}{\longrightarrow}0$ by applying Lemma \ref{zero}. For this, using \eqref{for1} of Lemma \ref{change} with $V=V^{a_n,b(\ell)}=\int_{t_k}^{t_{k+1}}(Y_s^{a_n,b(\ell)}-Y_{t_k}^{a_n,b(\ell)})ds$, \eqref{for2} of Lemma \ref{change} with $\widehat{V}=\widehat{V}^{a_0,b_0}=\int_{t_k}^{t_{k+1}}(X_s^{a_0,b_0}-X_{t_k}^{a_0,b_0})ds$, \eqref{for3} and \eqref{for4} of Lemma \ref{deviation1} with $q=p=2$ and $q_1$ large enough, \eqref{incrementY} with $p=1$, and Young's inequality for products with $\frac{1}{\overline{p}_0}+\frac{1}{\overline{q}_0}=1$ and $\overline{q}_0$ close to $1$, under $\frac{a_0}{\sigma}>3$ for $n$ large enough $\vert\sum_{k=0}^{n-1}\widehat{\E}^{a_0,b_0}[M_{k,n,2}\vert \widehat{\mathcal{F}}_{t_k}]\vert$ is bounded by
\begin{align*}
& \sum_{k=0}^{n-1}C\dfrac{\vert v\vert}{2\sigma}\varphi_{2,n}(a_0,b_0)\int_0^1\vert b(\ell)\vert \Big\{\frac{\sqrt{\Delta_n}}{\sqrt{X_{t_k}^{a_0,b_0}}} \big\vert\int_{a_n}^{a_0}\big(\widehat{\E}_{t_k,X_{t_k}^{a_0,b_0}}^{\alpha,b_0}\big[\vert \widehat{V}^{\alpha,b_0}\vert^2\big]\big)^{\frac{1}{2}}d\alpha\big\vert \\
&\qquad+\sqrt{\Delta_n}(1+\sqrt{X_{t_k}^{a_0,b_0}})\big\vert\int_{b(\ell)}^{b_0}(\widehat{\E}_{t_k,X_{t_k}^{a_0,b_0}}^{a_n,\beta}[\vert \widehat{V}^{a_n,\beta}\vert^2])^{\frac{1}{2}}d\beta\big\vert+C_1\sqrt{\Delta_n}\big(\widetilde{\E}_{t_k,X_{t_k}^{a_0,b_0}}^{a_n,b(\ell)}[\vert V^{a_n,b(\ell)}\vert^2]\big)^{\frac{1}{2}}\\
&\qquad \times e^{C_2v^2(\varphi_{2,n}(a_0,b_0))^2\Delta_nX_{t_k}^{a_0,b_0}}\Big(\vert u\vert\varphi_{1,n}(a_0,b_0)\big(\frac{1}{\sqrt{X_{t_k}^{a_0,b_0}}}+\frac{\sqrt{\Delta_n}}{X_{t_k}^{a_0,b_0}}\vert u\vert\varphi_{1,n}(a_0,b_0)\\
&\qquad+\sqrt{\Delta_n}\vert v\vert\varphi_{2,n}(a_0,b_0)\big)\big(1+\dfrac{1}{(X_{t_k}^{a_0,b_0})^{\frac{1}{16q_1}(\frac{a_n}{\sigma}-1)}}\big) +\vert v\vert\varphi_{2,n}(a_0,b_0)\big(1+\sqrt{X_{t_k}^{a_0,b_0}} \\
& \qquad +\sqrt{\Delta_n}\vert v\vert\varphi_{2,n}(a_0,b_0)X_{t_k}^{a_0,b_0}\big) \Big)\Big\}d\ell\\
&\leq C\varphi_{2,n}(a_0,b_0)\Delta_n^{2}\sum_{k=0}^{n-1} \Big( \vert u\vert\varphi_{1,n}(a_0,b_0)\big(\frac{1}{\sqrt{X_{t_k}^{a_0,b_0}}}+ \sqrt{X_{t_k}^{a_0,b_0}}\big)+\vert v\vert\varphi_{2,n}(a_0,b_0)(1+(X_{t_k}^{a_0,b_0})^{\frac{3}{2}})\Big)\\
&\qquad+C\varphi_{2,n}(a_0,b_0)\Delta_n^{2}\frac{1}{\overline{p}_0}\big(\vert u\vert\varphi_{1,n}(a_0,b_0)+\vert v\vert\varphi_{2,n}(a_0,b_0)\big)\sum_{k=0}^{n-1}e^{C_2\overline{p}_0v^2(\varphi_{2,n}(a_0,b_0))^2\Delta_nX_{t_k}^{a_0,b_0}}\\
&\qquad+C\varphi_{2,n}(a_0,b_0)\Delta_n^{2}\frac{1}{\overline{q}_0}\sum_{k=0}^{n-1}\Big(\varphi_{1,n}(a_0,b_0)\big(1+(\vert b_0\vert+\vert v\vert\varphi_{2,n}(a_0,b_0))^{\overline{q}_0}(X_{t_k}^{a_0,b_0})^{\frac{\overline{q}_0}{2}}\\
&\qquad+(\vert b_0\vert+\vert v\vert\varphi_{2,n}(a_0,b_0))^{\overline{q}_0}(\sqrt{\Delta_n}\vert v\vert\varphi_{2,n}(a_0,b_0))^{\overline{q}_0}(X_{t_k}^{a_0,b_0})^{\overline{q}_0} +\dfrac{1}{(X_{t_k}^{a_0,b_0})^{\overline{q}_0(\frac{1}{2}+\frac{1}{16q_1}(\frac{a_n}{\sigma}-1))}}\\
&\qquad    +\frac{(\sqrt{\Delta_n})^{\overline{q}_0}}{(X_{t_k}^{a_0,b_0})^{\overline{q}_0(1+\frac{1}{16q_1}(\frac{a_n}{\sigma}-1))}}(\vert u\vert\varphi_{1,n}(a_0,b_0))^{\overline{q}_0}\big)+\varphi_{2,n}(a_0,b_0)\big(1+(\vert b_0\vert+\vert v\vert\varphi_{2,n}(a_0,b_0))^{\overline{q}_0}\\
&\qquad\times (X_{t_k}^{a_0,b_0})^{\frac{3}{2}\overline{q}_0}+(\vert b_0\vert+\vert v\vert\varphi_{2,n}(a_0,b_0))^{\overline{q}_0}(\sqrt{\Delta_n}\varphi_{2,n}(a_0,b_0))^{\overline{q}_0}(X_{t_k}^{a_0,b_0})^{2\overline{q}_0}\big) \Big),
\end{align*}
where we use $\widehat{\E}_{t_k,X_{t_k}^{a_0,b_0}}^{a_n,b(\ell)}[\widehat{V}^{a_n,b(\ell)}]-\widetilde{\E}_{t_k,X_{t_k}^{a_0,b_0}}^{a_n,b(\ell)}[V^{a_n,b(\ell)}]=0$ since $Y^{a_n,b(\ell)}$ is the independent copy of $X^{a_n,b(\ell)}$. 

Thus, for subcritical case $b_0>0$ (respectively critical case $b_0=0$, respectively supercritical case $b_0<0$), using \eqref{boundedexp}, Lemma \ref{moment2} \textnormal{(i)} (respectively Lemma \ref{moment2} \textnormal{(ii)}, respectively Lemma \ref{moment2} \textnormal{(iii)}, \eqref{supercriticalconvergence2}) and standard calculations,  $\widehat{\E}^{a_0,b_0}[\vert\sum_{k=0}^{n-1}\widehat{\E}^{a_0,b_0}[M_{k,n,2}\vert \widehat{\mathcal{F}}_{t_k}]\vert]\leq C\Delta_n$ (respectively $\widehat{\E}^{a_0,0}[\vert\sum_{k=0}^{n-1}\widehat{\E}^{a_0,0}[M_{k,n,2}\vert \widehat{\mathcal{F}}_{t_k}]\vert]\leq C\sqrt{n\Delta_n^3}$, respectively $\widehat{\E}^{a_0,b_0}[\vert\sum_{k=0}^{n-1}\widehat{\E}^{a_0,b_0}[M_{k,n,2}\vert \widehat{\mathcal{F}}_{t_k}]\vert]\leq C(\Delta_n^{2}e^{-b_0n\Delta_n})^{\frac{1}{2}}$), which tends to zero, since all the obtained powers are smaller than $\frac{a_0}{\sigma}$ for $\overline{q}_0$ close to $1$ and $q_1$ and $n$ large enough.

Next, using Jensen's inequality, \eqref{for1} of Lemma \ref{change} with $(V^{a_n,b(\ell)})^2$, \eqref{for3} of Lemma \ref{deviation1} with $q=p=2$ and $q_1$ large enough, \eqref{incrementY} with $p\in\{1,2\}$, and the Young's inequality for products with $\frac{1}{\overline{p}_0}+\frac{1}{\overline{q}_0}=1$ and $\overline{q}_0$ close to $1$, under $\frac{a_0}{\sigma}>3$  for $n$ large enough $\sum_{k=0}^{n-1}\widehat{\E}^{a_0,b_0}[M_{k,n,2}^2\vert \widehat{\mathcal{F}}_{t_k}]$ is bounded by
\begin{align*}
& Cv^2(\varphi_{2,n}(a_0,b_0))^2\sum_{k=0}^{n-1}\int_0^1b(\ell)^2\Big\{\widehat{\E}_{t_k,X_{t_k}^{a_0,b_0}}^{a_0,b_0}[(\widehat{V}^{a_0,b_0})^2]+\widetilde{\E}_{t_k,X_{t_k}^{a_0,b_0}}^{a_n,b(\ell)}[(V^{a_n,b(\ell)})^2]\\
&\qquad+C_1\sqrt{\Delta_n}\big(\widetilde{\E}_{t_k,X_{t_k}^{a_0,b_0}}^{a_n,b(\ell)}[\vert V^{a_n,b(\ell)}\vert^4]\big)^{\frac{1}{2}} e^{C_2v^2(\varphi_{2,n}(a_0,b_0))^2\Delta_nX_{t_k}^{a_0,b_0}}\Big(\vert u\vert\varphi_{1,n}(a_0,b_0)\big(\frac{1}{\sqrt{X_{t_k}^{a_0,b_0}}}\\
&\qquad+\frac{\sqrt{\Delta_n}}{X_{t_k}^{a_0,b_0}}\vert u\vert\varphi_{1,n}(a_0,b_0)+\sqrt{\Delta_n}\vert v\vert\varphi_{2,n}(a_0,b_0)\big)\big(1+\dfrac{1}{(X_{t_k}^{a_0,b_0})^{\frac{1}{16q_1}(\frac{a_n}{\sigma}-1)}}\big)   \\
&\qquad +\vert v\vert\varphi_{2,n}(a_0,b_0)\big(1+\sqrt{X_{t_k}^{a_0,b_0}}+\sqrt{\Delta_n}\vert v\vert\varphi_{2,n}(a_0,b_0)X_{t_k}^{a_0,b_0}\big) \Big)\Big\}d\ell\\
&\leq Cv^2(\varphi_{2,n}(a_0,b_0))^2\Delta_n^3\sum_{k=0}^{n-1}\big(1+(X_{t_k}^{a_0,b_0})^{2}\big)\\
&\qquad+C(\varphi_{2,n}(a_0,b_0))^{2}\Delta_n^{\frac{7}{2}}\frac{1}{\overline{p}_0}\big(\vert u\vert\varphi_{1,n}(a_0,b_0)+\vert v\vert\varphi_{2,n}(a_0,b_0)\big)\sum_{k=0}^{n-1}e^{C_2\overline{p}_0v^2(\varphi_{2,n}(a_0,b_0))^2\Delta_nX_{t_k}^{a_0,b_0}}\\
&\qquad+C(\varphi_{2,n}(a_0,b_0))^{2}\Delta_n^{\frac{7}{2}}\frac{1}{\overline{q}_0}\sum_{k=0}^{n-1}\Big(\varphi_{1,n}(a_0,b_0)\big(1+(\vert b_0\vert+\vert v\vert\varphi_{2,n}(a_0,b_0))^{\overline{q}_0}(X_{t_k}^{a_0,b_0})^{\frac{3}{2}\overline{q}_0}\\
&\qquad+(\vert b_0\vert+\vert v\vert\varphi_{2,n}(a_0,b_0))^{\overline{q}_0}(\sqrt{\Delta_n}\vert v\vert\varphi_{2,n}(a_0,b_0))^{\overline{q}_0}(X_{t_k}^{a_0,b_0})^{2\overline{q}_0}+\dfrac{1}{(X_{t_k}^{a_0,b_0})^{\overline{q}_0(\frac{1}{2}+\frac{1}{16q_1}(\frac{a_n}{\sigma}-1))}}\\
&\qquad    +\frac{(\sqrt{\Delta_n})^{\overline{q}_0}}{(X_{t_k}^{a_0,b_0})^{\overline{q}_0(1+\frac{1}{16q_1}(\frac{a_n}{\sigma}-1))}}(\vert u\vert\varphi_{1,n}(a_0,b_0))^{\overline{q}_0}\big) +\varphi_{2,n}(a_0,b_0)\big(1+(\vert b_0\vert+\vert v\vert\varphi_{2,n}(a_0,b_0))^{\overline{q}_0}\\
&\qquad\times(X_{t_k}^{a_0,b_0})^{\frac{5}{2}\overline{q}_0}+(\vert b_0\vert+\vert v\vert\varphi_{2,n}(a_0,b_0))^{\overline{q}_0}(\sqrt{\Delta_n}\varphi_{2,n}(a_0,b_0))^{\overline{q}_0}(X_{t_k}^{a_0,b_0})^{3\overline{q}_0}\big) \Big).
\end{align*}
Thus, for subcritical case $b_0>0$ (respectively critical case $b_0=0$, respectively supercritical case $b_0<0$), using \eqref{boundedexp}, Lemma \ref{moment2} \textnormal{(i)} (respectively Lemma \ref{moment2} \textnormal{(ii)}, respectively Lemma \ref{moment2} \textnormal{(iii)}, \eqref{supercriticalconvergence2}) and standard calculations,  $\widehat{\E}^{a_0,b_0}[\vert\sum_{k=0}^{n-1}\widehat{\E}^{a_0,b_0}[M_{k,n,2}^2\vert \widehat{\mathcal{F}}_{t_k}]\vert]\leq C\Delta_n^2$ (respectively $\widehat{\E}^{a_0,0}[\vert\sum_{k=0}^{n-1}\widehat{\E}^{a_0,0}[M_{k,n,2}^2\vert \widehat{\mathcal{F}}_{t_k}]\vert]\leq Cn\Delta_n^3$ and respectively $\widehat{\E}^{a_0,b_0}[\vert\sum_{k=0}^{n-1}\widehat{\E}^{a_0,b_0}[M_{k,n,2}^2\vert \widehat{\mathcal{F}}_{t_k}]\vert]\leq C\Delta_n^{2}e^{-b_0n\Delta_n}$), which tends to zero, since all the obtained powers are smaller than $\frac{a_0}{\sigma}$ for $\overline{q}_0$ close to $1$ and $q_1$ and $n$ large enough. This completes the proof.
\end{proof}

\section{Appendix A: Proof of technical results}
\label{prooftechnical}
\subsection{Proof of Lemma \ref{estimates}}
\begin{proof}
	First, recall that the estimates \eqref{e1} and \eqref{e2} can be found in \cite[Lemma 2.1 and 3.1]{BD07} or \cite[Lemma 3.1]{DM11}. Next, we treat \eqref{e4}.
	
	For all $a, a_1\in\R^{\ast}_+$, $b, b_1\in\R$, $x\in\R^{\ast}_+$ and $k\in\{0,...,n-1\}$, by \cite[Lemma 3.1]{BP16}, for any $a>\sigma$ the probability measures $\widetilde{\P}_{t_k,x}^{a_1,b_1}$ and $\widetilde{\P}_{t_k,x}^{a,b}$ are absolutely continuous w.r.t. each other and its Radon-Nikodym derivative is given by
	\begin{align*}
	&\dfrac{d\widetilde{\P}_{t_k,x}^{a_1,b_1}}{d\widetilde{\P}_{t_k,x}^{a,b}}\big((Y_s^{a,b}(t_k,x))_{s\in I_k}\big)=\exp\Big\{\int_{t_k}^{t_{k+1}}\frac{a_1-a-(b_1-b)Y_s^{a,b}(t_k,x)}{2\sigma Y_s^{a,b}(t_k,x)}dY_s^{a,b}(t_k,x)\\
	&\qquad-\frac{1}{2}\int_{t_k}^{t_{k+1}}\frac{(a_1-b_1Y_s^{a,b}(t_k,x))^2-(a-bY_s^{a,b}(t_k,x))^2}{2\sigma Y_s^{a,b}(t_k,x)}ds\Big\}\\
	&=\exp\Big\{\int_{t_k}^{t_{k+1}}\frac{a_1-a-(b_1-b)Y_s^{a,b}(t_k,x)}{\sqrt{2\sigma Y_s^{a,b}(t_k,x)}}dW_s-\frac{1}{2}\int_{t_k}^{t_{k+1}}\frac{(a_1-a-(b_1-b)Y_s^{a,b}(t_k,x))^2}{2\sigma Y_s^{a,b}(t_k,x)}ds\Big\}.
	\end{align*}	
	Using \eqref{dxe} and the change of measures, we have that for any $p\in\R$,
	\begin{align}
	&\widetilde{\E}_{t_k,x}^{a,b}\big[\vert \partial_xY_t^{a,b}(t_k,x)\vert^p\big]=e^{-bp(t-t_k)}\widetilde{\E}_{t_k,x}^{a,b}\Big[\exp{\Big\{\frac{-p\sigma}{4}\int_{t_k}^t\frac{du}{Y_u^{a,b}(t_k,x)}+p\sqrt{\frac{\sigma}{2}}\int_{t_k}^t\dfrac{dW_u}{\sqrt{Y_u^{a,b}(t_k,x)}}\Big\}}\Big]\notag\\
	&=e^{-bp(t-t_k)}\widetilde{\E}_{t_k,x}^{a,b}\Big[\exp{\Big\{\dfrac{p(p-1)\sigma}{4}\int_{t_k}^t\dfrac{du}{Y_u^{a,b}(t_k,x)}\Big\}}\notag	\\
	&\qquad\times\exp{\Big\{p\sqrt{\dfrac{\sigma}{2}}\int_{t_k}^t\dfrac{dW_u}{\sqrt{Y_u^{a,b}(t_k,x)}}-\dfrac{p^2\sigma}{4}\int_{t_k}^t\dfrac{du}{Y_u^{a,b}(t_k,x)}\Big\}}\Big]\notag\\
	&=e^{-bp(t-t_k)}\widetilde{\E}_{t_k,x}^{a,b}\Big[\exp{\Big\{\dfrac{p(p-1)\sigma}{4}\int_{t_k}^t\dfrac{du}{Y_u^{a,b}(t_k,x)}\Big\}}\dfrac{d\widetilde{\P}_{t_k,x}^{a+p\sigma,b}}{d\widetilde{\P}_{t_k,x}^{a,b}}\big((Y_s^{a,b}(t_k,x))_{s\in [t_k,t]}\big)\Big]\notag\\
	&=e^{-bp(t-t_k)}\widetilde{\E}_{t_k,x}^{a+p\sigma,b}\Big[\exp{\Big\{\dfrac{p(p-1)\sigma}{4}\int_{t_k}^t\dfrac{du}{Y_u^{a+p\sigma,b}(t_k,x)}\Big\}}\Big].\label{measureq}
	\end{align}
	Then, applying Lemma \ref{expmoment} to the probability measure $\widetilde{\P}_{t_k,x}^{a+p\sigma,b}$ with $\mu=\frac{p(p-1)\sigma}{4}$, we get that for any $p\geq -\frac{(\frac{a}{\sigma}-1)^2}{2(\frac{a}{\sigma}-\frac{1}{2})}$ and $t\in[t_k,t_{k+1}]$, 
	\begin{align*}
	\widetilde{\E}_{t_k,x}^{a+p\sigma,b}\Big[\exp{\Big\{\dfrac{p(p-1)\sigma}{4}\int_{t_k}^t\dfrac{du}{Y_u^{a+p\sigma,b}(t_k,x)}\Big\}}\Big]\leq C_p\big(1+\dfrac{1}{x^{\frac{\frac{a}{\sigma}-1+p}{2}}}\big).
	\end{align*}
	This, together with \eqref{measureq}, gives the desired estimate \eqref{e4}.	
\end{proof}

\subsection{Proof of Lemma \ref{Malliderivable}}
\begin{proof}
	\textnormal{(i)} For $t_k\leq s\leq t \leq t_{k+1}$, using \eqref{expression2}, we have
	\begin{align*}
	&\partial_xY_t^{a,b}(t_k,x)\Big(\sqrt{\dfrac{\sigma}{2}}\dfrac{1}{\sqrt{Y_s^{a,b}(t_k,x)}}+\dfrac{\sigma}{4}\int_{s}^t\dfrac{1}{(Y_u^{a,b}(t_k,x))^2}D_sY_u^{a,b}(t_k,x)du\\
	&-\dfrac{1}{2}\sqrt{\dfrac{\sigma}{2}}\int_{s}^t\dfrac{1}{(Y_u^{a,b}(t_k,x))^{\frac{3}{2}}}D_sY_u^{a,b}(t_k,x)dW_u\Big){\bf 1}_{[t_k,t]}(s)=\partial_xY_t^{a,b}(t_k,x)\Big(\sqrt{\dfrac{\sigma}{2}}\dfrac{1}{\sqrt{Y_s^{a,b}(t_k,x)}}\\
	&+\dfrac{\sigma}{4}\int_{s}^t\dfrac{\sqrt{2\sigma}}{(Y_u^{a,b}(t_k,x))^{\frac{3}{2}}}\exp\Big\{\int_s^u\Big(-\dfrac{b}{2}-\big(\dfrac{a}{2}-\dfrac{\sigma}{4}\big)\dfrac{1}{Y_{\xi}^{a,b}(t_k,x)}\Big)d\xi\Big\}du\\
	&-\dfrac{\sigma}{2}\int_{s}^t\dfrac{1}{Y_u^{a,b}(t_k,x)}\exp\Big\{\int_s^u\Big(-\dfrac{b}{2}-\big(\dfrac{a}{2}-\dfrac{\sigma}{4}\big)\dfrac{1}{Y_{\xi}^{a,b}(t_k,x)}\Big)d\xi\Big\}dW_u\Big){\bf 1}_{[t_k,t]}(s).
	\end{align*}
	We will show that under condition $\frac{a}{\sigma}>4$, this expression is contained in $L^2(\widetilde{\Omega}\times [t_k,t_{k+1}])$. In fact, using the fact that the exponential terms can be bounded by a positive constant since $a>\sigma$, BDG's and H\"older's inequalities with $\frac{1}{p}+\frac{1}{q}=1$, \eqref{e2} and \eqref{e4}, we get that
	\begin{align*}
	&\widetilde{\E}_{t_k,x}^{a,b}\Big[\Big\vert \partial_xY_t^{a,b}(t_k,x)\Big(\sqrt{\dfrac{\sigma}{2}}\dfrac{1}{\sqrt{Y_s^{a,b}(t_k,x)}}+\dfrac{\sigma}{4}\int_{s}^t\dfrac{1}{(Y_u^{a,b}(t_k,x))^2}D_sY_u^{a,b}(t_k,x)du\\
	&\qquad-\dfrac{1}{2}\sqrt{\dfrac{\sigma}{2}}\int_{s}^t\dfrac{1}{(Y_u^{a,b}(t_k,x))^{\frac{3}{2}}}D_sY_u^{a,b}(t_k,x)dW_u\Big){\bf 1}_{[t_k,t]}(s)\Big\vert^2\Big]\\
	&\leq C\widetilde{\E}_{t_k,x}^{a,b}\Big[\Big\vert \frac{\partial_xY_t^{a,b}(t_k,x)}{\sqrt{Y_s^{a,b}(t_k,x)}}\Big\vert^2\Big]+C\widetilde{\E}_{t_k,x}^{a,b}\Big[\Big\vert \partial_xY_t^{a,b}(t_k,x)\int_{s}^t\frac{du}{(Y_u^{a,b}(t_k,x))^{\frac{3}{2}}}\Big\vert^2\Big]\\
	&\qquad+C\widetilde{\E}_{t_k,x}^{a,b}\Big[\Big\vert (\partial_xY_t^{a,b}(t_k,x))^2\int_{s}^t\frac{1}{(Y_u^{a,b}(t_k,x))^2}du\Big\vert\Big]\\
	&\leq C\Big(\widetilde{\E}_{t_k,x}^{a,b}\Big[\dfrac{1}{\vert Y_s^{a,b}(t_k,x)\vert^p}\Big]\Big)^{\frac{1}{p}}\Big(\widetilde{\E}_{t_k,x}^{a,b}\Big[\big\vert \partial_xY_t^{a,b}(t_k,x)\big\vert^{2q}\Big]\Big)^{\frac{1}{q}}\\
	&\qquad+C\Big(\widetilde{\E}_{t_k,x}^{a,b}\Big[\Big\vert \int_{s}^t\dfrac{du}{(Y_u^{a,b}(t_k,x))^{\frac{3}{2}}}\Big\vert^{2p}\Big]\Big)^{\frac{1}{p}}\Big(\widetilde{\E}_{t_k,x}^{a,b}\Big[\big\vert \partial_xY_t^{a,b}(t_k,x)\big\vert^{2q}\Big]\Big)^{\frac{1}{q}}\\
	&\qquad+C\Big(\widetilde{\E}_{t_k,x}^{a,b}\Big[\Big\vert \int_{s}^t\dfrac{du}{(Y_u^{a,b}(t_k,x))^2}\Big\vert^{p}\Big]\Big)^{\frac{1}{p}}\Big(\widetilde{\E}_{t_k,x}^{a,b}\Big[\big\vert \partial_xY_t^{a,b}(t_k,x)\big\vert^{2q}\Big]\Big)^{\frac{1}{q}}\\
	&\leq \dfrac{C}{x}\Big(1+\dfrac{1}{x^{\frac{\frac{a}{\sigma}-1}{2q}+1}}\Big)+C\Big(\Delta_n^{2p-1} \int_{s}^t\widetilde{\E}_{t_k,x}^{a,b}\Big[\dfrac{1}{\vert Y_u^{a,b}(t_k,x)\vert^{3p}}\Big]du\Big)^{\frac{1}{p}}\Big(1+\dfrac{1}{x^{\frac{\frac{a}{\sigma}-1}{2q}+1}}\Big)\\
	&\qquad+C\Big(\Delta_n^{p-1} \int_{s}^t\widetilde{\E}_{t_k,x}^{a,b}\Big[\dfrac{1}{\vert Y_u^{a,b}(t_k,x)\vert^{2p}}\Big]du\Big)^{\frac{1}{p}}\Big(1+\dfrac{1}{x^{\frac{\frac{a}{\sigma}-1}{2q}+1}}\Big)\leq \dfrac{C}{x}\Big(1+\dfrac{1}{x^{\frac{\frac{a}{\sigma}-1}{2q}+1}}\Big)\\
	&\qquad+C\Delta_n^2\dfrac{1}{x^3}\Big(1+\dfrac{1}{x^{\frac{\frac{a}{\sigma}-1}{2q}+1}}\Big)+C\Delta_n\dfrac{1}{x^2}\Big(1+\dfrac{1}{x^{\frac{\frac{a}{\sigma}-1}{2q}+1}}\Big)\leq C\Big(1+\dfrac{1}{x^{\frac{\frac{a}{\sigma}-1}{2q}+1}}\Big)\Big(\dfrac{1}{x}+\dfrac{1}{x^3}+\dfrac{1}{x^2}\Big),
	\end{align*}
	where $p>1$ and $p$ is close to $1$. Indeed, the condition required here is $3p<\frac{a}{\sigma}-1$ and $\frac{a}{\sigma}>4$, we only need to choose $p\in(1,\frac{1}{3}(\frac{a}{\sigma}-1))$. Thus, under $\frac{a}{\sigma}>4$,	
	\begin{align*}
	&\widetilde{\E}_{t_k,x}^{a,b}\Big[\int_{t_k}^{t_{k+1}}\Big\vert \partial_xY_t^{a,b}(t_k,x)\Big(\sqrt{\dfrac{\sigma}{2}}\dfrac{1}{\sqrt{Y_s^{a,b}(t_k,x)}}+\dfrac{\sigma}{4}\int_{s}^t\dfrac{1}{(Y_u^{a,b}(t_k,x))^2}D_sY_u^{a,b}(t_k,x)du\\
	&-\dfrac{1}{2}\sqrt{\dfrac{\sigma}{2}}\int_{s}^t\dfrac{1}{(Y_u^{a,b}(t_k,x))^{\frac{3}{2}}}D_sY_u^{a,b}(t_k,x)dW_u\Big)\Big\vert^2ds\Big]\\
	&\leq C\Delta_n\Big(1+\dfrac{1}{x^{\frac{\frac{a}{\sigma}-1}{2q}+1}}\Big)\Big(\dfrac{1}{x}+\dfrac{1}{x^3}+\dfrac{1}{x^2}\Big)<+\infty.
	\end{align*}	
	This combined with \eqref{e4} with $p=2$ guarantees that $\partial_xY_t^{a,b}(t_k,x)\in \mathbb{D}^{1,2}$ under $\frac{a}{\sigma}>4$ thanks to Theorem \ref{existenceMallideri}, and furthermore, its Malliavin derivative is given by \eqref{dmpx}. 
	
	Next, to prove \textnormal{(ii)} the following lemma will be useful.
	\begin{lemma}\label{Malliflow} Let $p>1$. Assume condition $\frac{a}{\sigma}>\frac{7p+4+\sqrt{p(49p+16)}}{4}$ for \eqref{dmif1} and \eqref{dmif2}, and condition $\frac{a}{\sigma}>p+1+\sqrt{p(p+1)}$ for \eqref{dmif3}. Then for any $(a,b)\in\Theta\times\Sigma$, $k \in \{0,...,n-1\}$, $t_k\leq s\leq t\leq t_{k+1}$, and $x\in\R^{\ast}_+$, there exists a constant $C>0$ which does not depend on $x$ such that 
		\begin{align}
		&\widetilde{\E}_{t_k,x}^{a,b}\Big[\Big\vert \frac{D_s(\partial_xY_t^{a,b}(t_k,x))}{(\partial_xY_t^{a,b}(t_k,x))^2}\Big\vert^p\Big]\leq C\Big(1+\dfrac{1}{x^{\frac{\frac{a}{\sigma}-1}{2p_0}-\frac{p}{2}}}\Big)\Big(\dfrac{1}{x^{\frac{p}{2}}}+\dfrac{1}{x^{\frac{3}{2}p}}+\dfrac{1}{x^{p}}\Big),\label{dmif1}\\
		&\widetilde{\E}_{t_k,x}^{a,b}\Big[\Big\vert \dfrac{D_sY_{t}^{a,b}(t_k,x)}{\partial_xY_{t}^{a,b}(t_k,x)}\Big\vert^p\Big]\leq C\Big(1+\dfrac{1}{x^{\frac{\frac{a}{\sigma}-1}{2p_2}-\frac{p}{2}}}\Big)\left(1+x^{\frac{p}{2}}\right),\label{dmif3}\\
		&\widetilde{\E}_{t_k,x}^{a,b}\Big[\Big\vert Y_{t}^{a,b}(t_k,x)D_s\Big(\dfrac{1}{\partial_xY_{t}^{a,b}(t_k,x)}\Big)\Big\vert^p\Big]\leq C\Big(1+\frac{1}{x^{\frac{\frac{a}{\sigma}-1}{2p_2p_0}-\frac{p}{2}}}\Big)(1+x^p)\Big(\frac{1}{x^{\frac{p}{2}}}+\frac{1}{x^{\frac{3}{2}p}}+\frac{1}{x^{p}}\Big),\label{dmif2}
		\end{align}
		where $p_0=\frac{21p+8+\sqrt{9p(49p+16)}}{8(3p+1)}$,  $p_2>1$ and $p_2$ is close to $1$.
	\end{lemma}
	\textnormal{(ii)} In the same way 
	we show that the expression $\frac{-1}{(\partial_xY_t^{a,b}(t_k,x))^2}D(\partial_xY_t^{a,b}(t_k,x))$ is contained in $L^2(\widetilde{\Omega}\times [t_k,t_{k+1}])$ under condition $\frac{a}{\sigma}>\frac{9+\sqrt{57}}{2}$. Indeed, applying \eqref{dmif1} of Lemma \ref{Malliflow} with $p=2$, condition $\frac{a}{\sigma}>\frac{9+\sqrt{57}}{2}$ ensures that
	\begin{align*}
	\widetilde{\E}_{t_k,x}^{a,b}\Big[\Big\vert \frac{D_s(\partial_xY_t^{a,b}(t_k,x))}{(\partial_xY_t^{a,b}(t_k,x))^2}\Big\vert^2\Big]&\leq C\Big(1+\dfrac{1}{x^{\frac{\frac{a}{\sigma}-1}{2p_0}-1}}\Big)\Big(\dfrac{1}{x}+\dfrac{1}{x^{3}}+\dfrac{1}{x^{2}}\Big),
	\end{align*}
	where $p_0=\frac{25+\sqrt{513}}{28}$. This implies that
	\begin{align*}
	&\widetilde{\E}_{t_k,x}^{a,b}\Big[\int_{t_k}^{t_{k+1}}\Big\vert \frac{D_s(\partial_xY_t^{a,b}(t_k,x))}{(\partial_xY_t^{a,b}(t_k,x))^2}\Big\vert^2ds\Big]\leq C\Delta_n\Big(1+\dfrac{1}{x^{\frac{\frac{a}{\sigma}-1}{2p_0}-1}}\Big)\Big(\dfrac{1}{x}+\dfrac{1}{x^{3}}+\dfrac{1}{x^{2}}\Big)<+\infty.
	\end{align*}
	Moreover, applying Lemma \ref{estimates} with $p=-2$, \eqref{e4} is satisfied under condition $\frac{a}{\sigma}>3+\sqrt{6}$. This guarantees that  $(\partial_xY_t^{a,b}(t_k,x))^{-1}\in \mathbb{D}^{1,2}$ under condition $\frac{a}{\sigma}>\frac{9+\sqrt{57}}{2}$ thanks to Lemma \ref{existenceMallideri} and, furthermore, its Malliavin derivative is given by \eqref{Mallinveflow1}. 
	
	Finally, in the same way 
	we show that under condition $\frac{a}{\sigma}>\frac{9+\sqrt{57}}{2}$,
	$$
	\dfrac{1}{\partial_xY_{t}^{a,b}(t_k,x)}DY_{t}^{a,b}(t_k,x)+Y_{t}^{a,b}(t_k,x)D\Big(\dfrac{1}{\partial_xY_{t}^{a,b}(t_k,x)}\Big)
	$$
	is contained in $L^2(\widetilde{\Omega}\times [t_k,t_{k+1}])$. Indeed, applying \eqref{dmif3} and \eqref{dmif2} of Lemma \ref{Malliflow} with $p=2$, condition $\frac{a}{\sigma}>\frac{9+\sqrt{57}}{2}$ ensures that
	\begin{align*}
	&\widetilde{\E}_{t_k,x}^{a,b}\Big[\Big\vert \dfrac{1}{\partial_xY_{t}^{a,b}(t_k,x)}D_sY_{t}^{a,b}(t_k,x)+Y_{t}^{a,b}(t_k,x)D_s\Big(\dfrac{1}{\partial_xY_{t}^{a,b}(t_k,x)}\Big)\Big\vert^2\Big]\\
	&\leq 2\widetilde{\E}_{t_k,x}^{a,b}\Big[\Big\vert \dfrac{1}{\partial_xY_{t}^{a,b}(t_k,x)}D_sY_{t}^{a,b}(t_k,x)\Big\vert^2\Big]+2\widetilde{\E}_{t_k,x}^{a,b}\Big[\Big\vert Y_{t}^{a,b}(t_k,x)D_s\Big(\dfrac{1}{\partial_xY_{t}^{a,b}(t_k,x)}\Big)\Big\vert^2\Big]\\
	&\leq C\Big(1+\dfrac{1}{x^{\frac{\frac{a}{\sigma}-1}{2p_2}-1}}\Big)(1+x)+C\Big(1+\frac{1}{x^{\frac{\frac{a}{\sigma}-1}{2p_2p_0}-1}}\Big)(1+x^2)\Big(\frac{1}{x}+\frac{1}{x^{3}}+\frac{1}{x^{2}}\Big),
	\end{align*}
	where $p_0=\frac{25+\sqrt{513}}{28}$, $p_2>1$ and $p_2$ is close to $1$. This implies that under condition $\frac{a}{\sigma}>\frac{9+\sqrt{57}}{2}$,
	\begin{align*}
	&\widetilde{\E}_{t_k,x}^{a,b}\Big[\int_{t_k}^{t_{k+1}}\Big\vert \dfrac{1}{\partial_xY_{t}^{a,b}(t_k,x)}D_sY_{t}^{a,b}(t_k,x)+Y_{t}^{a,b}(t_k,x)D_s\Big(\dfrac{1}{\partial_xY_{t}^{a,b}(t_k,x)}\Big)\Big\vert^2ds\Big]\\
	&\leq C\Delta_n\Big(1+\dfrac{1}{x^{\frac{\frac{a}{\sigma}-1}{2p_2}-1}}\Big)(1+x)+C\Delta_n\Big(1+\frac{1}{x^{\frac{\frac{a}{\sigma}-1}{2p_2p_0}-1}}\Big)(1+x^2)\Big(\frac{1}{x}+\frac{1}{x^{3}}+\frac{1}{x^{2}}\Big)<+\infty.
	\end{align*}
	On the other hand, using H\"older's inequality with $\frac{1}{p_1}+\frac{1}{q_1}=1$, we get
	\begin{align*} \widetilde{\E}_{t_k,x}^{a,b}\Big[\Big\vert\dfrac{Y_{t}^{a,b}(t_k,x)}{\partial_xY_{t}^{a,b}(t_k,x)}\Big\vert^2\Big]&\leq \Big(\widetilde{\E}_{t_k,x}^{a,b}\Big[\dfrac{1}{\vert\partial_xY_{t}^{a,b}(t_k,x)\vert^{2p_1}}\Big]\Big)^{\frac{1}{p_1}}\Big(\widetilde{\E}_{t_k,x}^{a,b}\Big[\vert Y_{t}^{a,b}(t_k,x)\vert^{2q_1}\Big]\Big)^{\frac{1}{q_1}}\\
	&\leq C\Big(1+\dfrac{1}{x^{\frac{\frac{a}{\sigma}-1}{2p_1}-1}}\Big)(1+x^2)<+\infty,
\end{align*}
where $p_1>1$ and $p_1$ is close to $1$. Indeed, the condition required here is $-2p_1\geq -\frac{(\frac{a}{\sigma}-1)^2}{2(\frac{a}{\sigma}-\frac{1}{2})}$ and $\frac{a}{\sigma}>3+\sqrt{6}$, thus we only need to choose $p_1\in(1,\frac{(\frac{a}{\sigma}-1)^2}{4(\frac{a}{\sigma}-\frac{1}{2})})$. Hence $\frac{Y_{t}^{a,b}(t_k,x)}{\partial_xY_{t}^{a,b}(t_k,x)}\in L^2(\widetilde{\Omega})$ under $\frac{a}{\sigma}>3+\sqrt{6}$. This guarantees thanks to Theorem \ref{existenceMallideri} that  $\frac{Y_{t}^{a,b}(t_k,x)}{\partial_xY_{t}^{a,b}(t_k,x)}\in \mathbb{D}^{1,2}$ under condition $\frac{a}{\sigma}>\frac{9+\sqrt{57}}{2}$ and, furthermore, its Malliavin derivative is given by \eqref{Mallinveflow2}. Thus, the result follows.
\end{proof}

\subsection{Proof of Lemma \ref{Malliflow}}
\begin{proof}
	First, we treat \eqref{dmif1}. For $t_k\leq s\leq t \leq t_{k+1}$, from \eqref{dmpx} and \eqref{expression2}, under $\frac{a}{\sigma}>4$ we have 
	\begin{align}\label{deriinver}
	&\frac{-1}{(\partial_xY_t^{a,b}(t_k,x))^2}D_s(\partial_xY_t^{a,b}(t_k,x))=-\frac{1}{\partial_xY_t^{a,b}(t_k,x)}\Big\{\sqrt{\frac{\sigma}{2}}\frac{1}{\sqrt{Y_s^{a,b}(t_k,x)}}\notag\\
	&\qquad+\dfrac{\sigma}{4}\int_{s}^{t}\dfrac{\sqrt{2\sigma}}{(Y_u^{a,b}(t_k,x))^{\frac{3}{2}}}\exp\Big\{\int_s^u\Big(-\dfrac{b}{2}-\big(\dfrac{a}{2}-\dfrac{\sigma}{4}\big)\dfrac{1}{Y_{\xi}^{a,b}(t_k,x)}\Big)d\xi\Big\}du\notag\\
	&\qquad-\dfrac{\sigma}{2}\int_{s}^{t}\dfrac{1}{Y_u^{a,b}(t_k,x)}\exp\Big\{\int_s^u\Big(-\dfrac{b}{2}-\big(\dfrac{a}{2}-\dfrac{\sigma}{4}\big)\dfrac{1}{Y_{\xi}^{a,b}(t_k,x)}\Big)d\xi\Big\}dW_u\Big\}.
	\end{align}	
	Then, using the fact that the exponential terms can be bounded by a positive constant since $a>\sigma$, BDG's and H\"older's inequalities with $\frac{1}{p_4}+\frac{1}{q_4}=1$, \eqref{e2} and \eqref{e4}, we get that
	\begin{align*}
	&\widetilde{\E}_{t_k,x}^{a,b}\Big[\Big\vert \frac{D_s(\partial_xY_t^{a,b}(t_k,x))}{(\partial_xY_t^{a,b}(t_k,x))^2}\Big\vert^{p}\Big]\leq C\widetilde{\E}_{t_k,x}^{a,b}\Big[\Big\vert \frac{1}{\partial_xY_t^{a,b}(t_k,x)}\dfrac{1}{\sqrt{Y_s^{a,b}(t_k,x)}}\Big\vert^{p}\Big]\\
	&+C\widetilde{\E}_{t_k,x}^{a,b}\Big[\Big\vert \frac{1}{\partial_xY_t^{a,b}(t_k,x)}\int_{s}^{t}\dfrac{du}{(Y_u^{a,b}(t_k,x))^{\frac{3}{2}}}\Big\vert^{p}\Big]+C\widetilde{\E}_{t_k,x}^{a,b}\Big[\Big\vert \frac{1}{(\partial_xY_t^{a,b}(t_k,x))^2}\int_{s}^{t}\dfrac{du}{(Y_u^{a,b}(t_k,x))^2}\Big\vert^{\frac{p}{2}}\Big]\\
	&\leq C\Big(\widetilde{\E}_{t_k,x}^{a,b}\Big[\frac{1}{\vert\partial_xY_t^{a,b}(t_k,x)\vert^{pp_4}}\Big]\Big)^{\frac{1}{p_4}}\Big(\widetilde{\E}_{t_k,x}^{a,b}\Big[\dfrac{1}{\vert Y_s^{a,b}(t_k,x)\vert^{\frac{1}{2}pq_4}}\Big]\Big)^{\frac{1}{q_4}}\\
	&\qquad+C\Big(\widetilde{\E}_{t_k,x}^{a,b}\Big[\frac{1}{\vert\partial_xY_t^{a,b}(t_k,x)\vert^{pp_4}}\Big]\Big)^{\frac{1}{p_4}}\Big(\widetilde{\E}_{t_k,x}^{a,b}\Big[\Big\vert\int_{s}^{t}\dfrac{du}{(Y_u^{a,b}(t_k,x))^{\frac{3}{2}}}\Big\vert^{pq_4}\Big]\Big)^{\frac{1}{q_4}}\\
	&\qquad+C\Big(\widetilde{\E}_{t_k,x}^{a,b}\Big[\frac{1}{\vert\partial_xY_t^{a,b}(t_k,x)\vert^{pp_4}}\Big]\Big)^{\frac{1}{p_4}}\Big(\widetilde{\E}_{t_k,x}^{a,b}\Big[\Big\vert\int_{s}^{t}\dfrac{du}{(Y_u^{a,b}(t_k,x))^2}\Big\vert^{\frac{1}{2}pq_4}\Big]\Big)^{\frac{1}{q_4}}\\
	&\leq C\Big(1+\dfrac{1}{x^{\frac{\frac{a}{\sigma}-1}{2p_4}-\frac{p}{2}}}\Big)\dfrac{1}{x^{\frac{p}{2}}}+C\Big(1+\dfrac{1}{x^{\frac{\frac{a}{\sigma}-1}{2p_4}-\frac{p}{2}}}\Big)\Big(\Delta_n^{pq_4-1}\int_{s}^{t}\widetilde{\E}_{t_k,x}^{a,b}\Big[\dfrac{1}{\vert Y_u^{a,b}(t_k,x)\vert^{\frac{3}{2}pq_4}}\Big]du\Big)^{\frac{1}{q_4}}\\
	&\qquad+C\Big(1+\dfrac{1}{x^{\frac{\frac{a}{\sigma}-1}{2p_4}-\frac{p}{2}}}\Big)\Big(\Delta_n^{\frac{1}{2}pq_4-1}\int_{s}^{t}\widetilde{\E}_{t_k,x}^{a,b}\Big[\dfrac{1}{\vert Y_u^{a,b}(t_k,x)\vert^{pq_4}}\Big]du\Big)^{\frac{1}{q_4}}\\
	&\leq C\Big(1+\dfrac{1}{x^{\frac{\frac{a}{\sigma}-1}{2p_4}-\frac{p}{2}}}\Big)\dfrac{1}{x^{\frac{p}{2}}}+C\Delta_n^{p}\Big(1+\dfrac{1}{x^{\frac{\frac{a}{\sigma}-1}{2p_4}-\frac{p}{2}}}\Big)\dfrac{1}{x^{\frac{3}{2}p}}+C\Delta_n^{\frac{p}{2}}\Big(1+\dfrac{1}{x^{\frac{\frac{a}{\sigma}-1}{2p_4}-\frac{p}{2}}}\Big)\dfrac{1}{x^{p}}\\
	&\leq C\Big(1+\dfrac{1}{x^{\frac{\frac{a}{\sigma}-1}{2p_4}-\frac{p}{2}}}\Big)\Big(\dfrac{1}{x^{\frac{p}{2}}}+\dfrac{1}{x^{\frac{3}{2}p}}+\dfrac{1}{x^{p}}\Big).
	\end{align*}
	Here, conditions are required as follows
	\begin{align*}
	-pp_4\geq -\frac{(\frac{a}{\sigma}-1)^2}{2(\frac{a}{\sigma}-\frac{1}{2})},\;
	\frac{3}{2}pq_4<\dfrac{a}{\sigma}-1.
	\end{align*}
	This implies that
	\begin{align*}
	\frac{a}{\sigma}\geq pp_4+\sqrt{pp_4\left(pp_4+1\right)}+1,\;
	\frac{a}{\sigma}>\frac{3}{2}pq_4+1=\frac{3}{2}\frac{pp_4}{p_4-1}+1.
	\end{align*}
	Hence, the optimal choice for $p_4$ is solution to
	$$
	pp_4+\sqrt{pp_4\left(pp_4+1\right)}=\frac{3}{2}\frac{pp_4}{p_4-1}.
	$$	
	The unique positive solution is given by $p_4=\frac{21p+8+\sqrt{9p(49p+16)}}{8(3p+1)}>1$. Thus, \eqref{dmif1} is valid under condition $\frac{a}{\sigma}>\frac{7p+4+\sqrt{p(49p+16)}}{4}$.
	
	Next, we treat \eqref{dmif3}. For $t_k\leq s\leq t \leq t_{k+1}$, using \eqref{expression2}, under condition $a\geq \sigma$ we have
	\begin{align*}
	\dfrac{1}{\partial_xY_{t}^{a,b}(t_k,x)}D_sY_{t}^{a,b}(t_k,x)=\frac{\sqrt{2\sigma Y_t^{a,b}(t_k,x)}}{\partial_xY_{t}^{a,b}(t_k,x)}\exp\Big\{\int_s^t\Big(-\dfrac{b}{2}-\big(\dfrac{a}{2}-\dfrac{\sigma}{4}\big)\dfrac{1}{Y_u^{a,b}(t_k,x)}\Big)du\Big\}
	\end{align*}	
	Then, using the fact that the exponential terms can be bounded by a positive constant since $a>\sigma$, H\"older's inequality with $\frac{1}{p_2}+\frac{1}{q_2}=1$, \eqref{e1} and \eqref{e4}, we get that
	\begin{align*}
	&\widetilde{\E}_{t_k,x}^{a,b}\Big[\Big\vert \dfrac{1}{\partial_xY_{t}^{a,b}(t_k,x)}D_sY_{t}^{a,b}(t_k,x)\Big\vert^p\Big]\leq C \widetilde{\E}_{t_k,x}^{a,b}\Big[\Big\vert \dfrac{\sqrt{Y_t^{a,b}(t_k,x)}}{\partial_xY_{t}^{a,b}(t_k,x)} \Big\vert^p\Big]\\
	&\leq C\Big(\widetilde{\E}_{t_k,x}^{a,b}\Big[\dfrac{1}{\vert\partial_xY_{t}^{a,b}(t_k,x)\vert^{pp_2}}\Big]\Big)^{\frac{1}{p_2}}\Big(\widetilde{\E}_{t_k,x}^{a,b}\Big[\big\vert\sqrt{Y_t^{a,b}(t_k,x)}\big\vert^{pq_2}\Big]\Big)^{\frac{1}{q_2}}\leq C\Big(1+\dfrac{1}{x^{\frac{\frac{a}{\sigma}-1}{2p_2}-\frac{p}{2}}}\Big)(1+x^{\frac{p}{2}}),
	\end{align*}
	where $p_2>1$ and $p_2$ is close to $1$. Indeed, the condition required here is $-pp_2\geq -\frac{(\frac{a}{\sigma}-1)^2}{2(\frac{a}{\sigma}-\frac{1}{2})}$ and $\frac{a}{\sigma}>p+1+\sqrt{p(p+1)}$, thus we only need to choose $p_2\in(1,\frac{(\frac{a}{\sigma}-1)^2}{2p(\frac{a}{\sigma}-\frac{1}{2})})$. Hence, \eqref{dmif3} is valid under condition $\frac{a}{\sigma}>p+1+\sqrt{p(p+1)}$.	
	
	Finally, we treat \eqref{dmif2}. For $t_k\leq s\leq t \leq t_{k+1}$, using \eqref{Mallinveflow1} and \eqref{deriinver}, under $\frac{a}{\sigma}>\frac{9+\sqrt{57}}{2}$, we have
	\begin{align}\label{deriinver2}
	&Y_{t}^{a,b}(t_k,x)D_s\Big(\dfrac{1}{\partial_xY_{t}^{a,b}(t_k,x)}\Big)=-\frac{Y_{t}^{a,b}(t_k,x)}{\partial_xY_t^{a,b}(t_k,x)}\bigg\{\sqrt{\dfrac{\sigma}{2}}\dfrac{1}{\sqrt{Y_s^{a,b}(t_k,x)}}\notag\\
	&\qquad+\dfrac{\sigma}{4}\int_{s}^{t}\dfrac{\sqrt{2\sigma}}{(Y_u^{a,b}(t_k,x))^{\frac{3}{2}}}\exp\Big\{\int_s^u\Big(-\dfrac{b}{2}-\big(\dfrac{a}{2}-\dfrac{\sigma}{4}\big)\dfrac{1}{Y_{\xi}^{a,b}(t_k,x)}\Big)d\xi\Big\}du\notag\\
	&\qquad-\dfrac{\sigma}{2}\int_{s}^{t}\dfrac{1}{Y_u^{a,b}(t_k,x)}\exp\Big\{\int_s^u\Big(-\dfrac{b}{2}-\big(\dfrac{a}{2}-\dfrac{\sigma}{4}\big)\dfrac{1}{Y_{\xi}^{a,b}(t_k,x)}\Big)d\xi\Big\}dW_u\bigg\}.
	\end{align}
	Then, using the fact that the exponential terms can be bounded by a positive constant since $a>\sigma$, BDG's and H\"older's inequalities with $\frac{1}{p_2}+\frac{1}{q_2}=1$ and $\frac{1}{p_3}+\frac{1}{q_3}=1$, \eqref{e1}, \eqref{e2} and \eqref{e4}, we get that
	\begin{align*}
	&\widetilde{\E}_{t_k,x}^{a,b}\Big[\Big\vert Y_{t}^{a,b}(t_k,x)D_s\Big(\dfrac{1}{\partial_xY_{t}^{a,b}(t_k,x)}\Big)\Big\vert^p\Big]\leq C\widetilde{\E}_{t_k,x}^{a,b}\Big[\Big\vert \frac{Y_{t}^{a,b}(t_k,x)}{\partial_xY_t^{a,b}(t_k,x)}\dfrac{1}{\sqrt{Y_s^{a,b}(t_k,x)}}\Big\vert^{p}\Big]\\
	&+C\widetilde{\E}_{t_k,x}^{a,b}\Big[\Big\vert \frac{Y_{t}^{a,b}(t_k,x)}{\partial_xY_t^{a,b}(t_k,x)}\int_{s}^{t}\dfrac{du}{(Y_u^{a,b}(t_k,x))^{\frac{3}{2}}}\Big\vert^{p}\Big]+C\widetilde{\E}_{t_k,x}^{a,b}\Big[\Big\vert \frac{(Y_{t}^{a,b}(t_k,x))^2}{(\partial_xY_t^{a,b}(t_k,x))^2}\int_{s}^{t}\dfrac{du}{(Y_u^{a,b}(t_k,x))^2}\Big\vert^{\frac{p}{2}}\Big]\\
	&\leq C\Big(\widetilde{\E}_{t_k,x}^{a,b}\Big[\Big\vert \frac{1}{\partial_xY_t^{a,b}(t_k,x)}\dfrac{1}{\sqrt{Y_s^{a,b}(t_k,x)}}\Big\vert^{pp_2}\Big]\Big)^{\frac{1}{p_2}}\Big(\widetilde{\E}_{t_k,x}^{a,b}\Big[\Big\vert Y_{t}^{a,b}(t_k,x)\Big\vert^{pq_2}\Big]\Big)^{\frac{1}{q_2}}\\
	&\qquad+C\Big(\widetilde{\E}_{t_k,x}^{a,b}\Big[\Big\vert \frac{1}{\partial_xY_t^{a,b}(t_k,x)}\int_{s}^{t}\dfrac{du}{(Y_u^{a,b}(t_k,x))^{\frac{3}{2}}}\Big\vert^{pp_2}\Big]\Big)^{\frac{1}{p_2}}\Big(\widetilde{\E}_{t_k,x}^{a,b}\Big[\Big\vert Y_{t}^{a,b}(t_k,x)\Big\vert^{pq_2}\Big]\Big)^{\frac{1}{q_2}}\\
	&\qquad+C\Big(\widetilde{\E}_{t_k,x}^{a,b}\Big[\Big\vert \frac{1}{(\partial_xY_t^{a,b}(t_k,x))^2}\int_{s}^{t}\dfrac{du}{(Y_u^{a,b}(t_k,x))^2}\Big\vert^{\frac{p}{2}p_2}\Big]\Big)^{\frac{1}{p_2}}\Big(\widetilde{\E}_{t_k,x}^{a,b}\Big[\Big\vert Y_{t}^{a,b}(t_k,x)\Big\vert^{pq_2}\Big]\Big)^{\frac{1}{q_2}}\\
	&\leq C\Big(\widetilde{\E}_{t_k,x}^{a,b}\Big[\frac{1}{\vert\partial_xY_t^{a,b}(t_k,x)\vert^{pp_2p_3}}\Big]\Big)^{\frac{1}{p_2p_3}}\Big(\widetilde{\E}_{t_k,x}^{a,b}\Big[\dfrac{1}{\vert Y_s^{a,b}(t_k,x)\vert^{\frac{1}{2}pp_2q_3}}\Big]\Big)^{\frac{1}{p_2q_3}}(1+x^p)\\
	&\qquad+C\Big(\widetilde{\E}_{t_k,x}^{a,b}\Big[\frac{1}{\vert\partial_xY_t^{a,b}(t_k,x)\vert^{pp_2p_3}}\Big]\Big)^{\frac{1}{p_2p_3}}\Big(\widetilde{\E}_{t_k,x}^{a,b}\Big[\Big\vert\int_{s}^{t}\dfrac{du}{(Y_u^{a,b}(t_k,x))^{\frac{3}{2}}}\Big\vert^{pp_2q_3}\Big]\Big)^{\frac{1}{p_2q_3}}(1+x^p)\\
	&\qquad+C\Big(\widetilde{\E}_{t_k,x}^{a,b}\Big[\frac{1}{\vert\partial_xY_t^{a,b}(t_k,x)\vert^{pp_2p_3}}\Big]\Big)^{\frac{1}{p_2p_3}}\Big(\widetilde{\E}_{t_k,x}^{a,b}\Big[\Big\vert\int_{s}^{t}\dfrac{du}{(Y_u^{a,b}(t_k,x))^2}\Big\vert^{\frac{p}{2}p_2q_3}\Big]\Big)^{\frac{1}{p_2q_3}}(1+x^p)\\
	&\leq C\Big(1+\dfrac{1}{x^{\frac{\frac{a}{\sigma}-1}{2p_2p_3}-\frac{p}{2}}}\Big)\dfrac{1}{x^{\frac{p}{2}}}(1+x^p)+C\Delta_n^{p}\Big(1+\dfrac{1}{x^{\frac{\frac{a}{\sigma}-1}{2p_2p_3}-\frac{p}{2}}}\Big)\dfrac{1}{x^{\frac{3}{2}p}}(1+x^p)\\
	&\qquad+C\Delta_n^{\frac{p}{2}}\Big(1+\dfrac{1}{x^{\frac{\frac{a}{\sigma}-1}{2p_2p_3}-\frac{p}{2}}}\Big)\dfrac{1}{x^{p}}(1+x^p)\leq C\Big(1+\frac{1}{x^{\frac{\frac{a}{\sigma}-1}{2p_2p_3}-\frac{p}{2}}}\Big)(1+x^p)\Big(\frac{1}{x^{\frac{p}{2}}}+\frac{1}{x^{\frac{3}{2}p}}+\frac{1}{x^{p}}\Big),
	\end{align*}
	where $p_2>1$ and $p_2$ is close to $1$. Here, conditions are required as follows
	\begin{align*}
	-pp_2p_3\geq -\frac{(\frac{a}{\sigma}-1)^2}{2(\frac{a}{\sigma}-\frac{1}{2})},\;
	\frac{3}{2}pp_2q_3<\dfrac{a}{\sigma}-1.
	\end{align*}
	This implies that
	\begin{align*}
	pp_3<\frac{(\frac{a}{\sigma}-1)^2}{2(\frac{a}{\sigma}-\frac{1}{2})},\;
	\frac{3}{2}pq_3<\dfrac{a}{\sigma}-1.
	\end{align*}
	Thus,
	\begin{align*}
	\frac{a}{\sigma}> pp_3+\sqrt{pp_3\left(pp_3+1\right)}+1,\;
	\frac{a}{\sigma}>\frac{3}{2}pq_3+1=\frac{3}{2}\frac{pp_3}{p_3-1}+1.
	\end{align*}
	Hence, the optimal choice for $p_3$ is solution to
	$$
	pp_3+\sqrt{pp_3\left(pp_3+1\right)}=\frac{3}{2}\frac{pp_3}{p_3-1}.
	$$	
	The unique positive solution is given by $p_3=\frac{21p+8+\sqrt{9p(49p+16)}}{8(3p+1)}>1$. Thus, \eqref{dmif2} is valid under condition $\frac{a}{\sigma}>\frac{7p+4+\sqrt{p(49p+16)}}{4}$.
\end{proof}	 

\subsection{Proof of Proposition \ref{c2prop1}}
\begin{proof}
	We are going to apply Theorem \ref{ruleproNualart}. First, we wish to show  $\partial_{a}Y_{t_{k+1}}^{a,b}(t_k,x)U^{a,b}(t_k,x)\in \textnormal{Dom}\ \delta$ under condition $\frac{a}{\sigma}>\frac{9+\sqrt{57}}{2}$. For this, we write 
	\begin{align}\label{Fu}
	\partial_{a}Y_{t_{k+1}}^{a,b}(t_k,x)U_{\cdot}^{a,b}(t_k,x)&=\frac{\partial_{a}Y_{t_{k+1}}^{a,b}(t_k,x)}{\sqrt{2\sigma Y_{\cdot}^{a,b}(t_k,x)}}(\partial_xY_{t_{k+1}}^{a,b}(t_k,x))^{-1}\partial_xY_{\cdot}^{a,b}(t_k,x)=Fu_{\cdot},
	\end{align}
	where $F=\partial_{a}Y_{t_{k+1}}^{a,b}(t_k,x)(\partial_xY_{t_{k+1}}^{a,b}(t_k,x))^{-1}$ and $u_{\cdot}=\frac{1}{\sqrt{2\sigma Y_{\cdot}^{a,b}(t_k,x)}}\partial_xY_{\cdot}^{a,b}(t_k,x)$. Here $u=(u_t, t\in [t_k,t_{k+1}])$ is an adapted process then it belongs to $\textnormal{Dom}\ \delta$. By Theorem \ref{ruleproNualart}, it suffices to show that $F\in \mathbb{D}^{1,2}$ and $Fu\in L^2(\widetilde{\Omega};H)\cong L^2([t_k,t_{k+1}]\times \widetilde{\Omega},\R)$, where $H=L^2([t_k,t_{k+1}],\R)$.
	
	From \eqref{dxa2}, we write
	\begin{align*}
	F=\dfrac{\partial_{a}Y_{t_{k+1}}^{a,b}(t_k,x)}{\partial_xY_{t_{k+1}}^{a,b}(t_k,x)}=\int_{t_k}^{t_{k+1}}\frac{dr}{\partial_xY_r^{a,b}(t_k,x)}.
	\end{align*}
	By assertion (ii) of Lemma \ref{Malliderivable}, under condition $\frac{a}{\sigma}>\frac{9+\sqrt{57}}{2}$, $(\partial_xY_t^{a,b}(t_k,x))^{-1}\in \mathbb{D}^{1,2}$ for any $t\in [t_k,t_{k+1}]$. Thus, it is straightforward that
	$F\in \mathbb{D}^{1,2}$ under condition $\frac{a}{\sigma}>\frac{9+\sqrt{57}}{2}$. Furthermore, for $t_k\leq s\leq r \leq t_{k+1}$,
	\begin{align*}
	D_sF=\int_{s}^{t_{k+1}}D_s\Big(\frac{1}{\partial_xY_r^{a,b}(t_k,x)}\Big)dr.
	\end{align*}	
	Next, we check $Fu\in L^2([t_k,t_{k+1}]\times \widetilde{\Omega},\R)$. For this, using H\"older's inequality repeatedly with $\frac{1}{p}+\frac{1}{q}=1$ and $\frac{1}{p_1}+\frac{1}{q_1}=1$, together with \eqref{e2}-\eqref{e4}, we get that
	\begin{align*}
	&\widetilde{\E}_{t_k,x}^{a,b}\big[(Fu_t)^2\big]=\widetilde{\E}_{t_k,x}^{a,b}\Big[\Big(\int_{t_k}^{t_{k+1}}\frac{dr}{\partial_xY_r^{a,b}(t_k,x)}\frac{1}{\sqrt{2\sigma Y_{t}^{a,b}(t_k,x)}}\partial_xY_{t}^{a,b}(t_k,x)\Big)^{2}\Big]\\
	&\leq\dfrac{1}{2\sigma}\Big(\widetilde{\E}_{t_k,x}^{a,b}\Big[\Big\vert\int_{t_k}^{t_{k+1}}\frac{dr}{\partial_xY_r^{a,b}(t_k,x)}\frac{1}{\sqrt{ Y_{t}^{a,b}(t_k,x)}}\Big\vert^{2p}\Big]\Big)^{\frac{1}{p}}\Big(\widetilde{\E}_{t_k,x}^{a,b}\Big[\big\vert \partial_xY_{t}^{a,b}(t_k,x)\big\vert^{2q}\Big]\Big)^{\frac{1}{q}}\\
	&\leq C\Big(\widetilde{\E}_{t_k,x}^{a,b}\Big[\Big\vert\int_{t_k}^{t_{k+1}}\frac{dr}{\partial_xY_r^{a,b}(t_k,x)}\Big\vert^{2pp_1}\Big]\Big)^{\frac{1}{pp_1}}\Big(\widetilde{\E}_{t_k,x}^{a,b}\Big[\frac{1}{\vert Y_{t}^{a,b}(t_k,x)\vert^{pq_1}}\Big]\Big)^{\frac{1}{pq_1}}\Big(1+\dfrac{1}{x^{\frac{\frac{a}{\sigma}-1+2q}{2}}}\Big)^{\frac{1}{q}}\\
	&\leq C\Big(\Delta_n^{2pp_1-1}\int_{t_k}^{t_{k+1}}\widetilde{\E}_{t_k,x}^{a,b}\Big[\frac{1}{\vert\partial_xY_r^{a,b}(t_k,x)\vert^{2pp_1}}\Big]dr\Big)^{\frac{1}{pp_1}}\dfrac{1}{x}\Big(1+\dfrac{1}{x^{\frac{\frac{a}{\sigma}-1}{2q}+1}}\Big)\\
	&\leq C\Delta_n^2\Big(1+\dfrac{1}{x^{\frac{\frac{a}{\sigma}-1}{2pp_1}-1}}\Big)\dfrac{1}{x}\Big(1+\dfrac{1}{x^{\frac{\frac{a}{\sigma}-1}{2q}+1}}\Big),
	\end{align*}
	for some constant $C>0$, where $p>1$ and $p$ is close to $1$. This shows that
	\begin{align*}
	\widetilde{\E}_{t_k,x}^{a,b}\Big[\int_{t_k}^{t_{k+1}}\left(Fu_t\right)^2dt\Big]&=\int_{t_k}^{t_{k+1}}\widetilde{\E}_{t_k,x}^{a,b}\Big[(Fu_t)^2\Big]dt<+\infty.
	\end{align*}
	All conditions required here are as follows
	\begin{align*}
	 -2pp_1\geq -\frac{(\frac{a}{\sigma}-1)^2}{2(\frac{a}{\sigma}-\frac{1}{2})},\;  pq_1<\dfrac{a}{\sigma}-1.
	\end{align*}
	That is,
	\begin{align*}
	-2p_1>-\frac{(\frac{a}{\sigma}-1)^2}{2(\frac{a}{\sigma}-\frac{1}{2})},\;  q_1<\dfrac{a}{\sigma}-1.
	\end{align*}
	This implies that
	\begin{align*}
	\frac{a}{\sigma}> 2p_1+\sqrt{2p_1\left(2p_1+1\right)}+1,\;
	\frac{a}{\sigma}>q_1+1=\frac{p_1}{p_1-1}+1.
	\end{align*}	
	Hence, the optimal choice for $p_1$ is solution to
	$$
	2p_1+\sqrt{2p_1\left(2p_1+1\right)}=\frac{p_1}{p_1-1}.
	$$	
	The unique positive solution is given by $p_1=\frac{9+\sqrt{33}}{12}$. Thus, $Fu\in L^2([t_k,t_{k+1}]\times \widetilde{\Omega},\R)$ under $\frac{a}{\sigma}>\frac{7+\sqrt{33}}{2}$. Hence, we have 
	shown that $\partial_{a}Y_{t_{k+1}}^{a,b}(t_k,x)U^{a,b}(t_k,x)\in \textnormal{Dom}\ \delta$ under condition $\frac{a}{\sigma}>\frac{9+\sqrt{57}}{2}$. 	
	
	Next, under condition $\frac{a}{\sigma}>\frac{9+\sqrt{57}}{2}$ we proceed as in the proof of Proposition 3.1 of \cite{KNT15} with $\beta=a$ (see pages 441 and 442) to get the following representation of the score function
	\begin{align*}
	\dfrac{\partial_{a}p^{a,b}}{p^{a,b}}\left(\Delta_n,x,y\right)=\dfrac{1}{\Delta_n}\widetilde{\E}_{t_k,x}^{a,b}\left[\delta\left(\partial_{a}Y_{t_{k+1}}^{a,b}(t_k,x)U^{a,b}(t_k,x)\right)\big\vert Y_{t_{k+1}}^{a,b}=y\right].
	\end{align*}	
	Now, we wish to show \eqref{deria}. In fact, using condition $\frac{a}{\sigma}>\frac{9+\sqrt{57}}{2}$ and the fact that the Skorohod integral and the It\^o integral of an adapted process coincide $\delta(u)=\int_{t_k}^{t_{k+1}}\frac{\partial_{x}Y_{s}^{a,b}(t_k,x)}{\sqrt{2\sigma Y_s^{a,b}(t_k,x)}}dW_s$, we have		
		\begin{equation}\label{decom1}\begin{split}
		&F\delta(u)-\left<DF,u\right>_H=\int_{t_k}^{t_{k+1}}\frac{dr}{\partial_xY_r^{a,b}(t_k,x)}\int_{t_k}^{t_{k+1}}\frac{\partial_{x}Y_{s}^{a,b}(t_k,x)}{\sqrt{2\sigma Y_s^{a,b}(t_k,x)}}dW_s\\
		&\qquad-\int_{t_k}^{t_{k+1}}\int_{s}^{t_{k+1}}D_s\Big(\dfrac{1}{\partial_xY_{r}^{a,b}(t_k,x)}\Big)dr\frac{\partial_{x}Y_{s}^{a,b}(t_k,x)}{\sqrt{2\sigma Y_s^{a,b}(t_k,x)}}ds.
		\end{split}
		\end{equation}
		We next add and subtract the term $\frac{\partial_{x}Y_{t_k}^{a,b}(t_k,x)}{\sqrt{2\sigma Y_{t_k}^{a,b}(t_k,x)}}$ in the second integral, and the term $\frac{1}{\partial_xY_{t_k}^{a,b}(t_k,x)}$ in the first integral. This, together with $Y_{t_k}^{a,b}(t_k,x)=x$, implies that 
		\begin{equation}\label{e0}
		F\delta(u)-\left<DF,u\right>_H
		=\frac{\Delta_n}{\sqrt{2\sigma x}} \left(W_{t_{k+1}}-W_{t_{k}}\right)+R_1^{a,b}+R_2^{a,b}+R_3^{a,b}.
		\end{equation}	
		Then, applying \eqref{es2} of Lemma \ref{estimate} with $q=1$, under condition $\frac{a}{\sigma}>5+3\sqrt{2}$, we get 
		\begin{align*}
		\widetilde{\E}_{t_k,x}^{a,b}\Big[\big(F\delta(u)-\left<DF,u\right>_H\big)^2\Big]&\leq \frac{\Delta_n^2}{\sigma x}\widetilde{\E}_{t_k,x}^{a,b}\Big[\big(W_{t_{k+1}}-W_{t_{k}}\big)^2\Big]+2\widetilde{\E}_{t_k,x}^{a,b}\Big[\big(R_1^{a,b}+R_2^{a,b}+R_3^{a,b}\big)^2\Big]\\
		&\leq \frac{\Delta_n^3}{\sigma x}+C	\Delta_n^{4}\sum_{\alpha(1)\in I}\dfrac{1}{x^{\alpha(1)}}<+\infty,
		\end{align*}
		where $I$ is a finite set. Thus, we have shown that $F\delta(u)-\left<DF,u\right>_H$ is square integrable under condition $\frac{a}{\sigma}>5+3\sqrt{2}$. Consequently, by Theorem \ref{ruleproNualart}, under condition $\frac{a}{\sigma}>5+3\sqrt{2}$ we have 
		\begin{align*}
		\delta(Fu)=F\delta(u)-\left<DF,u\right>_H.
		\end{align*}
		This, together with \eqref{Fu} and \eqref{e0}, gives \eqref{deria} under condition $\frac{a}{\sigma}>5+3\sqrt{2}$.	
			
		Finally, we write 
		\begin{align}\label{Fu2}
		\partial_{b}Y_{t_{k+1}}^{a,b}(t_k,x)U_{\cdot}^{a,b}(t_k,x)&=\frac{\partial_{b}Y_{t_{k+1}}^{a,b}(t_k,x)}{\sqrt{2\sigma Y_{\cdot}^{a,b}(t_k,x)}}(\partial_xY_{t_{k+1}}^{a,b}(t_k,x))^{-1}\partial_xY_{\cdot}^{a,b}(t_k,x)=\overline{F}u_{\cdot},
		\end{align}
		where $\overline{F}=\partial_{b}Y_{t_{k+1}}^{a,b}(t_k,x)(\partial_xY_{t_{k+1}}^{a,b}(t_k,x))^{-1}$. From \eqref{dxb2}, we have
		\begin{align*}
		\overline{F}=\dfrac{\partial_{b}Y_{t_{k+1}}^{a,b}(t_k,x)}{\partial_xY_{t_{k+1}}^{a,b}(t_k,x)}=-\int_{t_k}^{t_{k+1}}\dfrac{Y_{r}^{a,b}(t_k,x)}{\partial_xY_{r}^{a,b}(t_k,x)}dr.
		\end{align*}
		Then, by assertion (ii) of Lemma \ref{Malliderivable} and proceeding as above, we get that under condition $\frac{a}{\sigma}>\frac{9+\sqrt{57}}{2}$, $\partial_{b}Y_{t_{k+1}}^{a,b}(t_k,x)U^{a,b}(t_k,x)\in \textnormal{Dom}\ \delta$ and
			\begin{align*}
			\dfrac{\partial_{b}p^{a,b}}{p^{a,b}}\left(\Delta_n,x,y\right)=\dfrac{1}{\Delta_n}\widetilde{\E}_{t_k,x}^{a,b}\left[\delta\left(\partial_{b}Y_{t_{k+1}}^{a,b}(t_k,x)U^{a,b}(t_k,x)\right)\big\vert Y_{t_{k+1}}^{a,b}=y\right].
			\end{align*}
			Furthermore, using the same arguments as for \eqref{decom1}, we can easily check that under condition $\frac{a}{\sigma}>\frac{9+\sqrt{57}}{2}$,
			\begin{align*}
			&\overline{F}\delta(u)-\left<D\overline{F},u\right>_H=-\int_{t_k}^{t_{k+1}}\frac{Y_r^{a,b}(t_k,x)}{\partial_xY_r^{a,b}(t_k,x)}dr\int_{t_k}^{t_{k+1}}\frac{\partial_{x}Y_{s}^{a,b}(t_k,x)}{\sqrt{2\sigma Y_s^{a,b}(t_k,x)}}dW_s\\
			&\qquad+\int_{t_k}^{t_{k+1}}\int_{s}^{t_{k+1}}D_s\Big(\dfrac{Y_{r}^{a,b}(t_k,x)}{\partial_xY_{r}^{a,b}(t_k,x)}\Big)dr\frac{\partial_{x}Y_{s}^{a,b}(t_k,x)}{\sqrt{2\sigma Y_s^{a,b}(t_k,x)}}ds.
			\end{align*}
			We next add and subtract the term $
			\frac{Y_{t_k}^{a,b}(t_k,x)}{\partial_{x}Y_{t_k}^{a,b}(t_k,x)}$ in the first integral, and the term $\frac{\partial_{x}Y_{t_k}^{a,b}(t_k,x)}{\sqrt{2\sigma Y_{t_k}^{a,b}(t_k,x)}}$ in the second integral. This, together with $Y_{t_k}^{a,b}(t_k,x)=x$, shows that
			\begin{equation}\label{H} \begin{split}
			\overline{F}\delta(u)-\left<D\overline{F},u\right>_H=-\frac{\Delta_n}{\sqrt{2\sigma }}\sqrt{x} \left(W_{t_{k+1}}-W_{t_{k}}\right)-xR_1^{a,b}+H_2^{a,b}+H_3^{a,b}.
			\end{split}
			\end{equation}
			Then, applying \eqref{es4} of Lemma \ref{estimate} with $q=1$, we get that $\overline{F}\delta(u)-\left<D\overline{F},u\right>_H$ is square integrable under condition $\frac{a}{\sigma}>5+3\sqrt{2}$. Consequently, from \eqref{ruleproduct} of Theorem \ref{ruleproNualart}, \eqref{Fu2} and \eqref{H} we conclude \eqref{derib} under condition $\frac{a}{\sigma}>5+3\sqrt{2}$. Thus, the result follows.		
		\end{proof}

\subsection{Proof of Lemma \ref{estimate}}
\label{Alowerbound}
\begin{proof}
	{\it Proof of \eqref{es1} and \eqref{es3}.} These facts follow easily from \eqref{deria}, \eqref{derib}, and properties of the moment of the Skorohod integral and the Brownian motion.
	\vskip 5pt
	
	{\it Proof of \eqref{es2}.} Observe that
	\begin{equation}\label{r}
	\begin{split}
	&\widetilde{\E}_{t_k,x}^{a,b}\Big[\vert R_1^{a,b}+R_2^{a,b}+R_3^{a,b}\vert^{2q}\Big]\leq 3^{2q-1}\Big(\widetilde{\E}_{t_k,x}^{a,b}[\vert R_1^{a,b}\vert^{2q}]+\widetilde{\E}_{t_k,x}^{a,b}[\vert R_2^{a,b}\vert^{2q}]+\widetilde{\E}_{t_k,x}^{a,b}[\vert R_3^{a,b}\vert^{2q}]\Big).
	\end{split}
	\end{equation}
	First, we treat the term $R_1^{a,b}$. Using BDG's inequality, we have that
	\begin{align*}
	\widetilde{\E}_{t_k,x}^{a,b}[\vert R_1^{a,b}\vert^{2q}]&\leq  C\Delta_n^{2q} \widetilde{\E}_{t_k,x}^{a,b}\Big[\Big\vert\int_{t_k}^{t_{k+1}}\Big(\frac{\partial_{x}Y_{s}^{a,b}(t_k,x)}{\sqrt{Y_s^{a,b}(t_k,x)}}-\frac{\partial_{x}Y_{t_k}^{a,b}(t_k,x)}{\sqrt{Y_{t_k}^{a,b}(t_k,x)}}\Big)^2ds\Big\vert^q\Big]\\
	&\leq C\Delta_n^{2q} \Delta_n^{q-1}\int_{t_k}^{t_{k+1}}R_{11}^{a,b}ds,
	\end{align*}
	where 
	$$
	R_{11}^{a,b}=\widetilde{\E}_{t_k,x}^{a,b}\Big[\Big\vert\frac{\partial_{x}Y_{s}^{a,b}(t_k,x)}{\sqrt{Y_s^{a,b}(t_k,x)}}-\frac{\partial_{x}Y_{t_k}^{a,b}(t_k,x)}{\sqrt{Y_{t_k}^{a,b}(t_k,x)}}\Big\vert^{2q}\Big].
	$$
	By It\^o's formula, it can be checked that
	\begin{align*}
	&\frac{\partial_{x}Y_{s}^{a,b}(t_k,x)}{\sqrt{Y_s^{a,b}(t_k,x)}}-\frac{\partial_{x}Y_{t_k}^{a,b}(t_k,x)}{\sqrt{Y_{t_k}^{a,b}(t_k,x)}}=\int_{t_k}^{s}\partial_{x}Y_{u}^{a,b}(t_k,x)\Big(\frac{-\frac{a}{2}+\frac{\sigma}{4}}{(Y_{u}^{a,b}(t_k,x))^{\frac{3}{2}}}-\frac{b}{2\sqrt{Y_{u}^{a,b}(t_k,x)}}\Big)du,
	\end{align*}
	which, together with H\"older's inequality with $\frac{1}{p_0}+\frac{1}{q_0}=1$, \eqref{e2} and \eqref{e4}, implies that
	\begin{align*}
	R_{11}^{a,b}&\leq C\Delta_n^{2q-1}\int_{t_k}^{s}\Big\{\widetilde{\E}_{t_k,x}^{a,b}\Big[\Big\vert\frac{\partial_{x}Y_{u}^{a,b}(t_k,x)}{(Y_{u}^{a,b}(t_k,x))^{\frac{3}{2}}}\Big\vert^{2q}\Big]+\widetilde{\E}_{t_k,x}^{a,b}\Big[\Big\vert\frac{\partial_{x}Y_{u}^{a,b}(t_k,x)}{\sqrt{Y_{u}^{a,b}(t_k,x)}}\Big\vert^{2q}\Big]\Big\}du\\
	&\leq C\Delta_n^{2q-1}\int_{t_k}^{s}\Big\{\Big(\widetilde{\E}_{t_k,x}^{a,b}\left[\vert\partial_{x}Y_{u}^{a,b}(t_k,x)\vert^{2qp_0}\right]\Big)^{\frac{1}{p_0}}\Big(\widetilde{\E}_{t_k,x}^{a,b}\Big[\frac{1}{\vert Y_{u}^{a,b}(t_k,x)\vert^{3qq_0}}\Big]\Big)^{\frac{1}{q_0}}\\
	&\qquad+\Big(\widetilde{\E}_{t_k,x}^{a,b}\Big[\vert\partial_{x}Y_{u}^{a,b}(t_k,x)\vert^{2qp_0}\Big]\Big)^{\frac{1}{p_0}}\Big(\widetilde{\E}_{t_k,x}^{a,b}\Big[\frac{1}{\vert Y_{u}^{a,b}(t_k,x)\vert^{qq_0}}\Big]\Big)^{\frac{1}{q_0}}\Big\}du\\
	&\leq C\Delta_n^{2q-1}\int_{t_k}^{s}\Big\{\Big(1+\dfrac{1}{x^{\frac{\frac{a}{\sigma}-1+2qp_0}{2}}}\Big)^{\frac{1}{p_0}}\big(\dfrac{1}{x^{3qq_0}}\big)^{\frac{1}{q_0}}+\Big(1+\dfrac{1}{x^{\frac{\frac{a}{\sigma}-1+2qp_0}{2}}}\Big)^{\frac{1}{p_0}}\big(\dfrac{1}{x^{qq_0}}\big)^{\frac{1}{q_0}}\Big\}du\\
	&\leq C\Delta_n^{2q}\Big\{\Big(1+\dfrac{1}{x^{\frac{\frac{a}{\sigma}-1}{2p_0}+q}}\Big)\dfrac{1}{x^{3q}}+\Big(1+\dfrac{1}{x^{\frac{\frac{a}{\sigma}-1}{2p_0}+q}}\Big)\dfrac{1}{x^q}\Big\}.
	\end{align*}
	Here, $q_0$ should be chosen close to $1$ in order that $3qq_0<\frac{a}{\sigma}-1$. Thus, by choosing $q_0\in(1,\frac{1}{3q}(\frac{a}{\sigma}-1))$, under condition $\frac{a}{\sigma}>3q+1$, we have 
	\begin{equation}\label{r1}
	\begin{split}
	\widetilde{\E}_{t_k,x}^{a,b}[\vert R_1^{a,b}\vert^{2q}]&\leq C\Delta_n^{5q}\Big(1+\dfrac{1}{x^{\frac{\frac{a}{\sigma}-1}{2p_0}+q}}\Big)\Big(\dfrac{1}{x^{3q}}+\dfrac{1}{x^q}\Big)\leq C\Delta_n^{5q}\big(\dfrac{1}{x^{q}}+\dfrac{1}{x^{\frac{\frac{a}{\sigma}-1}{2p_0}+4q}}\big),
	\end{split}
	\end{equation}
	where $p_0>1$ with $q_0=\frac{p_0}{p_0-1}$ close to $1$. Then, when $q\in [1,\frac{13+\sqrt{89}}{20}]$ and $\frac{a}{\sigma}>\frac{11+\sqrt{89}}{2}$ it suffices to choose $p_0$ large enough to get $\frac{\frac{a}{\sigma}-1}{2p_0}+4q<\frac{a}{\sigma}-1$.
	
	Next, we treat the term $R_2^{a,b}$. Using H\"older's inequality with $\frac{1}{\overline{p}}+\frac{1}{\overline{q}}=1$, we have
	\begin{align*}
	\widetilde{\E}_{t_k,x}^{a,b}[\vert R_2^{a,b}\vert^{2q}]\leq\big(R_{21}^{a,b}\big)^{\frac{1}{\overline{p}}}\big(R_{22}^{a,b}\big)^{\frac{1}{\overline{q}}},
	\end{align*}
	where
	\begin{align*}
	R_{21}^{a,b}&=\widetilde{\E}_{t_k,x}^{a,b}\Big[\Big\vert \int_{t_k}^{t_{k+1}}\Big(\frac{1}{\partial_{x}Y_{s}^{a,b}(t_k,x)}-\frac{1}{\partial_{x}Y_{t_k}^{a,b}(t_k,x)}\Big)ds\Big\vert^{2q\overline{p}}\Big],\\
	R_{22}^{a,b}&=\widetilde{\E}_{t_k,x}^{a,b}\Big[\Big\vert \int_{t_k}^{t_{k+1}}\frac{\partial_{x}Y_{s}^{a,b}(t_k,x)}{\sqrt{2\sigma Y_s^{a,b}(t_k,x)}}dW_s\Big\vert^{2q\overline{q}}\Big].
	\end{align*}
	First, observe that $R_{21}^{a,b}\leq  C\Delta_n^{2q\overline{p}-1}\int_{t_k}^{t_{k+1}}R_{212}^{a,b}ds$,
	where
	\begin{align*}
	R_{212}^{a,b}&=\widetilde{\E}_{t_k,x}^{a,b}\Big[\Big\vert\frac{1}{\partial_{x}Y_{s}^{a,b}(t_k,x)}-\frac{1}{\partial_{x}Y_{t_k}^{a,b}(t_k,x)}\Big\vert^{2q\overline{p}}\Big].
	\end{align*}
	From \eqref{px}, It\^o's formula and  $\partial_xY_{t_{k}}^{a,b}(t_k,x)=1$, for any $t\in [t_k,t_{k+1}]$, we have
	\begin{equation}\label{increa}\begin{split}
	\frac{1}{\partial_xY_{s}^{a,b}(t_k,x)}-\frac{1}{\partial_xY_{t_{k}}^{a,b}(t_k,x)}&=\int_{t_k}^s\Big(\frac{b}{\partial_{x}Y_{u}^{a,b}(t_k,x)}+\frac{\sigma}{Y_{u}^{a,b}(t_k,x)\partial_{x}Y_{u}^{a,b}(t_k,x)}\Big)du\\
	&\qquad-\int_{t_k}^s\frac{\sqrt{\sigma}}{\sqrt{2Y_{u}^{a,b}(t_k,x)}\partial_{x}Y_{u}^{a,b}(t_k,x)}dW_u.
	\end{split}
	\end{equation}
	Therefore, $R_{212}^{a,b}\leq C(R_{2121}^{a,b}+R_{2122}^{a,b}+R_{2123}^{a,b})$,
	where
	\begin{align*}
	R_{2121}^{a,b}&=\widetilde{\E}_{t_k,x}^{a,b}\Big[\Big\vert\int_{t_k}^s\frac{du}{\partial_{x}Y_{u}^{a,b}(t_k,x)}\Big\vert^{2q\overline{p}}\Big],\; R_{2122}^{a,b}=\widetilde{\E}_{t_k,x}^{a,b}\Big[\Big\vert\int_{t_k}^s\frac{du}{Y_{u}^{a,b}(t_k,x)\partial_{x}Y_{u}^{a,b}(t_k,x)}\Big\vert^{2q\overline{p}}\Big],\\
	R_{2123}^{a,b}&=\widetilde{\E}_{t_k,x}^{a,b}\Big[\Big\vert\int_{t_k}^s\frac{1}{\sqrt{Y_{u}^{a,b}(t_k,x)}\partial_{x}Y_{u}^{a,b}(t_k,x)}dW_u\Big\vert^{2q\overline{p}}\Big].
	\end{align*}
	Using \eqref{e4},
	\begin{align*}
	R_{2121}^{a,b}\leq \Delta_n^{2q\overline{p}-1}\int_{t_k}^s\widetilde{\E}_{t_k,x}^{a,b}\Big[\frac{1}{\vert\partial_{x}Y_{u}^{a,b}(t_k,x)\vert^{2q\overline{p}}}\Big]du\leq  C\Delta_n^{2q\overline{p}}\Big(1+\dfrac{1}{x^{\frac{\frac{a}{\sigma}-1}{2}-q\overline{p}}}\Big).
	\end{align*}		
	Next, using H\"older's inequality with $\frac{1}{p_1}+\frac{1}{q_1}=1$, \eqref{e2} and \eqref{e4},
	\begin{align*}
	R_{2122}^{a,b}&\leq \Delta_n^{2q\overline{p}-1}\int_{t_k}^s\widetilde{\E}_{t_k,x}^{a,b}\Big[\frac{1}{\vert Y_{u}^{a,b}(t_k,x)\partial_{x}Y_{u}^{a,b}(t_k,x)\vert^{2q\overline{p}}}\Big]du\\
	&\leq \Delta_n^{2q\overline{p}-1}\int_{t_k}^s\Big(\widetilde{\E}_{t_k,x}^{a,b}\Big[\frac{1}{\vert Y_{u}^{a,b}(t_k,x)\vert^{2p_1q\overline{p}}}\Big]\Big)^{\frac{1}{p_1}}\Big(\widetilde{\E}_{t_k,x}^{a,b}\Big[\frac{1}{\vert \partial_{x}Y_{u}^{a,b}(t_k,x)\vert^{2q_1q\overline{p}}}\Big]\Big)^{\frac{1}{q_1}}du\\
	&\leq C\Delta_n^{2q\overline{p}}\frac{1}{x^{2q\overline{p}}}\Big(1+\dfrac{1}{x^{\frac{\frac{a}{\sigma}-1}{2q_1}-q\overline{p}}}\Big).
	\end{align*}		
	Finally, using BDG's and H\"older's inequalities with $\frac{1}{p_1}+\frac{1}{q_1}=1$, \eqref{e2} and \eqref{e4},
	\begin{align*}
	R_{2123}^{a,b}&\leq C\Delta_n^{q\overline{p}-1}\int_{t_k}^s\widetilde{\E}_{t_k,x}^{a,b}\Big[\frac{1}{\vert\sqrt{Y_{u}^{a,b}(t_k,x)}\partial_{x}Y_{u}^{a,b}(t_k,x)\vert^{2q\overline{p}}}\Big]du\\
	&\leq C\Delta_n^{q\overline{p}-1}\int_{t_k}^s\Big(\widetilde{\E}_{t_k,x}^{a,b}\Big[\frac{1}{\vert Y_{u}^{a,b}(t_k,x)\vert^{p_1q\overline{p}}}\Big]\Big)^{\frac{1}{p_1}}\Big(\widetilde{\E}_{t_k,x}^{a,b}\Big[\frac{1}{\vert\partial_{x}Y_{u}^{a,b}(t_k,x)\vert^{2q_1q\overline{p}}}\Big]\Big)^{\frac{1}{q_1}}du\\
	&\leq C\Delta_n^{q\overline{p}}\frac{1}{x^{q\overline{p}}}\Big(1+\dfrac{1}{x^{\frac{\frac{a}{\sigma}-1}{2q_1}-q\overline{p}}}\Big).
	\end{align*}		
	Thus, we have shown that
	\begin{align*}
	R_{212}^{a,b}&\leq C\Delta_n^{2q\overline{p}}\Big(1+\dfrac{1}{x^{\frac{\frac{a}{\sigma}-1}{2}-q\overline{p}}}\Big)+C\Delta_n^{2q\overline{p}}\frac{1}{x^{2q\overline{p}}}\Big(1+\dfrac{1}{x^{\frac{\frac{a}{\sigma}-1}{2q_1}-q\overline{p}}}\Big)+C\Delta_n^{q\overline{p}}\frac{1}{x^{q\overline{p}}}\Big(1+\dfrac{1}{x^{\frac{\frac{a}{\sigma}-1}{2q_1}-q\overline{p}}}\Big),
	\end{align*}
	which implies that
	\begin{align*}
	R_{21}^{a,b}&\leq C\Delta_n^{2q\overline{p}}\Big\{\Delta_n^{2q\overline{p}}\Big(1+\dfrac{1}{x^{\frac{\frac{a}{\sigma}-1}{2}-q\overline{p}}}\Big)+\Delta_n^{2q\overline{p}}\frac{1}{x^{2q\overline{p}}}\Big(1+\dfrac{1}{x^{\frac{\frac{a}{\sigma}-1}{2q_1}-q\overline{p}}}\Big)+\Delta_n^{q\overline{p}}\frac{1}{x^{q\overline{p}}}\Big(1+\dfrac{1}{x^{\frac{\frac{a}{\sigma}-1}{2q_1}-q\overline{p}}}\Big)\Big\}.
	\end{align*}
	Next, using BDG's and H\"older's inequalities with $\frac{1}{p_2}+\frac{1}{q_2}=1$,  \eqref{e2} and \eqref{e4}, 
	\begin{align*}
	R_{22}^{a,b}&\leq C\widetilde{\E}_{t_k,x}^{a,b}\Big[\Big\vert \int_{t_k}^{t_{k+1}}\frac{(\partial_{x}Y_{s}^{a,b}(t_k,x))^2}{Y_s^{a,b}(t_k,x)}ds\Big\vert^{q\overline{q}}\Big]\leq C\Delta_n^{q\overline{q}-1}\int_{t_k}^{t_{k+1}}\widetilde{\E}_{t_k,x}^{a,b}\Big[\Big\vert \frac{\partial_{x}Y_{s}^{a,b}(t_k,x)}{\sqrt{Y_s^{a,b}(t_k,x)}}\Big\vert^{2q\overline{q}}\Big]ds\\
	&\leq C\Delta_n^{q\overline{q}-1}\int_{t_k}^{t_{k+1}}\Big(\widetilde{\E}_{t_k,x}^{a,b}\Big[\frac{1}{\vert Y_s^{a,b}(t_k,x)\vert^{p_2q\overline{q}}}\Big]\Big)^{\frac{1}{p_2}}\Big(\widetilde{\E}_{t_k,x}^{a,b}\Big[\big\vert \partial_{x}Y_{s}^{a,b}(t_k,x)\big\vert^{2q_2q\overline{q}}\Big]\Big)^{\frac{1}{q_2}}ds\\
	&\leq C\Delta_n^{q\overline{q}}\dfrac{1}{x^{q\overline{q}}}\Big(1+\dfrac{1}{x^{\frac{\frac{a}{\sigma}-1}{2q_2}+q\overline{q}}}\Big).
	\end{align*}
	Here, $p_2$ should be chosen close to $1$ in order that $p_2q\overline{q}<\dfrac{a}{\sigma}-1$. In order to apply \eqref{e2} and \eqref{e4} to estimate two terms above $R_{21}^{a,b}$ and $R_{22}^{a,b}$, all conditions required here are as follows
	\begin{align*}
	-2q_1q\overline{p}\geq -\frac{(\frac{a}{\sigma}-1)^2}{2(\frac{a}{\sigma}-\frac{1}{2})},\;
	2p_1q\overline{p}<\dfrac{a}{\sigma}-1,\;
	q\overline{q}<\dfrac{a}{\sigma}-1.
	\end{align*}
	This implies that
	\begin{align*}
	\frac{a}{\sigma}\geq 2q_1q\overline{p}+\sqrt{2q_1q\overline{p}\left(2q_1q\overline{p}+1\right)}+1,\;
	\frac{a}{\sigma}>\frac{2q_1q\overline{p}}{q_1-1}+1,\;
	\frac{a}{\sigma}>\frac{q\overline{p}}{\overline{p}-1}+1.
	\end{align*}	
	Here, the optimal choice for $\overline{p}$ and $q_1$ corresponds to choose them in a way which gives minimal restrictions on the ratio $\frac{a}{\sigma}$. That is,
	$$
	2q_1q\overline{p}+\sqrt{2q_1q\overline{p}\left(2q_1q\overline{p}+1\right)}=\frac{2q_1q\overline{p}}{q_1-1}=\frac{q\overline{p}}{\overline{p}-1}.
	$$	
	Thus, the unique solution is given by $\overline{p}=\frac{21q+4+3\sqrt{q(49q+8)}}{2(11q+2+\sqrt{q(49q+8)})}$ and $q_1=\frac{11q+2+\sqrt{q(49q+8)}}{2(6q+1)}$, which implies that $\frac{a}{\sigma}>\frac{7}{2}q+\frac{1}{2}\sqrt{q(49q+8)}+1$.	
	Therefore, under $\frac{a}{\sigma}>\frac{7}{2}q+\frac{1}{2}\sqrt{q(49q+8)}+1$, we have shown that
	\begin{align}\label{r2}
	\widetilde{\E}_{t_k,x}^{a,b}[\vert R_2^{a,b}\vert^{2q}]&\leq  C\dfrac{\Delta_n^{4q}}{x^{q}}\Big(1+\dfrac{1}{x^{\frac{\frac{a}{\sigma}-1}{2q_2\overline{q}}+q}}\Big)\Big\{\Big(1+\dfrac{1}{x^{\frac{\frac{a}{\sigma}-1}{2\overline{p}}-q}}\Big)+\Big(1+\dfrac{1}{x^{\frac{\frac{a}{\sigma}-1}{2q_1\overline{p}}-q}}\Big)\Big(\frac{1}{x^{2q}}+\frac{1}{x^q}\Big)\Big\}\notag\\
	&\leq C\Delta_n^{4q}\big(\dfrac{1}{x^{q}}+\dfrac{1}{x^{(\frac{a}{\sigma}-1)(\frac{1}{2\overline{p}}+\frac{1}{2q_2\overline{q}})+3q}}\big),
	\end{align}
	where $\overline{p}=\frac{21q+4+3\sqrt{q(49q+8)}}{2(11q+2+\sqrt{q(49q+8)})}$, $\overline{q}=\frac{\overline{p}}{\overline{p}-1}=\frac{21q+4+3\sqrt{q(49q+8)}}{-q+\sqrt{q(49q+8)}}$, $q_1=\frac{11q+2+\sqrt{q(49q+8)}}{2(6q+1)}$, $q_1\overline{p}=\frac{21q+4+3\sqrt{q(49q+8)}}{4(6q+1)}$, $q_2>1$ with $\frac{q_2}{q_2-1}$ close to $1$. When $q\in [1,\frac{13+\sqrt{89}}{20}]$ and $\frac{a}{\sigma}>\frac{11+\sqrt{89}}{2}$, it suffices to choose $q_2$ large enough to get $(\frac{a}{\sigma}-1)(\frac{1}{2\overline{p}}+\frac{1}{2q_2\overline{q}})+3q<\frac{a}{\sigma}-1$ since $\overline{p}>1.04$. 	
	
	Finally, we treat the term $R_3^{a,b}$. Using H\"older's inequality with $\frac{1}{p_3}+\frac{1}{q_3}=1$,
	\begin{align*}
	&\widetilde{\E}_{t_k,x}^{a,b}[\vert R_3^{a,b}\vert^{2q}]\leq \Delta_n^{2q-1}\int_{t_k}^{t_{k+1}}\widetilde{\E}_{t_k,x}^{a,b}\Big[\Big\vert \int_{s}^{t_{k+1}}D_s\Big(\dfrac{1}{\partial_xY_{r}^{a,b}(t_k,x)}\Big)dr\frac{\partial_{x}Y_{s}^{a,b}(t_k,x)}{\sqrt{2\sigma Y_s^{a,b}(t_k,x)}}\Big\vert^{2q}\Big]ds\\
	&\leq \Delta_n^{2q-1}\int_{t_k}^{t_{k+1}}\Big(\widetilde{\E}_{t_k,x}^{a,b}\Big[\Big\vert \int_{s}^{t_{k+1}}D_s\Big(\dfrac{1}{\partial_xY_{r}^{a,b}(t_k,x)}\Big)dr\Big\vert^{2qp_3}\Big]\Big)^{\frac{1}{p_3}}\\
	&\qquad\times\Big(\widetilde{\E}_{t_k,x}^{a,b}\Big[\Big\vert \frac{\partial_{x}Y_{s}^{a,b}(t_k,x)}{\sqrt{2\sigma Y_s^{a,b}(t_k,x)}}\Big\vert^{2qq_3}\Big]\Big)^{\frac{1}{q_3}}ds\\
	&\leq \Delta_n^{2q-1}\int_{t_k}^{t_{k+1}}\Big(\Delta_n^{2qp_3-1} \int_{s}^{t_{k+1}}\widetilde{\E}_{t_k,x}^{a,b}\Big[\Big\vert D_s\Big(\dfrac{1}{\partial_xY_{r}^{a,b}(t_k,x)}\Big)\Big\vert^{2qp_3}\Big]dr\Big)^{\frac{1}{p_3}}\\
	&\qquad\times\Big(\widetilde{\E}_{t_k,x}^{a,b}\Big[\Big\vert \frac{\partial_{x}Y_{s}^{a,b}(t_k,x)}{\sqrt{2\sigma Y_s^{a,b}(t_k,x)}}\Big\vert^{2qq_3}\Big]\Big)^{\frac{1}{q_3}}ds.
	\end{align*}
	By using \eqref{Mallinveflow1} and the same computations as in the proof of \eqref{dmif1} with $p=2qp_3$, we get
	\begin{align*}
	&\widetilde{\E}_{t_k,x}^{a,b}\Big[\Big\vert D_s\Big(\dfrac{1}{\partial_xY_{r}^{a,b}(t_k,x)}\Big)\Big\vert^{2qp_3}\Big]=\widetilde{\E}_{t_k,x}^{a,b}\Big[\Big\vert \frac{D_s(\partial_xY_r^{a,b}(t_k,x))}{(\partial_xY_r^{a,b}(t_k,x))^2}\Big\vert^{2qp_3}\Big]\leq C\Big(1+\dfrac{1}{x^{\frac{\frac{a}{\sigma}-1}{2p_4}-qp_3}}\Big)\dfrac{1}{x^{qp_3}}\\
	&\qquad+C\Delta_n^{2qp_3}\Big(1+\dfrac{1}{x^{\frac{\frac{a}{\sigma}-1}{2p_4}-qp_3}}\Big)\dfrac{1}{x^{3qp_3}}+C\Delta_n^{qp_3}\Big(1+\dfrac{1}{x^{\frac{\frac{a}{\sigma}-1}{2p_4}-qp_3}}\Big)\dfrac{1}{x^{2qp_3}},
	\end{align*}
	where $p_4$ and $q_4$ satisfying $\frac{1}{p_4}+\frac{1}{q_4}=1$ are given in the proof of \eqref{dmif1}. Next, using H\"older's inequality with $\frac{1}{p_5}+\frac{1}{q_5}=1$, \eqref{e2} and \eqref{e4},
	\begin{align*}
	\widetilde{\E}_{t_k,x}^{a,b}\Big[\Big\vert \frac{\partial_{x}Y_{s}^{a,b}(t_k,x)}{\sqrt{2\sigma Y_s^{a,b}(t_k,x)}}\Big\vert^{2qq_3}\Big]&\leq C\Big(\widetilde{\E}_{t_k,x}^{a,b}\Big[\vert \partial_{x}Y_{s}^{a,b}(t_k,x)\vert^{2qq_3p_5}\Big]\Big)^{\frac{1}{p_5}}\Big(\widetilde{\E}_{t_k,x}^{a,b}\Big[\frac{1}{\vert Y_s^{a,b}(t_k,x)\vert^{qq_3q_5}}\Big]\Big)^{\frac{1}{q_5}}\\
	&\leq C\Big(1+\dfrac{1}{x^{\frac{\frac{a}{\sigma}-1}{2p_5}+qq_3}}\Big)\frac{1}{x^{qq_3}}.
	\end{align*}
	Here, $q_5$ should be chosen close to $1$ in order that $qq_3q_5<\dfrac{a}{\sigma}-1$. In order to apply \eqref{e2} and \eqref{e4} to estimate the term $R_{3}^{a,b}$, all conditions required here are as follows
	\begin{align*}
	-2qp_3p_4\geq -\frac{(\frac{a}{\sigma}-1)^2}{2(\frac{a}{\sigma}-\frac{1}{2})},\;
	3qp_3q_4<\dfrac{a}{\sigma}-1,\;
	qq_3<\dfrac{a}{\sigma}-1.
	\end{align*}
	This implies that
	\begin{align*}
	\frac{a}{\sigma}\geq 2qp_3p_4+\sqrt{2qp_3p_4\left(2qp_3p_4+1\right)}+1,\;
	\frac{a}{\sigma}>\frac{3qp_3p_4}{p_4-1}+1,\;
	\frac{a}{\sigma}>\frac{qp_3}{p_3-1}+1.
	\end{align*}	
	Here, the optimal choice for $p_3$ and $p_4$ corresponds to choose them in a way which gives minimal restrictions on the ratio $\frac{a}{\sigma}$. That is,
	$$
	2qp_3p_4+\sqrt{2qp_3p_4\left(2qp_3p_4+1\right)}=\frac{3qp_3p_4}{p_4-1}=\frac{qp_3}{p_3-1}.
	$$	
	Thus, the unique solution is given by $p_3=\frac{2\sqrt{1+8q}+4\sqrt{2q}}{2\sqrt{1+8q}+3\sqrt{2q}}$ and $p_4=1+\frac{3}{2}\sqrt{\frac{2q}{1+8q}}$, which implies that
	$\frac{a}{\sigma}>4q+1+\sqrt{2q(8q+1)}$. Therefore, under $\frac{a}{\sigma}>4q+1+\sqrt{2q(8q+1)}$, we obtain
	\begin{align}\label{r3}
	&\widetilde{\E}_{t_k,x}^{a,b}[\vert R_3^{a,b}\vert^{2q}]\leq  C\Delta_n^{2q}\Big\{\Delta_n^{2q}\Big(1+\dfrac{1}{x^{\frac{\frac{a}{\sigma}-1}{2p_3p_4}-q}}\Big)\dfrac{1}{x^q}+\Delta_n^{4q}\Big(1+\dfrac{1}{x^{\frac{\frac{a}{\sigma}-1}{2p_3p_4}-q}}\Big)\dfrac{1}{x^{3q}}\notag\\
	&\qquad+\Delta_n^{3q}\Big(1+\dfrac{1}{x^{\frac{\frac{a}{\sigma}-1}{2p_3p_4}-q}}\Big)\dfrac{1}{x^{2q}}\Big\}\Big(1+\dfrac{1}{x^{\frac{\frac{a}{\sigma}-1}{2p_5q_3}+q}}\Big)\frac{1}{x^q}\leq C\Delta_n^{4q}\Big(1+\dfrac{1}{x^{\frac{\frac{a}{\sigma}-1}{2p_5q_3}+q}}\Big)\Big(1+\dfrac{1}{x^{\frac{\frac{a}{\sigma}-1}{2p_3p_4}-q}}\Big)\notag\\
	&\qquad\times\Big(\dfrac{1}{x^{2q}}+\dfrac{1}{x^{4q}}+\dfrac{1}{x^{3q}}\Big)\leq C\Delta_n^{4q}\big(\dfrac{1}{x^{2q}}+\dfrac{1}{x^{(\frac{a}{\sigma}-1)(\frac{1}{2p_3p_4}+\frac{1}{2p_5q_3})+4q}}\big),
	\end{align}
	where $p_3=\frac{2\sqrt{1+8q}+4\sqrt{2q}}{2\sqrt{1+8q}+3\sqrt{2q}}$, $q_3=\frac{p_3}{p_3-1}=4+\sqrt{2(8+\frac{1}{q})}$, $p_4=1+\frac{3}{2}\sqrt{\frac{2q}{1+8q}}$, $p_3p_4=1+2\sqrt{\frac{2q}{1+8q}}$, $p_5>1$ with $\frac{p_5}{p_5-1}$ close to $1$. Then, when $q\in [1,\frac{13+\sqrt{89}}{20}]$ and $\frac{a}{\sigma}>\frac{11+\sqrt{89}}{2}$, it suffices to choose $p_5$ is large enough to get $(\frac{a}{\sigma}-1)(\frac{1}{2p_3p_4}+\frac{1}{2p_5q_3})+4q<\frac{a}{\sigma}-1$ since $p_3p_4>1.89$. 	
	
	From \eqref{r}, \eqref{r1}, \eqref{r2} and \eqref{r3}, under condition $\frac{a}{\sigma}>4q+1+\sqrt{2q(8q+1)}$, we obtain 
	\begin{align*}
	&\widetilde{\E}_{t_k,x}^{a,b}[\vert R_1^{a,b}+R_2^{a,b}+R_3^{a,b}\vert^{2q}]\leq C\Delta_n^{4q}\big(\dfrac{1}{x^{q}}+\dfrac{1}{x^{\frac{\frac{a}{\sigma}-1}{2p_0}+4q}}+\dfrac{1}{x^{(\frac{a}{\sigma}-1)(\frac{1}{2\overline{p}}+\frac{1}{2q_2\overline{q}})+3q}}\\
	&\qquad+\dfrac{1}{x^{(\frac{a}{\sigma}-1)(\frac{1}{2p_3p_4}+\frac{1}{2p_5q_3})+4q}}\big).
	\end{align*}
	Moreover, when $q\in [1,\frac{13+\sqrt{89}}{20}]$ and condition {\bf(A)} holds, we obtain
    \begin{align*}
	\widetilde{\E}_{t_k,x}^{a,b}\big[\big\vert R_1^{a,b}+R_2^{a,b}+R_3^{a,b}\big\vert^{2q}\big]\leq C\Delta_n^{4q}\big(\frac{1}{x^q}+\frac{1}{x^{\frac{a}{\sigma}-1}}\big).
    \end{align*}	
    Thus, we conclude the desired estimate \eqref{es2}.
	
	\vskip 5pt
	{\it Proof of \eqref{es4}.} We proceed in the same way as in the proof of \eqref{es2} where we use 
	\begin{equation*}\begin{split}
	\frac{Y_{s}^{a,b}(t_k,x)}{\partial_{x}Y_{s}^{a,b}(t_k,x)}-\frac{Y_{t_k}^{a,b}(t_k,x)}{\partial_{x}Y_{t_k}^{a,b}(t_k,x)}&=a\int_{t_k}^s\frac{du}{\partial_{x}Y_{u}^{a,b}(t_k,x)}+\sqrt{\frac{\sigma }{2}}\int_{t_k}^s\frac{\sqrt{Y_{u}^{a,b}(t_k,x)}}{\partial_{x}Y_{u}^{a,b}(t_k,x)}dW_u
	\end{split}
	\end{equation*}	
	instead of \eqref{increa}.
\end{proof}
\begin{remark}\label{minimal}	
	When we use Cauchy-Schwarz's inequality instead of H\"older's inequality to estimate $\vert R_1^{a,b}+R_2^{a,b}+R_3^{a,b}\vert^{2q}$, the required condition will be $\frac{a}{\sigma}>8q+1+\sqrt{8q(8q+1)}$ which is actually bigger than $4q+1+\sqrt{2q(8q+1)}$.
\end{remark}

\section{Appendix B: Useful results}
We recall a classical convergence result and a central limit theorem on triangular arrays of random variables. For each $n\in\mathbb{N}$, let $(\zeta_{k,n})_{k\geq 1}$ be a sequence of random variables defined on $(\Omega, \mathcal{F}, \{\mathcal{F}_t\}_{t\geq 0}, \P)$, and we assume that $\zeta_{k,n}$ are $\mathcal{F}_{t_{k+1}}$-measurable for all $k$.
\begin{lemma}\label{zero} \textnormal{\cite[Lemma 3.4]{J12}} Assume that as $n  \rightarrow \infty$,  
	\begin{equation*} 
	\textnormal{(i)}\;  \sum_{k=0}^{n-1}\E\left[\zeta_{k,n}\vert \mathcal{F}_{t_k}\right] \overset{\P}{\longrightarrow} 0, \quad \text{ and } \quad \textnormal{(ii)} \,  \sum_{k=0}^{n-1}\E\left[\zeta_{k,n}^2\vert \mathcal{F}_{t_k} \right]\overset{\P}{\longrightarrow} 0.
	\end{equation*}
	Then as $n  \rightarrow \infty$, 
	$
	\sum_{k=0}^{n-1}\zeta_{k,n}\overset{\P}{\longrightarrow} 0.
	$
\end{lemma}

\begin{theorem}\label{clt} \textnormal{\cite[Lemma 3.6]{J12}} 
	Assume that there exist real numbers $M$ and $V>0$ such that as $n  \rightarrow \infty$, 
	\begin{equation*}\begin{split}
	&\textnormal{(i)}\; \sum_{k=0}^{n-1}\E\left[\zeta_{k,n}\vert \mathcal{F}_{t_k}\right] \overset{\P}{\longrightarrow} M, \qquad \textnormal{(ii)}\; \sum_{k=0}^{n-1}\left(\E\left[\zeta_{k,n}^2\vert \mathcal{F}_{t_k} \right]-\left(\E\left[\zeta_{k,n}\vert \mathcal{F}_{t_k}\right]\right)^2\right)\overset{\P}{\longrightarrow} V, \text{ and }\\
	&\textnormal{(iii)}\; \sum_{k=0}^{n-1}\E\left[\zeta_{k,n}^4\vert \mathcal{F}_{t_k}\right] \overset{\P}{\longrightarrow} 0.
	\end{split}
	\end{equation*}
	Then as $n  \rightarrow \infty$,  $\sum_{k=0}^{n-1}\zeta_{k,n}\overset{\mathcal{L}(\P)}{\longrightarrow} \mathcal{N}+M$, where $\mathcal{N}$ is a centered  Gaussian random variable with variance $V$.
\end{theorem}

\begin{theorem}\label{ruleproNualart}\textnormal{\cite[Proposition 1.3.3]{N}} Let $F\in \mathbb{D}^{1,2}$ and $u$ be in the domain of $\delta$ such that $Fu\in L^2(\Omega;H)$. Then $Fu$ belongs to the domain of $\delta$ and the following equality is true 
\begin{equation}\label{ruleproduct}
\delta(Fu)=F\delta(u)-\left<DF,u\right>_H,
\end{equation} 
provided the right-hand side of \eqref{ruleproduct}	is square integrable.
\end{theorem}

\begin{theorem}\label{existenceMallideri}\textnormal{\cite[Lemma 2.1]{LNN03}} Let $\varphi\in \mathcal{C}^1(\mathbb{R})$ be a continuously differentiable function and let $F\in \mathbb{D}^{1,2}$. Then $\varphi(F)\in \mathbb{D}^{1,2}$ if and only if $\varphi(F)\in L^2(\Omega)$ and $\varphi'(F)DF\in L^2(\Omega\times [0,T])$, and under these hypotheses 
	$$
	D(\varphi(F))=\varphi'(F)DF.
	$$
\end{theorem}

\noindent{\bf Acknowledgements.} We would like to thank anonymous referees and the editor for their valuable comments that helped us to improve the paper.

\end{document}